\documentclass[reqno,11pt]{amsart}
\usepackage{amsthm,amsfonts,amssymb,euscript,mathrsfs,graphics,color,amsmath,amssymb,latexsym,marginnote, bigints}
\usepackage{soul}
\usepackage{graphicx}
\usepackage{multicol,bbm}
\usepackage{stmaryrd,empheq}
\usepackage{amsthm,amsfonts,amssymb,euscript,mathrsfs,graphics,xcolor,amsmath,amssymb,latexsym,marginnote,esint}
\usepackage{subcaption}
\usepackage[hidelinks]{hyperref}
\usepackage{comment}
\usepackage[margin=1.0in]{geometry}
\numberwithin{equation}{section}

\def\al{\alpha}

\def\R{\mathbb{R}}

\def\bar{\overline}

\definecolor{bluegreen}{rgb}{0.0, 0.3, 0.9}

\usepackage[normalem]{ulem}

\newtheorem{theorem}{Theorem}[section]
\newtheorem{lemma}[theorem]{Lemma}
\newtheorem{proposition}[theorem]{Proposition}
\newtheorem{corollary}[theorem]{Corollary}
\newtheorem{definition}[theorem]{Definition}

\title[]{Mixing Zone Dynamics for the Saffman--Taylor Instability with Viscosity Jump}
\author{Paula Luna-Velasco and Francisco Mengual}

\begin{document}

\begin{abstract}
We construct mixing solutions to the incompressible porous media equation for fully unstable Muskat initial data. The mixing zone expands linearly in time around an evolving central interface. Our approach combines the $h$-principle with the inversion of the associated Birkhoff–Rott operator, required when the viscosities differ. A key ingredient is to formulate this inversion in terms of the Saffman–Taylor function rather than the vorticity. This considerably simplifies the analysis and may be of independent interest for the Muskat problem.
\end{abstract}

\maketitle
 
\setcounter{tocdepth}{1}
\tableofcontents


\section{Introduction and main results}


We study the Saffman--Taylor instability \cite{SaffmanTaylor58}, which refers to the formation of finger-like patterns at the interface between two incompressible fluids in a porous medium \cite{EscherMatioc11}. The interface dynamics is modeled by the Muskat problem \cite{Muskat37}, originally introduced to describe oil-water motion through sand in the petroleum industry, and is closely related to Hele--Shaw models \cite{HeleShaw1898}.

In this work, we follow the program of modelling hydrodynamic instabilities through convex integration, addressing density and viscosity jumps in \eqref{IPM} for fully unstable Muskat initial interfaces. This allows us to circumvent the classical ill-posedness and identify the macroscopic evolution of the mixing zone.

\subsection{Incompressible porous media equation \eqref{IPM}}
We consider a multiphase incompressible fluid in a porous medium under the action of gravity.
The density $\rho=\rho(t,x)$ and viscosity $\mu=\mu(t,x)$ are transported by the flow:
\begin{equation}\label{transport}
\partial_t a+u\cdot\nabla a=0,\qquad a=\rho,\mu.
\end{equation}
The velocity field $u=u(t,x)$ satisfies the incompressibility condition:
\begin{equation}\label{incompressibility}
\nabla\cdot u =0,
\end{equation}
together with Darcy's law:
\begin{equation}\label{Darcy}
\frac{\mu}{\kappa}u
=-\nabla p-g\rho e_2,
\end{equation}
where $p=p(t,x)$ denotes the pressure, $e_2=(0,1)$ is the vertical unit vector, and $\kappa,g>0$ are the permeability of the porous medium and the gravitational constant, respectively.

For a two-phase flow with densities $\rho_+>\rho_->0$ and viscosities $\mu_+,\mu_->0$, it is convenient to introduce a phase variable $\theta(t,x)$, taking values $+1$ and $-1$, according to whether, at time $t\geq 0$, the pores near $x\in\R^2$ are occupied by the $+$ phase or the $-$ phase:
\begin{equation}\label{atheta}
a(t,x)=\frac{a_++a_-}{2}+\frac{a_+-a_-}{2}\theta(t,x),\qquad
a=\rho,\mu.
\end{equation}
Substituting \eqref{atheta} into the transport equation \eqref{transport}, we find that the latter is equivalent to the phase  being transported by the flow. 
Similarly, Darcy's law \eqref{Darcy} can be rewritten in terms of the phase and the Atwood numbers
$$
A_\rho:=\kappa g\frac{\rho_+-\rho_-}{\mu_++\mu_-}>0,\qquad
A_\mu:=\frac{\mu_+-\mu_-}{\mu_++\mu_-}\in (-1,1).
$$
By rescaling time and redefining the velocity and pressure as necessary, this can be normalized so that $A_\rho=1$. Henceforth, we write $A:=A_\mu$.
Notice that $A=0$ corresponds to $\mu_+=\mu_-$.

The resulting system is the two-phase \emph{incompressible porous media equation}:
\begingroup
\renewcommand{\theequation}{IPM}
\begin{subequations}\label{IPM}
\begin{align}
\partial_t\theta+u\cdot\nabla\theta
&=0,\tag{IPMa}\label{IPMa}\\
\nabla\cdot u
&=0,\tag{IPMb}\label{IPMb}\\
u+A\theta u
&=-\nabla p-\theta e_2,\tag{IPMc}\label{IPMc}
\intertext{coupled with the pointwise constraint}
|\theta(t,x)|&=1.\tag{IPMd}\label{IPMd}
\end{align}
\end{subequations}
\endgroup

In this system, the velocity and pressure are determined in terms of the phase by the incompressibility condition and Darcy's law \eqref{IPMb}-\eqref{IPMc}. Thus, $\theta$ is the active scalar governing the evolution and is prescribed by an initial condition
\begin{equation}
\theta|_{t=0}=\theta^\circ.
\end{equation}

\subsection{Muskat equation}
It follows directly that stratified profiles are stationary solutions
$$
\theta=\theta^\circ(x_2),
\qquad
u=0,
\qquad
\partial_2 p=-\theta^\circ.
$$
The simplest two-phase stratified flows correspond to the Riemann initial data
\begin{equation}\label{Riemann}
\theta_{\mathrm{s}}^\circ(x_2)=
\begin{cases}
-1, & x_2>0,\\
+1, & x_2<0,
\end{cases}
\qquad
\theta_{\mathrm{u}}^\circ(x_2)=
\begin{cases}
+1, & x_2>0,\\
-1, & x_2<0.
\end{cases}
\end{equation}
In both configurations, the fluids are separated by the flat interface $z^\circ(\alpha)=(\alpha,0)$, with the denser fluid lying below in $\theta_{\mathrm{s}}^\circ$ and above in $\theta_{\mathrm{u}}^\circ$. As we discuss next, these stationary solutions are respectively stable and unstable under perturbations of the interface.

Assuming that the two fluids remain separated by a sufficiently
regular interface $z=z(t,\alpha)$, the \eqref{IPM} system reduces to an interface evolution problem known as the \emph{Muskat equation}:
\begin{align}
(\partial_t z-U)\cdot\partial_\alpha z^\perp&=0,
&\varpi&=-\llbracket\theta\rrbracket(AU+e_2)\cdot\partial_\alpha z,\tag{M$\varpi$}\label{Mvarpi:intro}
\intertext{where $\varpi$ is the vorticity density, $U=U[\varpi,z]$ is the
mean velocity trace on the interface, and
$\llbracket\theta\rrbracket
=\theta_{\mathrm{right}}-\theta_{\mathrm{left}}$
is the phase jump across the oriented curve $z$.
In this paper, we obtain an equivalent formulation in terms
of the Saffman--Taylor function $\sigma$:}
(\partial_tz +Q)\cdot\partial_\alpha z^\perp&=0,
&\sigma&=-\llbracket\theta\rrbracket(AQ-e_2)\cdot\partial_\alpha z^\perp,\tag{M$\sigma$}\label{Msigma:intro}
\end{align}
where $Q=Q[\sigma,z]$ is the mean pressure-gradient trace
on the interface.
Both formulations, their equivalence, and the Birkhoff--Rott
representations of $U$ and $Q$ are developed in Section~\ref{section3}.

Although $\sigma$ is commonly called the Rayleigh--Taylor
function, we use the term Saffman--Taylor to distinguish
the porous-media setting from its Euler counterpart.

The stability of the interface depends on the sign of the Saffman--Taylor function:
\[
\begin{array}{c@{\qquad\qquad}c}
\text{stable} & \text{unstable}\\
\sigma>0     & \sigma\le0
\end{array}
\]
For the flat interface,
$\sigma=\llbracket\theta\rrbracket$, recovering the stability
classification of the Riemann data \eqref{Riemann}.
For $A=0$, $\sigma=\llbracket\theta\rrbracket\partial_\alpha z_1$,
so stability is again determined by the position of the denser fluid.
When $A\ne0$, stability also depends on the flow, with both
$\varpi$ and $\sigma$ determined by implicit, nonlocal relations,
making the analysis more delicate.

\subsection{Background}
We briefly review the literature on both regimes to place our result in context.

In the stable regime, local well-posedness for 2D
flows with density and viscosity jumps was established for
uniformly positive $\sigma$, including periodic and non-graph
interfaces \cite{CCG11}.
Subsequent works lowered the regularity requirements
(see e.g.~\cite{Matioc19,CordobaLazar21}), reaching critical regularity
for equal viscosities
\cite{AlazardNguyen21a,AlazardNguyen21b,AlazardNguyen23}.
For global results with small data, see
\cite{CCGS13,CCGRPS16,CGS17,Cameron19,GGJPS19,CordobaLazar21,Nguyen22},
and with large monotone data, see \cite{DLL17}.
Related results cover domains with boundaries
\cite{CGO14,Granero14,NguyenPausader20,Zlatos26}, non-asymptotically flat interfaces \cite{GGNP22,Sanchez25,Zlatos26}, and the desingularization of small moving corners \cite{GGHP24}.
Finite-time singularity scenarios include initially stable
interfaces that overturn \cite{CCFGL12} and subsequently lose
smoothness \cite{CCFG13}, and singularity formation in the half-plane
\cite{Zlatospp}.

In the unstable regime, the classical Muskat problem is ill-posed in Sobolev spaces \cite{SCH04,GGPS25},
reflecting the Saffman--Taylor instability.
With surface tension, see \cite{Nguyen20} for local
well-posedness and \cite{GGPS25} for global results on Muskat bubbles.

A first macroscopic description was obtained through a Lagrangian
relaxation, which predicts a linearly growing mixing zone for the
flat interface \cite{Otto99}.
Within the differential-inclusion framework, the relaxation of \eqref{IPM}
was explicitly computed for $A=0$, and convex integration was used
to construct infinitely many mixing solutions \cite{Szekelyhidi12} (see also \cite{CFG11,Shvydkoy11,Knott13,HitruhinLindberg21}).
Remarkably, the two relaxation approaches agree when selecting
the subsolution with the largest mixing zone.

Beyond the flat configuration, mixing solutions for $A=0$ were
constructed for fully unstable $H^5$ interfaces using continuous
subsolutions \cite{CCF21} (see also \cite{ACF21}), and for
interfaces in weighted $C^{3,\alpha}$ spaces using piecewise-constant
subsolutions \cite{ForsterSzekelyhidi18}
(see also \cite{NoisetteSzekelyhidi21}).
The approach has also been applied to localized mixing for bubbles and turned interfaces \cite{CFM22}. Recently, the Lagrangian relaxation has been extended to non-flat analytic interfaces \cite{CFG25}.

For different viscosities, $|A|<1$, the $h$-principle and mixing
solutions for the flat interface were obtained in \cite{Mengual22},
preserving the connection with the Lagrangian relaxation. In this work, we take the opportunity to complement this result
by establishing the exact relaxation in Appendix~\ref{AppendixA}.

This strategy of modelling turbulence zones through convex
integration has also been applied to the Rayleigh--Taylor
\cite{GKS21,GebhardKolumban22,GHK24,GebhardKolumban25}, Kelvin--Helmholtz \cite{Szekelyhidi11,MengualSzekelyhidi23,GebhardKolumban22b}, and Widnall \cite{GHM25} instabilities.

\subsection{Main result}

In this work, we construct mixing solutions for the two-phase \eqref{IPM} equation with different densities and viscosities in the fully unstable regime.
Unlike the equal-viscosity case, the viscosity jump requires inverting an implicit relation, usually formulated in terms of vorticity in the classical Muskat problem \eqref{Mvarpi:intro}. In the mixing setting, this approach also involves interior vorticity coupled to the correction of the macroscopic momentum.

A key ingredient of our construction is to work instead with the pressure and the Saffman–Taylor boundary densities. The pressure remains harmonic in each region, and the resulting boundary conditions allow us to determine these densities without introducing the interior vorticity as an independent unknown. This formulation also yields an equivalent description \eqref{Msigma:intro} of the classical Muskat problem that may be of independent interest. We explain these ideas in Section~ \ref{section3}.

\begin{theorem}\label{thm:main}
Let $A\in(-1,1)$, and let $z^\circ$ be a chord--arc curve
satisfying $z^\circ-(\alpha,0)\in H^5(\mathbb{R};\mathbb{R}^2)$.
Assume that its Saffman--Taylor function $\sigma^\circ$
is uniformly negative.
Then there exists $T>0$ such that infinitely many mixing
solutions of the \eqref{IPM} system on $[0,T]\times\R^2$ emanate from
the corresponding unstable Muskat initial data.
\end{theorem}

The mixing solutions constructed here retain the two pure phases
in the regions $\Omega_+(t)$ and $\Omega_-(t)$, separated by an
evolving mixing zone $\Omega_{\mathrm{mix}}(t)$, where both fluids
are present at arbitrarily small scales (see Definition~\ref{def:mixing-solution}). This strip lies between
two curves $z_\pm=z_\pm(t,\alpha)$
\[
  z_+=z+\frac{tc}{1-A}\tau^\perp,
  \qquad
  z_-=z-\frac{tc}{1+A}\tau^\perp,
\]
where $\tau=\partial_\alpha z^\circ/|\partial_\alpha z^\circ|$
is the initial unit tangent and $z=z(t,\alpha)$ is an evolving central curve
whose initial normal velocity agrees with the Muskat evolution (see Figure \ref{fig:muskat-fingering:z}).
For a fixed constant $0<\bar{c}<1$, we choose the opening rate
$c=c(\alpha)$ as
\[
  c=-\bar{c}\,
  \frac{\sigma^\circ}{|\partial_\alpha z^\circ|}
\]
where $\sigma^\circ=\sigma^\circ(\alpha)$ is the initial Saffman--Taylor function.
Thus, the mixing zone opens linearly in time at a rate proportional
to the strength of the initial instability. The viscosity jump
makes this opening asymmetric about $z$, with different speeds
on the two sides whenever $A\ne0$.

\begin{figure}[htbp]
    \centering
    \begin{subfigure}[t]{0.84\textwidth}
        \centering
        \includegraphics[width=\linewidth]{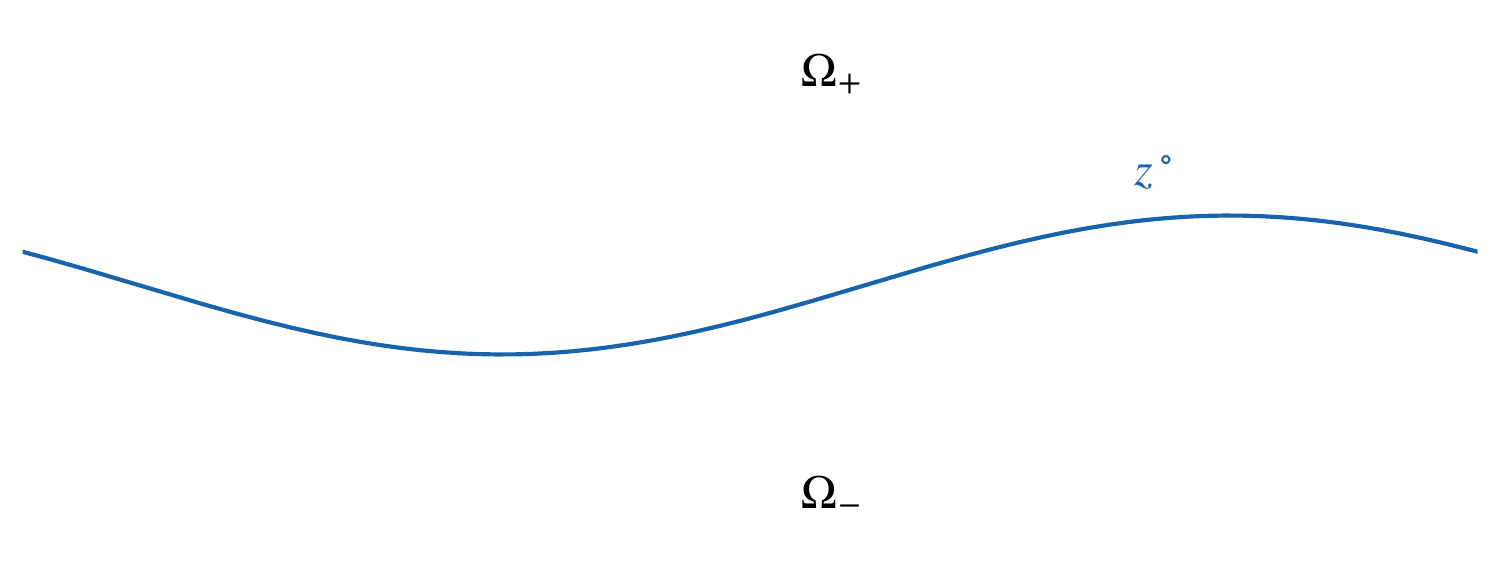}
        \caption{Initial interface $z^\circ$.}
        \label{fig:initial-interface}
    \end{subfigure}


    \begin{subfigure}[t]{0.84\textwidth}
        \centering
        \includegraphics[width=\linewidth]{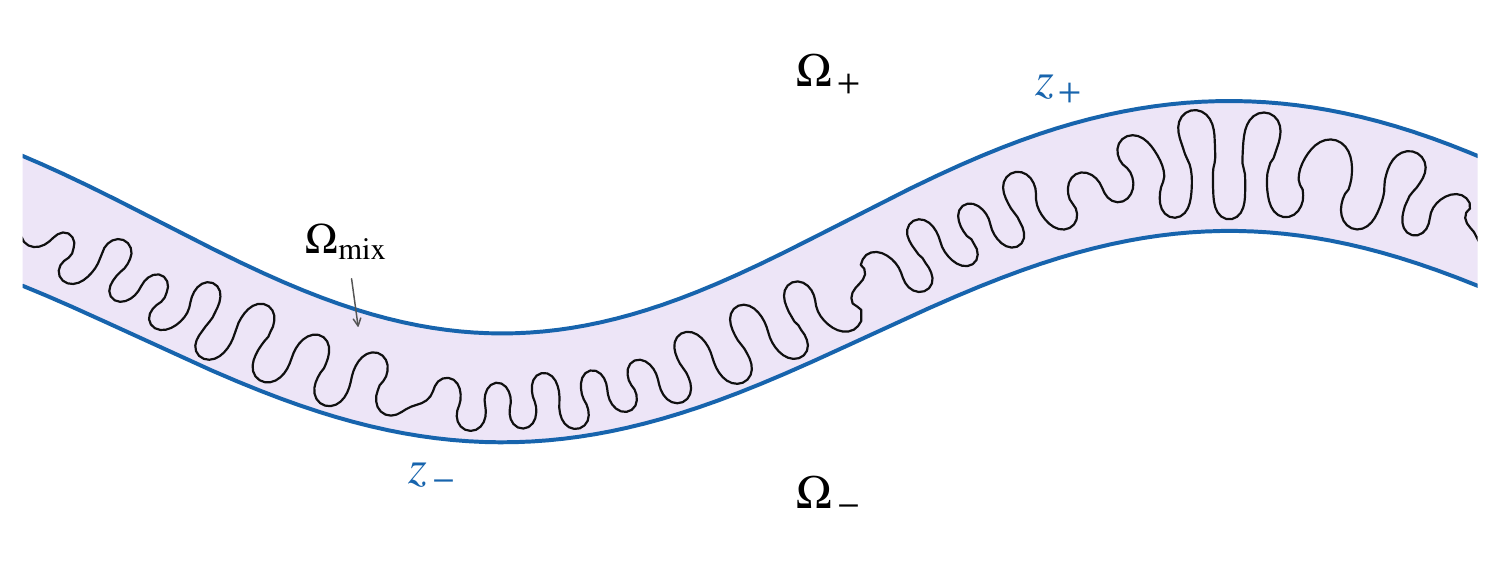}
        \caption{Mixing zone and Saffman--Taylor fingering.}
        \label{fig:mixing-fingers}
    \end{subfigure}

    \caption{%
(a) Initial interface
$z^\circ(\alpha)=(\alpha,-0.3\sin(\alpha-0.5))$. We use a periodic setting for numerical simplicity.
(b) Mixing zone (shaded violet) bounded by the curves
$z_+$ and $z_-$ (blue), for $A=0.3$ at $t=0.6$.
The fingers (black) result from a vortex-blob simulation
of a small multisine perturbation of $z^\circ$.
    }
    \label{fig:muskat-fingering:z}
\end{figure}
Inspired by spontaneous stochasticity in singular shear flows
\cite{TBM20}, we speculate that a suitable
subsolution may describe an ensemble average of flows arising
from vanishingly small initial perturbations, with its mixing
zone containing the Saffman--Taylor fingering region.
Figure~\ref{fig:muskat-fingering:z} illustrates this interpretation,
whose validity remains an open question.

\subsection{Organization of the paper.}
The remainder of the paper is organized as follows:
\begin{description}
  \item[Section~\ref{section2}]
  Introduces subsolutions and mixing solutions,
  recalls the $h$-principle, and states the relaxation results.

  \item[Section~\ref{section3}]
  Presents the Saffman--Taylor formulation of the Muskat problem
  and outlines the construction, starting with the flat-interface case.

  \item[Section~\ref{section4}]
  Defines the mixing zone and the macroscopic fields
  forming our subsolution.

  \item[Section~\ref{section5}]
  Derives the first- and second-order compatibility conditions
  governing the evolution of the mixing zone.

  \item[Section~\ref{section6}]
  Solves these conditions by inverting the Saffman--Taylor equation
  and establishes the regularity of the resulting coefficients.

  \item[Section~\ref{section7}]
  Proves the main theorem by verifying the strict subsolution
  condition and applying the $h$-principle.

  \item[Section~\ref{section8}]
  Collects the estimates for the Birkhoff--Rott and pressure
  operators used in the construction, including the inversion
  estimates required in Section~\ref{section6}.

  \item[Appendix~\ref{AppendixA}]
  Proves the relaxation results stated in Section~\ref{section2}.
\end{description}

\subsection{Notation}

\begin{itemize}
    \item We identify $\mathbb R^2$ with $\mathbb C$ via
$(v_1,v_2)\leftrightarrow v_1+iv_2$, so that $e_2$ corresponds to $i$.
We denote complex conjugation by $v^*$ and write $v^\perp=iv$.
Thus,
\[
v\cdot w=\operatorname{Re}(vw^*),
\qquad
\nabla^\perp g=(-\partial_2g,\partial_1g).
\]

    \item Interfaces are parametrized by $\alpha\in\mathbb R$.
    We suppress the time dependence when convenient. For asymptotically flat interfaces, Sobolev regularity
is understood for $z-(\alpha,0)$; for functions with a constant
far-field value $f_\infty$, it is understood for $f-f_\infty$.

    \item Across an oriented interface, we use the jump convention
\[
\llbracket f\rrbracket=f_{\mathrm{right}}-f_{\mathrm{left}},
\]
where the subscripts denote the right and left traces.
The indices $a,b=\pm1$ label the boundaries $z_\pm$
of the mixing zone, and $\llbracket f\rrbracket_a$
denotes the jump across $z_a$.

    \item Initial data carry a superscript $\circ$.
    Temporal derivatives at $t=0$ are denoted by
    \[
    f^{(j)}=\partial_t^j f\big|_{t=0}.
    \]
    In particular, $f^{(0)}$ denotes the initial value.

    \item We denote the Birkhoff--Rott operator by $B$, which determines
the mean velocity and pressure-gradient traces $U$ and $Q$.
The operator $P$ is obtained from the normal component of $Q$
along the initial interface.
See Section~\ref{section3}.
\end{itemize}


\section{H-principle and mixing solutions}\label{section2}
In this section, we recall the $h$-principle established in
\cite{Szekelyhidi12,Mengual22}. This yields mixing solutions
from suitable subsolutions. Therefore, our task in the subsequent sections
will be to construct such subsolutions.
In Section~\ref{subsec:relaxation}, we discuss the relaxation
underlying this construction.

The
\eqref{IPM} system can be reformulated as a differential inclusion by introducing a new variable $m$, which plays the role of the
momentum $\theta u$. By incompressibility, the transport
term can be written in divergence form as
$u\cdot\nabla\theta=\nabla\cdot(\theta u)$.
The linear part of the system is therefore
\begingroup
\renewcommand{\theequation}{mIPM}
\begin{subequations}\label{mIPM}
\begin{align}
\partial_t\theta+\nabla\cdot m
&=0,\tag{mIPMa}\label{mIPMa}\\
\nabla\cdot u
&=0,\tag{mIPMb}\label{mIPMb}\\
u+Am
&=-\nabla p-\theta i,
\tag{mIPMc}\label{mIPMc}
\end{align}
\end{subequations}
\endgroup
subject to the (nonlinear) pointwise constraint
$$
(\theta,u,m)\in K,
$$
where $K$ is the subset of
$\mathbb R\times\mathbb R^2\times\mathbb R^2$ defined by
\begin{equation}
m=\theta u,\qquad |\theta|=1.
\tag{$K$}\label{mIPMd}
\end{equation}

The relaxation of this differential inclusion consists of keeping
the same linear system \eqref{mIPM}, while replacing the pointwise constraint
$(\theta,u,m)\in K$ with
$$
(\bar\theta,\bar u,\bar m)\in\overline{\mathcal O},
$$
where $\overline{\mathcal O}$ denotes the closure in
$\mathbb R\times\mathbb R^2\times\mathbb R^2$ of the open set
$\mathcal O$ defined by
\begin{equation}
\left|2(1-A\theta)(m-\theta u)
+(1-\theta^2)(Au+i)\right|
<
(1-\theta^2)|Au+i|.
\tag{$\mathcal O$}\label{mathcalO}
\end{equation}
The precise meaning of this relaxation is discussed in
Section~\ref{subsec:relaxation}.

We now define subsolutions and (weak) solutions. We eliminate $p$ by taking the curl of Darcy's
law \eqref{mIPMc}. The pressure can subsequently be recovered by the Helmholtz decomposition.

\begin{definition}[Subsolution]
\label{definition_subsolution}
Given $T>0$, we say that a triple
\[
(\bar\theta,\bar u,\bar m)\in
C([0,T];L^\infty_{w^*}(\mathbb R^2;\mathbb R^5)),
\qquad |\bar\theta|\le1,
\]
is a \emph{subsolution} to \eqref{IPM} if it satisfies the linear
system \eqref{mIPM} in the sense of distributions and $(\bar\theta,\bar u,\bar m)\in\overline{\mathcal O}$
almost everywhere in $[0,T]\times\mathbb R^2$.

A subsolution is a \emph{solution} if $|\bar\theta|=1$
almost everywhere; in this case, the pointwise
constraint implies $\bar m=\bar\theta\bar u$.
\end{definition}

We recall the definition of mixing solution.
The mixing property \eqref{eq:mixing-property} was proved
in \cite{CCF21} and extended to every time slice in \cite{CFM19}.

\begin{definition}[Mixing solution]
\label{def:mixing-solution}
Let $\Omega_{\mathrm{mix}}\subset(0,T]\times\mathbb R^2$
be a nonempty relatively open set. A solution
$(\theta,u)$ to \eqref{IPM} is a \emph{mixing solution associated
with $\Omega_{\mathrm{mix}}$} if, for every $t\in(0,T]$,
$\Omega_{\mathrm{mix}}(t)$ is the largest open set such that
\begin{equation}\label{eq:mixing-property}
\int_B(1-\theta(t,x))\,dx
\int_B(1+\theta(t,x))\,dx>0
\end{equation}
for every nonempty bounded open set
$B\subset\Omega_{\mathrm{mix}}(t)$.
We denote by $\Omega_+(t)$ and $\Omega_-(t)$ the largest
open sets on which $\theta(t,\cdot)=1$ and
$\theta(t,\cdot)=-1$ almost everywhere, respectively.
\end{definition}

Finally, we recall the $h$-principle for the \eqref{IPM} system,
first proved for $A=0$ in \cite{Szekelyhidi12}, and later
extended to the case of viscosity jump in \cite{Mengual22}.

\begin{theorem}[$h$-principle for \eqref{IPM}]
\label{thm:h-principle}
Let $|A|<1$ and $T>0$. Let $\Omega_+$, $\Omega_-$,
and $\Omega_{\mathrm{mix}}$ be pairwise disjoint relatively
open subsets of $(0,T]\times\mathbb R^2$, with
$\Omega_{\mathrm{mix}}\neq\varnothing$, whose time slices
cover $\mathbb R^2$ up to boundaries of measure zero.

Suppose that $(\bar\theta,\bar u,\bar m)$ is a subsolution to \eqref{IPM}
with initial density $\bar\theta(0,\cdot)=\theta^\circ$,
where $|\theta^\circ|=1$ almost everywhere, satisfying
\[
\bar\theta=\pm1,\qquad \bar m=\bar\theta\bar u
\qquad\text{in }\Omega_\pm,
\]
and
\[
(\bar\theta,\bar u,\bar m)
\in C(\Omega_{\mathrm{mix}};\mathcal O).
\]
Then there exist infinitely many mixing solutions
$(\theta,u)$ associated with $\Omega_{\mathrm{mix}}$,
with initial density $\theta^\circ$, such that
\[
(\theta,u)=(\bar\theta,\bar u)
\qquad\text{outside }\Omega_{\mathrm{mix}}.
\]
\end{theorem}

Since the initial density $\theta^\circ$ and the solution
$\theta(t,\cdot)$ satisfy $|\theta^\circ|=|\theta(t,\cdot)|=1$
almost everywhere for every $t\in(0,T]$, weak continuity
in time implies
\[
\theta\in C\bigl([0,T];L^p_{\mathrm{loc}}(\mathbb R^2)\bigr)
\qquad\text{for every }1\le p<\infty.
\]

\subsection{Relaxation of \eqref{IPM}}
\label{subsec:relaxation}

Following \cite{Mengual22}, we use the wave cone
$\Lambda$ associated with plane waves of \eqref{mIPM} having nonzero
spatial frequency. Thus,
$\bar z=(\bar\theta,\bar u,\bar m)\in\Lambda$ if there
exist  $s\in\mathbb R$ and
$\xi\in\mathbb R^2\setminus\{0\}$ such that
\[
s\bar\theta+\xi\cdot\bar m=0,
\qquad
\xi\cdot\bar u=0,
\qquad
\xi^\perp\cdot(\bar u+A\bar m+\bar\theta e_2)=0.
\]
We now recall three notions of relaxation of the constraint set $K$. Se Appendix~\ref{AppendixA} for further details.

\begin{itemize}
\item The \emph{$\Lambda$-lamination hull} $K^{\mathrm{lc},\Lambda}$
consists of the states obtained from $K$ by finitely
many successive laminations along directions in $\Lambda$.

\item The \emph{$\Lambda$-convex hull} $K^\Lambda$ consists of
the states that cannot be separated from $K$ by a
$\Lambda$-convex function, that is, a function convex
along every direction in $\Lambda$.

\item The \emph{weak relaxation} $\mathcal R$ consists of
the constant states $z\in\mathbb R^5$ that can be
approximated by
$K$-valued solutions $z_j$ to the linear system \eqref{mIPM} in $C([0,T];L^2_w(B;\mathbb R^5))$,
for some $T>0$ and some bounded open ball
$B\subset\mathbb R^2$.
The pointwise constraint is imposed almost everywhere
in $(0,T)\times B$, and no initial or boundary conditions
are prescribed.
\end{itemize}

For $A=0$, these three notions of relaxation were shown
to coincide with $\overline{\mathcal O}$ in \cite{Szekelyhidi12}.
For $A\neq0$, the identity
$K^{\mathrm{lc},\Lambda}=\overline{\mathcal O}$
was established in \cite{Mengual22}.
The following theorem completes this characterization
by identifying $\overline{\mathcal O}$ also with
$K^\Lambda$ and $\mathcal R$.

\begin{theorem}[Relaxation of \eqref{IPM}]
\label{thm:relaxation}
For every $|A|<1$,
\[
K^{\mathrm{lc},\Lambda}
=
K^\Lambda
=
\mathcal R
=
\overline{\mathcal O}.
\]
\end{theorem}
\begin{proof}
See Appendix~\ref{sec:thm:relaxation}.
\end{proof}

To obtain bounded solutions through convex integration,
we also consider a bounded constraint set introduced in
\cite{Szekelyhidi12,Mengual22}.
For $M>1$, set
\[
M_\pm:=\sqrt{\frac{M^2\pm A}{1\pm A}},
\qquad
K_M:=\bigl\{(\theta,u,m)\in K:
|2u+\theta i|\le M_\theta\bigr\},
\]
where $M_\theta=M_\pm$ for $\theta=\pm 1$.
Let $\mathcal O_M$ be the bounded open subset of
$\mathcal O$ defined by
\begin{subequations}\label{eq:OM-constraints}
\begin{align}
4u\cdot(u+Am+\theta i+Ai)
&<M^2-1,
\label{eq:OM-quadratic}\\
|2(m-u)+(1-\theta)i|
&<M_-(1-\theta),
\label{eq:OM-minus}\\
|2(m+u)+(1+\theta)i|
&<M_+(1+\theta).
\label{eq:OM-plus}
\end{align}
\end{subequations}
We define $\mathcal R_M$ as above, with $K_M$-valued
solutions in place of $K$-valued solutions.

As in the unbounded case, the following characterization
was established for $A=0$ in \cite{Szekelyhidi12},
while the identity
$K_M^{\mathrm{lc},\Lambda}=\overline{\mathcal O_M}$
was proved in \cite{Mengual22}.

\begin{theorem}[Bounded relaxation of \eqref{IPM}]
\label{thm:bounded-relaxation}
For every $|A|<1$ and $M>1$,
\[
K_M^{\mathrm{lc},\Lambda}
=
K_M^\Lambda
=
\mathcal R_M
=
\overline{\mathcal O_M}.
\]
\end{theorem}
\begin{proof}
See Appendix~\ref{sec:thm:bounded-relaxation}.
\end{proof}

\section{Sketch of the proof}\label{section3}
The goal of this section is to outline the proof of Theorem~\ref{thm:main}.
In Section~\ref{sec:sketch:ST}, we begin by deriving a Saffman--Taylor
formulation \eqref{Msigma:intro} of the classical Muskat problem, which will be useful
for constructing the subsolution.
We then examine the flat-interface case in Section~\ref{sec:sketch:flat}
to motivate the ansatz.
Finally, Section~\ref{sec:sketch:fully} presents the main ideas
of the construction for fully unstable interfaces,
providing a roadmap for Sections~\ref{section4}--\ref{section8}.

\subsection{A Saffman--Taylor formulation of the Muskat problem}\label{sec:sketch:ST}

The classical Muskat problem corresponds to the \eqref{IPM} system
under the \textit{Muskat assumption}: the phase variable $\theta$ remains constant in the two regions
$\Omega_{\mathrm{left}}$ and $\Omega_{\mathrm{right}}$ separated by an interface
parametrized by $z=z(t,\alpha)$. The labels \textit{left} and \textit{right} refer to the
prescribed orientation of $z$. This convention is also used for the
corresponding restrictions of the phase $\theta$, velocity $u$, and pressure $p$ to each region.

The velocity and pressure are determined in terms of the phase through the
incompressibility condition and Darcy's law \eqref{IPMb}--\eqref{IPMc}. The variable $\theta$ is normalized so as to take opposite values in the two
regions, namely one of them being $+1$ and the other $-1$. The regions $\Omega_{\mathrm{left}}(t)$ and $\Omega_{\mathrm{right}}(t)$ evolve
in time according to the interface $z(t)$, which encodes the full dynamics of
the problem through the phase conservation equation \eqref{IPMa} and is the
unknown to be determined.

\begin{figure}[htbp]
    \centering
    \includegraphics[width=0.84\textwidth]{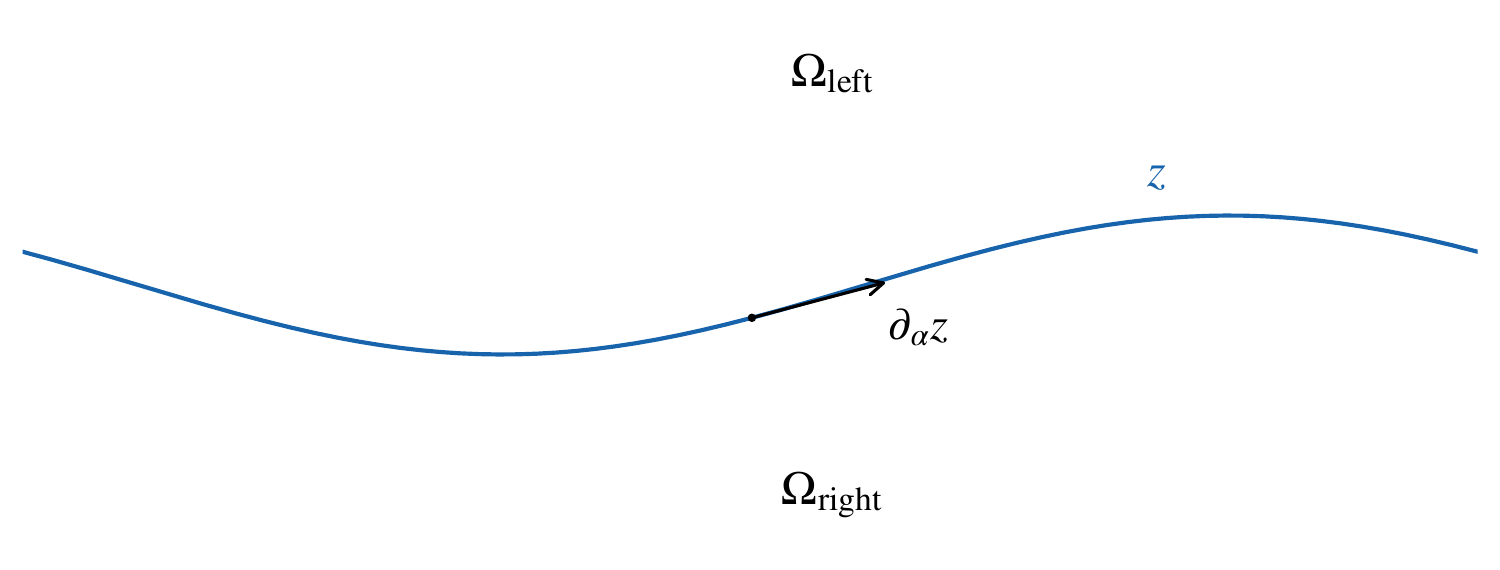}
    \caption{%
        The Muskat assumption: two pure phases separated by an
        oriented interface $z=z(t,\alpha)$, with $\theta(t)$ constant in each region
        $\Omega_{\mathrm{left}}(t)$ and $\Omega_{\mathrm{right}}(t)$.
    }
    \label{fig:oriented-interface}
\end{figure}

\subsubsection{Solving \eqref{IPMa}}\label{sec:IPMa}
Away from the interface $z$, the phase conservation equation reduces to
$\nabla\cdot u=0$, namely \eqref{IPMb}. Across the interface, the weak
formulation of \eqref{IPMa} yields
\begin{equation}\label{eq:Muskat:1}
\llbracket\theta\rrbracket(\partial_t z-U)\cdot\partial_\alpha z^\perp=0,
\end{equation}
where $U=\frac{1}{2}(u_{\mathrm{left}}+u_{\mathrm{right}})$ denotes the average velocity trace evaluated on the interface $z$. This determines the evolution of $z$ up to tangential terms, which reflect the
freedom of parametrization.
The expression of the operator $U\cdot\partial_\alpha z^\perp$ will be derived in next subsection.

\subsubsection{Solving \eqref{IPMb}--\eqref{IPMc}}\label{sec:IPMbc}
First, the incompressibility condition implies that the velocity can be written as $u=\nabla^\perp\psi$, for some stream
function $\psi$. Second, Darcy's law implies that both $\psi$ and $p$ are
harmonic away from the interface. Moreover, $\psi$ and $p$ are given by the Poisson equations
$$
\Delta\psi=\varpi\delta_z,
\qquad
\Delta p=\sigma\delta_z,
$$
in terms of some vorticity and Saffman--Taylor densities $\varpi$ and $\sigma$, to be determined.
With the usual far-field normalization, the Biot--Savart law
determines $u=u(t,x)$ and $\nabla p=\nabla p(t,x)$, for $x$ away from $z$,
\begin{equation}\label{Muskat:uDp}
u^*=\frac{1}{2\pi i}\int\frac{\varpi(\beta)}{x-z(\beta)}d\beta,
\qquad
(\nabla p)^*
=\frac{1}{2\pi}\int\frac{\sigma(\beta)}{x-z(\beta)}d\beta,
\end{equation}
where $*$ denotes complex conjugation. Using the Sokhotski--Plemelj formulas, the limits of \eqref{Muskat:uDp} as $x\to z(\alpha)$ from the
right and from the left are, respectively,
$$
u\rightarrow U\pm\frac{1}{2}\frac{\varpi}{\partial_\alpha z^*},
\qquad
\nabla p\rightarrow Q\mp\frac{1}{2}\frac{i\sigma}{\partial_\alpha z^*},
$$
with 
$U^*=B[\varpi,z]$ and $Q^*=iB[\sigma,z]$, 
where $B$ denotes the Birkhoff--Rott operator
$$
B[f,z](\alpha):=\frac{1}{2\pi i}\int\frac{f(\beta)}{z(\alpha)-z(\beta)}d\beta,
$$
with the integral understood in the principal value sense.
Substituting the right and left traces into Darcy's law, 
we deduce the relations
\begin{subequations}
\begin{align}
(1+\theta_r A)\left(U+\frac{1}{2}\frac{\varpi}{\partial_\alpha z^*}\right)+\left(Q-\frac{1}{2}\frac{i\sigma}{\partial_\alpha z^*}\right)&=-i\theta_r,\label{eq:Darcy:right}\\
(1+\theta_l A)\left(U-\frac{1}{2}\frac{\varpi}{\partial_\alpha z^*}\right)+\left(Q+\frac{1}{2}\frac{i\sigma}{\partial_\alpha z^*}\right)&=-i\theta_l\label{eq:Darcy:left}.
\end{align}
\end{subequations}
where we have abbreviated $\theta_r=\theta_{\text{right}}$ and $\theta_l=\theta_{\text{left}}$.

Subtracting these expressions, first
$\eqref{eq:Darcy:right}-\eqref{eq:Darcy:left}$ and then
$(1+\theta_l A)\eqref{eq:Darcy:right}
-(1+\theta_r A)\eqref{eq:Darcy:left}$, we obtain
\begin{subequations}\label{eq:Darcy:subtrac}
\begin{align}
\varpi+i\sigma
&=-\llbracket\theta\rrbracket(AU+i)^*\partial_\alpha z,\label{eq:Darcy:subtrac:1}\\
(1-A^2)\varpi+i\sigma
&=+\llbracket\theta\rrbracket(AQ-i)^*\partial_\alpha z.\label{eq:Darcy:subtrac:2}
\end{align}
\end{subequations}
Similarly, adding these expressions, we obtain
\begin{subequations}\label{eq:Darcy:add}
\begin{align}
U+Q
&=-\frac{1}{4}A\llbracket\theta\rrbracket\frac{\varpi}{\partial_\alpha z^*},\label{eq:Darcy:add:1}\\
(1-A^2)U+Q
&=
-\frac{1}{4}A\llbracket\theta\rrbracket
\frac{i\sigma}{\partial_\alpha z^*}
+Ai.\label{eq:Darcy:add:2}
\end{align}
\end{subequations}
In particular, we deduce that
\begin{equation}\label{U+Q:tangent}
(U+Q)\cdot\partial_\alpha z^\perp=0.
\end{equation}

\subsubsection{Recovering $(\varpi,\sigma)$}
The vorticity and Saffman--Taylor densities are given by
\begin{subequations}
\begin{align}
\varpi
&=+\llbracket u\rrbracket\cdot\partial_\alpha z,
&\sigma&=-\llbracket \nabla p\rrbracket\cdot\partial_\alpha z^\perp.\notag
\intertext{Taking the real and imaginary parts in \eqref{eq:Darcy:subtrac}, they satisfy}
\varpi&=-\llbracket\theta\rrbracket(AU+i)\cdot\partial_\alpha z ,
&\sigma&=+\llbracket\theta\rrbracket(AU+i)\cdot\partial_\alpha z^\perp,\label{eq:Darcy:subtrac:realimag:1}\\
(1-A^2)\varpi
&=+\llbracket\theta\rrbracket(AQ-i)\cdot\partial_\alpha z,
&\sigma&=-\llbracket\theta\rrbracket(AQ-i)\cdot\partial_\alpha z^\perp.\label{eq:Darcy:subtrac:realimag:2}
\intertext{For $A=0$, these identities reduce to}
\varpi&=-\llbracket\theta\rrbracket\partial_\alpha z_2,
&\sigma&=+\llbracket\theta\rrbracket\partial_\alpha z_1,\notag
\end{align}
\end{subequations}
or, more compactly, \eqref{eq:Darcy:subtrac} reduces to
$\varpi+i\sigma
=\llbracket\theta\rrbracket i\partial_\alpha z.$
In addition, \eqref{eq:Darcy:add} reduces to
$$
0=(U+Q)^*
=B[\varpi+i\sigma,z]
=\frac{\llbracket\theta\rrbracket}{2\pi }\int\frac{\partial_\alpha z(\beta)}{z(\alpha)-z(\beta)}d\beta.
$$
For $A\neq0$, these relations are no longer explicit,
since $U^*=B[\varpi,z]$ and $Q^*=iB[\sigma,z]$.
One may either solve the relation for $\varpi$ on the left of
\eqref{eq:Darcy:subtrac:realimag:1} and then recover $\sigma$
from the relation on the right,
or either solve the relation for $\sigma$
on the right of \eqref{eq:Darcy:subtrac:realimag:2}
and then recover $\varpi$ from the relation on the left.
It turns out that both problems reduce to inverting operators of the form $I\pm AP$:
\begin{align*}
(I+AP)\varpi&=-\llbracket\theta\rrbracket\partial_\alpha z_2,
&(I-AP)\sigma&=+\llbracket\theta\rrbracket\partial_\alpha z_1,
\end{align*}
where $P$ is the real operator
$$
P[f]:=\llbracket\theta\rrbracket B[f,z]^*\cdot\partial_\alpha z
=\llbracket\theta\rrbracket 
\left(\frac{1}{2\pi i}\int f(\beta)\frac{\partial_\alpha z(\alpha)}{z(\alpha)-z(\beta)}d\beta\right)_1.
$$

In \cite{CCG11}, Sobolev estimates for $P$ and quantitative
bounds for the inverse of $I-\xi P$, $|\xi|\le1$, are established
in the horizontally periodic setting under suitable regularity assumptions
on the interface.
The operator $T$ in \cite{CCG11} coincides with $P$ up to
a sign determined by the phase jump.
For asymptotically flat curves, the corresponding Sobolev
estimates and the inversion needed here are developed
in Section~\ref{section8} for completeness.

\subsubsection{Summary}\label{sec:Muskat:summary}
We have seen that the IPM system under the Muskat assumption is
equivalent to solving the Cauchy problem for $z$, referred to as the \textit{Muskat equation}:
\begin{align}
(\partial_t z-U[\varpi,z])\cdot\partial_\alpha z^\perp&=0,
&\varpi&=-\llbracket\theta\rrbracket(AU[\varpi,z]+i)\cdot\partial_\alpha z.\tag{M$\varpi$}\label{Mvarpi}
\intertext{Using \eqref{U+Q:tangent}, this can be rewritten as}
(\partial_tz +Q[\sigma,z])\cdot\partial_\alpha z^\perp&=0,
&\sigma&=-\llbracket\theta\rrbracket(AQ[\sigma,z]-i)\cdot\partial_\alpha z^\perp.\tag{M$\sigma$}\label{Msigma}
\end{align}
Recall that $U^*=B[\varpi,z]$ and  $Q^*=iB[\sigma,z]$, where $B$ denotes the Birkhoff-Rott operator.

Although the two formulations are equivalent, \eqref{Mvarpi} is the form that
appears recurrently in the literature, whereas \eqref{Msigma} seems to be new,
to the best of our knowledge. Notice that the latter allows us to
rewrite the equation for $\sigma$ as
\begin{equation}\label{eq:sigma:zt}
\sigma=\llbracket\theta\rrbracket
\bigl(A\partial_t z+i\bigr)\cdot\partial_\alpha z^\perp.
\end{equation}
The advantage of this formulation becomes apparent when passing from the classical Muskat problem to subsolutions with a mixing zone: the pressure remains harmonic in each region, whereas the velocity may develop vorticity in the interior of the mixing zone. We explain this distinction at the end of Section~\ref{sec:sketch:fully}.

\subsubsection{Alternative derivations}\label{section3.1.5}
We now present two more direct derivations of the Muskat equation.
The first one corresponds to \eqref{Mvarpi} and is the standard approach in the
literature. The second one corresponds to the new formulation \eqref{Msigma}
and, as already mentioned, will play a crucial role in our treatment of the
unstable regime. 
In both cases, the derivation of the evolution equation
\eqref{eq:Muskat:1}
follows Section~\ref{sec:IPMa}, while the expression of the operator
$U\cdot\partial_\alpha z^\perp$ is obtained through two different shortcuts within the
approach of Section~\ref{sec:IPMbc}.
\begin{itemize}
    \item[\eqref{Mvarpi}] The first approach starts from the observation that $\psi$ must be harmonic away
from $z$, which leads to the velocity representation in \eqref{Muskat:uDp}.
This automatically satisfies the incompressibility condition \eqref{IPMb}. Then,
instead of using the pressure-gradient representation in \eqref{Muskat:uDp}, one
can define the pressure directly through Darcy's law \eqref{IPMc} 
\begin{equation}\label{eq:Dp}
-\nabla p:=u+A\theta u+\theta i,
\end{equation}
provided that the right-hand side of \eqref{eq:Dp} is curl-free. This condition
holds trivially away from $z$, while across the interface it requires the jump
to be normal to the interface, that is, $\llbracket \nabla p \rrbracket \cdot \partial_\alpha z = 0.$
    This condition leads to
    \begin{equation}\label{eq:Mvarpi:2}
    \varpi
        =
        -\llbracket \theta \rrbracket (AU+i)\cdot \partial_\alpha z.
    \end{equation}
    Combining \eqref{eq:Muskat:1} and \eqref{eq:Mvarpi:2} yields \eqref{Mvarpi}.
    \item[\eqref{Msigma}] The second approach starts from the observation that $p$ must be harmonic away
from $z$, which leads to the pressure-gradient representation in \eqref{Muskat:uDp}. Then, instead of using the velocity representation in \eqref{Muskat:uDp}, one defines the velocity directly through Darcy's law \eqref{IPMc} 
    $$
u:=-\frac{1}{1+\theta A}(\theta i+\nabla p).
$$
    This is automatically divergence-free away from \(z\). To verify \eqref{IPMb}, it remains to check that the jump is tangential to the interface, that is, $\llbracket u \rrbracket \cdot \partial_\alpha z^\perp = 0$.
    This condition leads to
    \begin{equation}\label{eq:Msigma:2}
    \sigma
        =
        -\llbracket \theta \rrbracket (AQ-i)\cdot \partial_\alpha z^\perp.
    \end{equation}
    Equivalently, taking the divergence in Darcy's law \eqref{IPMc} and using \eqref{IPMa},
the incompressibility condition \eqref{IPMb} becomes $\Delta p =A\partial_t\theta-\nabla\cdot(\theta i)$, which yields \eqref{eq:sigma:zt}.

    Using that $U\cdot\partial_\alpha z^\perp = u_r\cdot\partial_\alpha z^\perp$, together with $\llbracket\theta\rrbracket=2\theta_r$, and \eqref{eq:Msigma:2}, we deduce \eqref{U+Q:tangent}
    \begin{align*}
    U\cdot\partial_\alpha z^\perp
    &=-\frac{1}{1+\theta_r A}\left(i\theta_r+Q-\frac{1}{2}\frac{i\sigma}{\partial_\alpha z^*}\right)\cdot\partial_\alpha z^\perp\\
    &=-\frac{1}{1+\theta_r A}\left((i\theta_r+Q)\cdot\partial_\alpha z^\perp-\frac{\sigma}{2}\right)
    =-Q\cdot\partial_\alpha z^\perp.
    \end{align*}
    Combining this with \eqref{eq:Muskat:1} and \eqref{eq:Msigma:2} yields \eqref{Msigma}.
\end{itemize}

\subsection{The flat interface}\label{sec:sketch:flat}

In this subsection, we construct a suitable subsolution
$(\bar\theta,\bar u,\bar m)$ for the unstable flat interface
$$
\theta^\circ(x)
=\pm 1,\qquad
\pm x_2>0.
$$
Following the piecewise-constant approach of
\cite{ForsterSzekelyhidi18,CFM22}, we declare 
$$
\bar{\theta}(t,x)
=\left\lbrace
\begin{array}{cc}
   \pm 1,  & \pm x_2>h_{\pm}t, \\[0.1cm]
   \bar{\theta}_{\text{mix}},  & -h_-t<x_2<h_+t,
\end{array}
\right.
$$
with $h_\pm>0$ and $\bar{\theta}_{\text{mix}}$ constants to be determined.

\begin{figure}[htbp]
    \centering
    \begin{subfigure}[t]{0.84\textwidth}
        \centering
        \includegraphics[width=\linewidth]{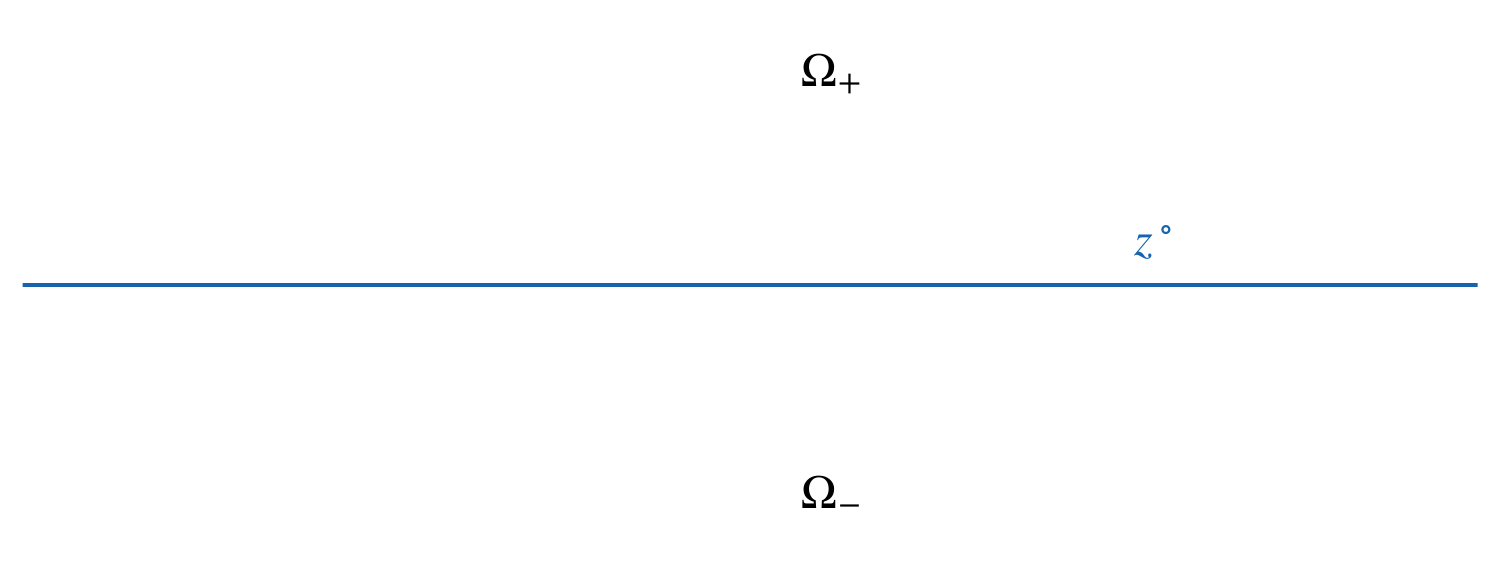}
        \caption{Flat initial interface $z^\circ(\alpha)=\alpha$.}
        \label{fig:flat-initial}
    \end{subfigure}

    \par\medskip

    \begin{subfigure}[t]{0.84\textwidth}
        \centering
        \includegraphics[width=\linewidth]{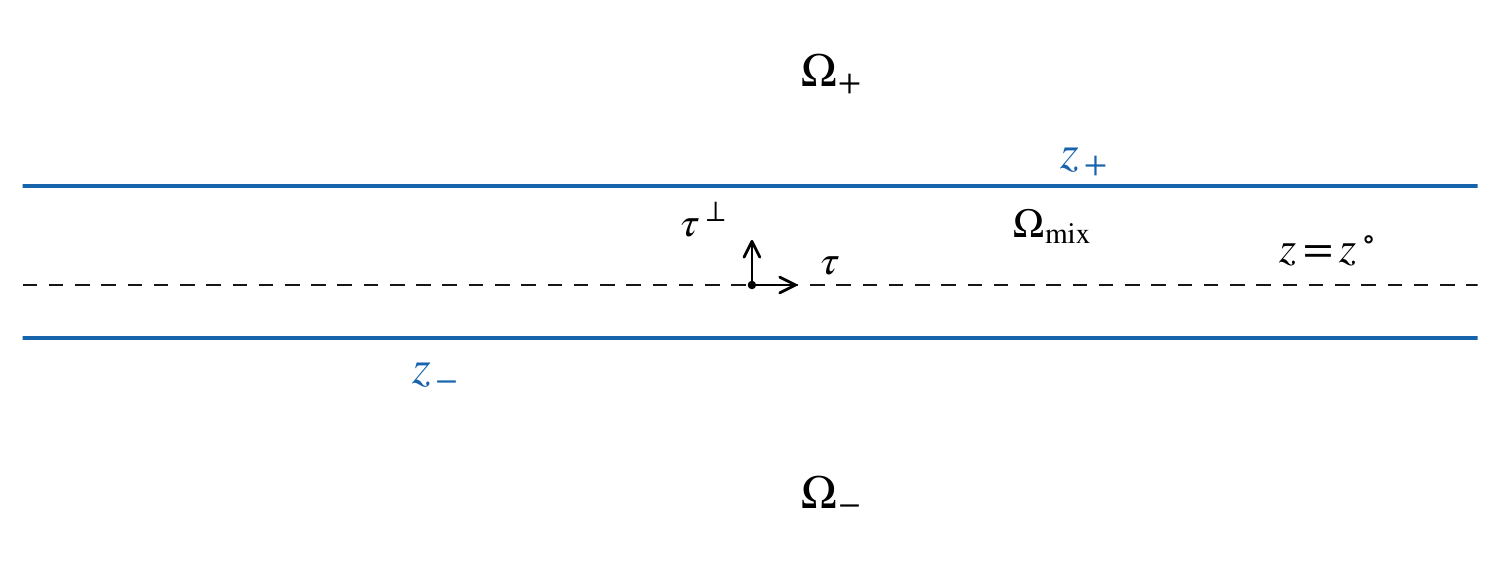}
        \caption{Mixing zone for the flat interface.}
        \label{fig:flat-mixing}
    \end{subfigure}

    \caption{%
        Mixing-zone geometry for the flat initial interface.
        The boundaries $z_\pm(t,\alpha)=(\alpha,\pm h_\pm t)$
        are shown in blue, while $z=z^\circ$ is dashed.
        Here $\theta_{\mathrm{mix}}=A=0.3$,
        and $t=0.60$, so that $h_\pm=1/(1\mp A)$.
        The vectors $\tau$ and $\tau^\perp$ denote
        the unit tangent and normal to the initial interface.
    }
    \label{fig:flat-interface}
\end{figure}

Moreover, we declare
$$
\bar{u}=0
$$
which satisfies $\nabla\cdot{\bar{u}}=0$ trivially.

The hull condition \eqref{mathcalO} suggest to make the change of variables
$$
\bar{m}
=-\frac{1-\bar{\theta}^2}{1-\bar{\theta}A}(\gamma+\tfrac{1}{2})i,
$$
in terms of a constant $\gamma\in\R$, to be determined. Indeed, the hull condition reads as
$$
|\gamma|<\tfrac{1}{2}.
$$
Darcy's law is automatically satisfied outside $\pm x_2=h_\pm t$. Inside, it is equivalent to 
$$
\llbracket\bar{m}+\bar{\theta}i\rrbracket_\pm
\cdot(1,0)=0,
$$
which is satisfied as well.

Finally, the conservation of mass is equivalent to
$$
(1\mp\bar{\theta}_{\text{mix}})h_\pm
=\frac{1-\bar{\theta}_{\text{mix}}^2}{1-\bar{\theta}_{\text{mix}}A}(\gamma+\tfrac{1}{2})
$$
which determine $h_\pm$ in terms of $\bar{\theta}_{\text{mix}}$ and $\gamma$.
We observe that the quantity
$$
h_+h_-
=\frac{1-\bar{\theta}_{\text{mix}}^2}{(1-\bar{\theta}_{\text{mix}}A)^2}(\gamma+\tfrac{1}{2})^2
$$
is maximized for $\bar{\theta}_{\text{mix}}=A$ and $\gamma\to\tfrac{1}{2}$. Hence, $h_\pm=1/(1\mp A)$.

\subsection{Fully unstable interfaces}\label{sec:sketch:fully}

In this subsection, we outline how to extend the construction
of Section~\ref{sec:sketch:flat} to more general curves using
the Saffman--Taylor formulation from Section~\ref{sec:sketch:ST}.
We first describe how to solve the macroscopic equations
for a prescribed boundary motion, and then explain how this
motion ensures the strict hull condition.

Motivated by the flat-interface construction in Section~3.2,
we set $\bar\theta=\pm1$ in the pure phases and $\bar\theta=A$
in the mixing zone, bounded by
\[
    z_a(t,\alpha)
    =
    z(t,\alpha)
    +\frac{a}{1-aA}\,t c(\alpha)\tau(\alpha)^\perp,
    \qquad a=\pm1.
\]
Here
$\tau=\partial_\alpha z^\circ/|\partial_\alpha z^\circ|$,
and both the central curve $z$ and the positive opening speed $c$
remain to be chosen.
The factors $1/(1-aA)$ describe the asymmetric opening induced
by the viscosity jump.
With the orientation convention of Section~\ref{sec:sketch:ST},
the phase jumps are
$\llbracket\bar\theta\rrbracket_a=-(1-aA)$.

We introduce an unknown scalar potential $g$ in the mixing zone and prescribe
\[
    \bar m
    =
    \bar\theta\,\bar u
    -
    \frac{1-\bar\theta^2}{1-A\bar\theta}
    \left(
        \nabla^\perp g+\frac12(A\bar u+i)
    \right).
\]
This gives $\bar m=\bar\theta\bar u$ in the pure phases
and reduces the strict hull condition in the mixing zone to
\[
    2|\nabla g|<|A\bar u+i|.
\]
The geometry, phase, and momentum are introduced in
Sections~\ref{section4.1}--\ref{sec:momentum};
the hull reduction is established in Section~\ref{sec:hullcondition}.

\subsubsection{Solving (mIPMa)}

The phase conservation \eqref{mIPMa} reduces to \eqref{mIPMb} away from the boundaries.
Across each $z_a$, it requires
\[
    \left(
        \llbracket\bar\theta\rrbracket_a\,\partial_t z_a
        -
        \llbracket\bar m\rrbracket_a
    \right)
    \cdot\partial_\alpha z_a^\perp=0.
\]
This is the counterpart of the kinematic condition
in Section~\ref{sec:IPMa}.
Here the flux is determined by $\bar m$,
so the mixing boundaries need not move with the normal fluid velocity.

Let $\bar U_a$ denote the velocity trace on $z_a$
from the pure phase $\Omega_a$, and set $G_a=g\circ z_a$.
Assuming continuity of the normal velocity,
which will be imposed below, the jump condition becomes
\[
    a\,\partial_\alpha G_a
    =
    \left(
        (1-aA)(\partial_t z_a-\bar U_a)
        -\frac a2(A\bar U_a+i)
    \right)
    \cdot\partial_\alpha z_a^\perp.
\]
Thus the boundary motion and the pure-phase velocity traces
determine the tangential derivatives of $G_\pm$.
We normalize $G_\pm(t,0)=0$ and interpolate linearly across the strip:
\[
    G(t,\alpha,\lambda)
    =
    \frac{1+\lambda}{2}G_+(t,\alpha)
    +
    \frac{1-\lambda}{2}G_-(t,\alpha),
    \qquad -1\le\lambda\le1.
\]
This defines $g$ through $g\circ Z=G$, where
\[
    Z(t,\alpha,\lambda)
    =
    z(t,\alpha)
    +t\,\frac{\lambda+A}{1-A^2}\,
    c(\alpha)\tau(\alpha)^\perp.
\]
parametrizes the mixing zone, that is $\Omega_{\textrm{mix}}(t)=Z(t,\R\times(-1,1))$.
The boundary conditions and interpolation are developed
in Section~\ref{sec:potential}.

\subsubsection{Solving (mIPMb)}
\label{sec:sketch-incompressibility}

Darcy's law, combined with the momentum ansatz and
incompressibility, requires the pressure to be harmonic
within each region. Indeed, the coefficients are constant
there and the correction $\nabla^\perp g$ is divergence-free.
As in Section~\ref{sec:sketch:ST}, we therefore represent
the pressure gradient by Saffman--Taylor densities
supported on the two boundaries:
\[
    (\nabla\bar{p})^*(t,x)
    =
    \sum_{b=\pm1}\frac{1}{2\pi}
    \int_{\mathbb R}
    \frac{\sigma_b(t,\beta)}
    {x-z_b(t,\beta)}\,d\beta.
\]
Here the integral is understood in the principal value sense
at infinity, and boundary traces are obtained by the
Sokhotski--Plemelj formulas.

It remains to impose continuity of the normal velocity
across $z_\pm$. Using Darcy's law and phase conservation,
these conditions eliminate $\partial_\alpha G_a$ and give
\[
    \sigma_a
    =
    \llbracket\bar\theta\rrbracket_a
    (A\partial_t z_a+i)
    \cdot\partial_\alpha z_a^\perp,
\]
the analogue of the classical Saffman--Taylor
identity~\eqref{eq:sigma:zt}.
Thus the boundary motion determines the pressure densities
and hence the pressure.
These prescriptions are verified in
Section~\ref{sec:pressure}.

\subsubsection{Solving (mIPMc)}

Once the pressure is determined, Darcy's law gives
the velocity explicitly. Substituting the momentum ansatz
yields
\[
    \bar u
    =
    \begin{cases}
        \displaystyle
        -\frac{\nabla\bar{p}+ai}{1+aA},
        &\text{in }\Omega_a,\quad a=\pm1,
        \\[8pt]
        \displaystyle
        \frac{A\nabla^\perp g-\nabla\bar{p}-\frac A2i}
        {1+\frac{A^2}{2}},
        &\text{in }\Omega_{\mathrm{mix}}.
    \end{cases}
\]
The pure-phase velocity traces depend only on the pressure
and can therefore be computed before extending $g$
into the mixing zone. The potential affects the interior
velocity through the explicit local term $A\nabla^\perp g$.
These formulas are given in Section~\ref{sec:velocity}.

The pure-phase velocity traces then determine the boundary
potentials $G_\pm$ through phase conservation.
Their interpolation, described in
Section~\ref{sec:potential}, completes the construction
of the macroscopic fields.

\subsubsection{Compatibility conditions and the hull condition}

The remaining task is to choose the boundary motion and opening speed
so that the potential satisfies the strict hull inequality.
Following \cite{ForsterSzekelyhidi18,CFM22}, we declare a second order expansion in time of the central curve
\[
    z(t)
    =
    z^\circ+t z^{(1)}+\frac12t^2z^{(2)}.
\]
At first order, we solve the initial Saffman--Taylor equation
for $\sigma^{(0)}$ and set
\[
    z^{(1)}=-Q_0[\sigma^{(0)}],
\]
where $Q_0$ is the pressure operator on the initial interface.
Its normal component agrees with the initial Muskat velocity.
This choice gives the same leading tangential derivative
on both boundaries and hence
\[
    \partial_\alpha G
    =
    c|\partial_\alpha z^\circ|
    +\frac14\sigma^{(0)}
    +O(t).
\]
The second-order coefficient $z^{(2)}$ is chosen to cancel
the first temporal derivative of
$\partial_\alpha G_+-\partial_\alpha G_-$ at $t=0$.
Together with the primitive estimates, this yields
\[
    \partial_\lambda G
    =
    \frac12(G_+-G_-)
    =
    O(t^2).
\]
This extra order is essential:
the strip has width proportional to $t$,
so the physical transverse derivative involves
$t^{-1}\partial_\lambda G$.
The resulting gradient expansion is
\[
    \nabla g\circ Z
    =
    \left(
        c+\frac{\sigma^{(0)}}{4|\partial_\alpha z^\circ|}
    \right)\tau
    +O(t).
\]
The first- and second-order compatibility conditions are derived
in Sections~5.2--5.3.
Both steps involve the same Saffman--Taylor inversion,
stated in Proposition~\ref{p:inverse};
Section~\ref{section6} establishes the solvability
and regularity of the coefficients.
The uniform estimates above are proved in Sections~\ref{section7.1}--\ref{section7.2}.
Section~\ref{section7.3} supplies the additional ingredient
required by the viscosity jump, a uniform positive lower bound for $|A\bar u+i|$
throughout the mixing zone.

We finally choose
\[
    c
    =
    -\bar c\,\frac{\sigma^{(0)}}{|\partial_\alpha z^\circ|},
    \qquad
    0<\bar c<\frac12.
\]
Uniform instability makes $c$ uniformly positive.
Combining these identities give
the strict hull inequality for sufficiently small times.
Section~\ref{section7.4} completes the subsolution verification
and applies the h-principle.
The operator estimates underlying the argument
are collected in Section~\ref{section8}.

\subsubsection{Why the pressure formulation?}

The pressure remains harmonic in each region,
whereas the mixing zone generally carries interior vorticity. Namely, the curl of \eqref{mIPMc} reads as
\[
    \left(1+\frac{A^2}{2}\right)
    \operatorname{curl}\bar u
    =
    A\Delta g.
\]
A direct vorticity formulation would therefore couple
the velocity reconstruction to the interior extension of $g$.
The Darcy jump conditions would also involve
its normal derivatives at the boundaries.
The pressure construction avoids this coupled reconstruction.
It determines the pure-phase velocity traces before extending $g$,
leaving the explicit interpolation available for the hull estimates.

\section{The subsolution}\label{section4}
In this section, we give the subsolution ansatz outlined in
Section~\ref{sec:sketch:fully}. We first define the mixing zone
and the macroscopic phase, momentum, pressure, and velocity.
We then construct the potential $g$ to ensure phase conservation
and continuity of the normal velocity, and express the hull
condition in terms of its gradient. The choice of the interface
motion and the verification of this condition are left to the
subsequent sections.

\subsection{The mixing zone}\label{section4.1}
Let $z^\circ$ be the initial chord--arc curve, with the orientation convention of Section~\ref{section3}, and
$z^\circ-(\alpha,0)\in H^k(\mathbb R;\mathbb R^2)$ with
$k\geq5$. Let
\[
    \tau(\alpha):=
    \frac{\partial_\alpha z^\circ(\alpha)}
    {|\partial_\alpha z^\circ(\alpha)|}
\]
denote its unit tangent. We fix a function
$c\in C^1(\mathbb R)$ satisfying
\[
    0<c_*\leq c(\alpha)\leq c^*<\infty.
\]
The choice of $c$ will be specified in Section~\ref{section5}.

Consider a central curve of the form
\[
    z(t,\alpha)=z^\circ(\alpha)+t\,z^{(1)}(\alpha)
    +\frac{t^2}{2}\,z^{(2)}(\alpha),
\]
where $z^{(1)},z^{(2)}\in C^1$ are the first- and
second-order coefficients of the expansion in time.
For $\lambda\in[-1,1]$, define
\begin{equation}\label{eq:mixing-coordinates}
    Z(t,\alpha,\lambda)=z_\lambda(t,\alpha)
    :=z(t,\alpha)+t\,h(\lambda)c(\alpha)\tau^\perp(\alpha),
    \qquad
    h(\lambda):=\frac{\lambda+A}{1-A^2}.
\end{equation}
In particular, the two boundary curves are
\begin{equation}\label{eq:mixing-boundaries}
    z_a(t,\alpha)
    =z(t,\alpha)+\frac{a}{1-aA}\,
      t\,c(\alpha)\tau^\perp(\alpha),
    \qquad a=\pm1.
\end{equation}
Thus the opening is prescribed along the initial normal,
with different speeds on the two sides when $A\neq0$.
For $0<t\leq T$, we define
\[
    \Omega_{\mathrm{mix}}(t)
    :=Z\bigl(t,\mathbb R\times(-1,1)\bigr).
\]
We denote by $\Omega_+(t)$ the region to the left of $z_+(t)$,
and by $\Omega_-(t)$ the region to the right of $z_-(t)$.
These are the two exterior regions.

At $t=0$, both boundary curves coincide with $z^\circ$,
and the mixing region collapses to the initial interface.

\begin{figure}[htbp]
    \centering
    \includegraphics[width=0.84\textwidth]{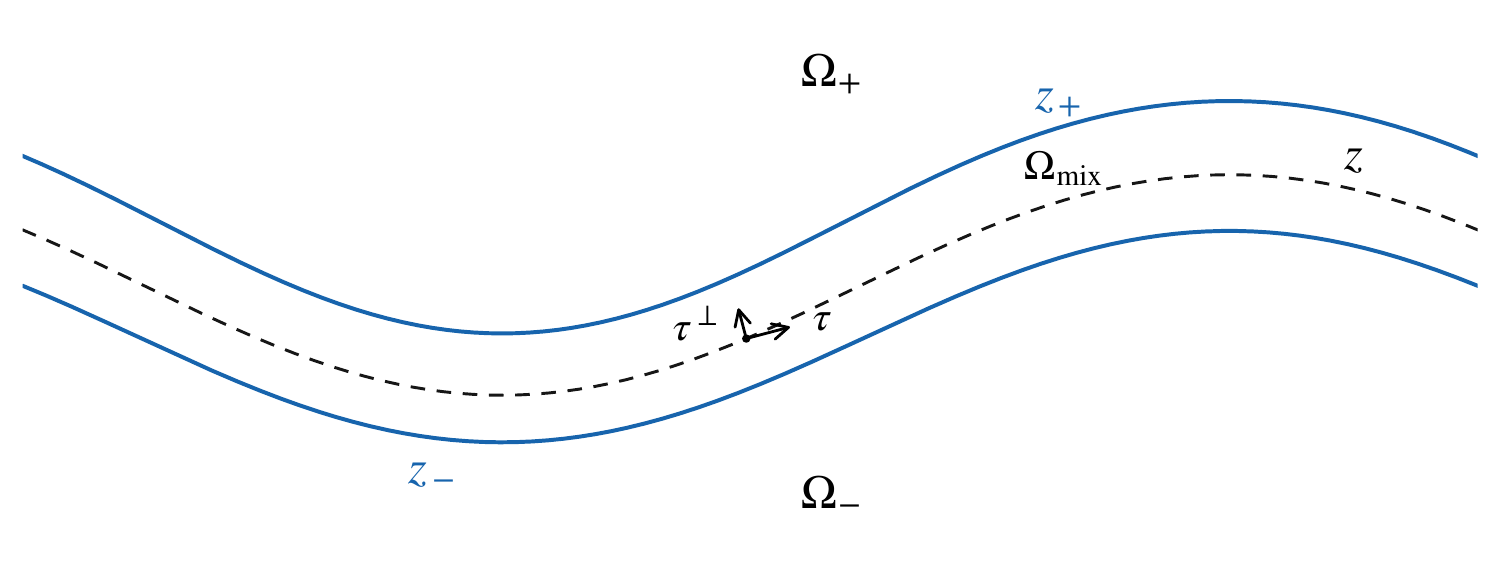}
    \caption{%
        Geometry of the mixing zone $\Omega_{\mathrm{mix}}$,
        with the same boundaries $z_\pm$ as in
        Figure~\ref{fig:mixing-fingers}.
        The central curve $z$ is dashed.
        The initial unit tangent $\tau$ and normal
        $\tau^\perp$ are displayed at the corresponding
        point of $z$.
        The exterior regions are denoted by
        $\Omega_+$ and $\Omega_-$.
    }
    \label{fig:mixing-geometry}
\end{figure}

The following geometric observation ensures that these curves delimit
a mixing region for sufficiently small positive times.
It follows by the same argument as \cite[Lemma~B.1]{CFM22}, with the affine
transverse parameter $h(\lambda)$ accounting for the asymmetric opening.

\begin{lemma}\label{lem:mixing-geometry}
Let $z$ and $c$ be as above. Then there exist $T>0$
and constants $d_0,d_1>0$ such that, for every $t\in(0,T]$,
$\alpha,\beta\in\mathbb R$, and $\lambda,\mu\in[-1,1]$,
\begin{equation}\label{eq:mixing-transversality}
    \partial_\alpha z_\lambda(t,\alpha)\cdot\tau(\alpha)
    \geq d_0,
\end{equation}
and
\begin{equation}\label{eq:mixing-chord-arc}
    |z_\lambda(t,\alpha)-z_\mu(t,\beta)|
    \geq d_1\bigl(|\alpha-\beta|+t|\lambda-\mu|\bigr).
\end{equation}
Consequently, $Z(t,\cdot,\cdot)$ is a $C^1$ diffeomorphism
from $\mathbb R\times(-1,1)$ onto $\Omega_{\mathrm{mix}}(t)$.
The curves $z_+(t)$ and $z_-(t)$ are disjoint chord--arc curves
and form the boundary of the mixing zone.
\end{lemma}

\subsection{The phase}\label{sec:phase}

We prescribe the macroscopic phase by
\begin{equation}\label{eq:phase}
    \overline\theta=
    \begin{cases}
        \pm1,&\text{in }\Omega_\pm,\\
        A,&\text{in }\Omega_{\mathrm{mix}}.
    \end{cases}
\end{equation}
With our convention for jumps,
$\llbracket\overline\theta\rrbracket_a=-(1-aA)$.

\subsection{The momentum}\label{sec:momentum}

We define
\begin{equation}\label{eq:momentum}
    \overline m
    =
    \overline\theta\,\overline u
    -\frac{1-\overline\theta^2}{1-A\overline\theta}
    \left(\nabla^\perp g+\frac12(A\overline u+i)\right),
\end{equation}
where $g$ is a scalar potential in the mixing region, to be determined.
We require $g\in C^1(\Omega_{\mathrm{mix}})$ with continuous boundary traces
\[
    (\nabla^\perp g)_a(t,\alpha)
    =
    \lim_{\Omega_{\mathrm{mix}}(t)\ni x\to z_a(t,\alpha)}
    \nabla^\perp g(t,x).
\]
In the pure-phase regions, \eqref{eq:momentum} is understood
as $\overline m=\pm\overline u$.
In the mixing region it reads
\[
    \overline m
    =A\overline u-\nabla^\perp g-\frac12(A\overline u+i).
\]

\subsection{The pressure}\label{sec:pressure}

Following the pressure formulation of Section~\ref{section3},
we prescribe the boundary densities
\begin{equation}\label{eq:boundary-densities}
\sigma_a
    =
    -(1-aA)(A\partial_tz_a+i)\cdot\partial_\alpha z_a^\perp.
\end{equation}
The pressure gradient is represented by
\begin{equation}\label{eq:pressure-gradient}
    (\nabla\overline p)^*(t,x)
    =
    \sum_{b=\pm1}\frac1{2\pi}
    \int_{\mathbb R}
    \frac{\sigma_b(t,\beta)}
    {x-z_b(t,\beta)}\,d\beta,
\end{equation}
where the principal value at infinity is understood.

By the Sokhotski--Plemelj formulas, the boundary values are
\begin{equation}\label{eq:pressure-traces}
    \nabla\overline p\longrightarrow
    \begin{cases}
        \displaystyle
        Q_a+\frac{ai\sigma_a}{2\partial_\alpha z_a^*},
        &\text{from }\Omega_a,\\[6pt]
        \displaystyle
        Q_a-\frac{ai\sigma_a}{2\partial_\alpha z_a^*},
        &\text{from }\Omega_{\mathrm{mix}},
    \end{cases}
\end{equation}
where
\[
    Q_a^*(t,\alpha)
    :=
    \sum_{b=\pm1}\frac1{2\pi}
    \int_{\mathbb R}
    \frac{\sigma_b(t,\beta)}
    {z_a(t,\alpha)-z_b(t,\beta)}\,d\beta.
\]

\subsection{The velocity}\label{sec:velocity}

We define $\overline u$ through the macroscopic Darcy law.
Substituting \eqref{eq:momentum}, this becomes
\begin{equation}\label{eq:velocity}
    (1+AC(\overline\theta))\overline u
    +iC(\overline\theta)
    -A\frac{1-\overline\theta^2}{1-A\overline\theta}\nabla^\perp g
    =-\nabla\overline p,
\end{equation}
where
\[
    C(\overline\theta)
    =
    \overline\theta
    -\frac{A(1-\overline\theta^2)}
    {2(1-A\overline\theta)}.
\]
Thus $C(\pm1)=\pm1$ and $C(A)=A/2$. Explicitly,
\[
    \overline u=
    \begin{cases}
        \displaystyle
        -\frac{\nabla\overline p\pm i}{1\pm A},
        &\text{in }\Omega_\pm,\\[8pt]
        \displaystyle
        \frac{A\nabla^\perp g-\nabla\overline p-\frac A2i}
        {1+\frac{A^2}{2}},
        &\text{in }\Omega_{\mathrm{mix}}.
    \end{cases}
\]
We denote by $\overline U_a$ the trace of $\overline u$
on $z_a$ from the pure-phase region $\Omega_a$.

Since $\overline p$ is harmonic within each region,
$\overline u$ is divergence-free there. Its normal continuity
across $z_\pm$, and hence global incompressibility, follows
from \eqref{eq:boundary-densities}.

\subsection{The potential $G$}\label{sec:potential}

In the coordinates introduced in Section~\ref{section4.1}, define
\[
    G(t,\alpha,\lambda):=g(Z(t,\alpha,\lambda)).
\]
The chain rule gives
\[
    \partial_\alpha G
    =(\nabla g\circ Z)\cdot\partial_\alpha z_\lambda,
    \qquad
    \partial_\lambda G
    =\frac{c(\alpha)t}{1-A^2}
    (\nabla g\circ Z)\cdot\tau^\perp.
\]
Hence
\begin{equation}\label{eq:gradient-G}
    \nabla g\circ Z
    =
    \frac1{\partial_\alpha z_\lambda\cdot\tau}
    \left(
        \partial_\alpha G\,\tau
        +\frac{1-A^2}{c(\alpha)t}
        \partial_\lambda G\,\partial_\alpha z_\lambda^\perp
    \right).
\end{equation}

Write $G_a(t,\alpha):=G(t,\alpha,a)$. Then
\[
    \partial_\alpha G_a
    =(\nabla^\perp g)_a\cdot\partial_\alpha z_a^\perp.
\]
The phase conservation equation requires
\[
    \bigl(
        \llbracket\overline\theta\rrbracket_a\partial_tz_a
        -\llbracket\overline m\rrbracket_a
    \bigr)\cdot\partial_\alpha z_a^\perp=0,
\]
where
\[
    \llbracket\overline m\rrbracket_a
    \cdot\partial_\alpha z_a^\perp
    =
    -\left(
        (1-aA)\overline U_a+\frac a2(A\overline U_a+i)
    \right)\cdot\partial_\alpha z_a^\perp
    -a\partial_\alpha G_a.
\]
Thus, we obtain the boundary conditions for $G$
\begin{equation}\label{eq:potential-boundary-condition}
    a\partial_\alpha G_a
    =
    \left(
        (1-aA)(\partial_tz_a-\overline U_a)
        -\frac a2(A\overline U_a+i)
    \right)\cdot\partial_\alpha z_a^\perp.
\end{equation}
Together with \eqref{eq:boundary-densities},
direct substitution into the velocity and momentum traces
shows that this prescription satisfies both normal-velocity
continuity and the phase conservation jump condition.

We normalize $G_a(t,0)=0$ and interpolate linearly
between the two boundary values:
\begin{equation}\label{eq:G-interpolation}
    G(t,\alpha,\lambda)
    =
    \sum_{a=\pm1}\frac{1+a\lambda}{2}
    \int_0^\alpha\partial_\alpha G_a(t,\beta)\,d\beta.
\end{equation}
This defines $g$ through $g\circ Z=G$.

\subsection{The hull condition}\label{sec:hullcondition}

In the pure phases, $\overline\theta=\pm1$ and $\overline m=\overline\theta\,\overline u$, so $(\overline\theta,\overline u,\overline m)\in K$. In the mixing region, $\overline\theta=A$ and \eqref{eq:momentum} gives
\[
2(1-A\overline\theta)(\overline m-\overline\theta\,\overline u)+(1-\overline\theta^2)(A\overline u+i)
=-2(1-A^2)\nabla^\perp g.
\]
Thus, the strict inequality defining $\mathcal O$ is equivalent to
\begin{equation}\label{eq:hull-condition}
2|\nabla g|<|A\overline u+i|\qquad\text{in }\Omega_{\mathrm{mix}}.
\end{equation}

\section{Contour dynamics}\label{section5}
Throughout Sections~\ref{section5}--\ref{section7}, until the completion
of the proof, $c$ is prescribed independently of $\sigma^{(0)}$ and satisfies
\begin{equation}\label{eq:c-assumptions}
    c-c_\infty\in H^{k-1}(\mathbb R).,
\end{equation}
for some $c_\infty>0$.
The particular opening speed will be chosen in
\eqref{eq:opening-speed}.

We seek the central curve in the form
\[
    z(t)=z^\circ+t z^{(1)}+\frac{t^2}{2}z^{(2)},
\]
and determine its coefficients through the boundary condition
\eqref{eq:potential-boundary-condition}.

\subsection{Asymptotics of the normal velocity traces}

The pure-phase Darcy law \eqref{eq:velocity} and the pressure trace formula \eqref{eq:pressure-traces} give
\begin{equation}\label{eq:normal-velocity-trace}
    (1+aA)\overline U_a\cdot\partial_\alpha z_a^\perp
    =
    -Q_a\cdot\partial_\alpha z_a^\perp
    -\frac a2\sigma_a
    -a\,i\cdot\partial_\alpha z_a^\perp.
\end{equation}

We compute the first two temporal coefficients of this expression.
The calculations below are performed for sufficiently regular curves
and densities; their justification and the uniform remainder estimates
are provided in Section~\ref{section6}.
The total Saffman-Taylor density
\[
    \sigma=\sigma_++\sigma_-
\]
satisfies
\begin{equation}\label{eq:sigma}
    \sigma
    =-2(A\partial_t z+i)\cdot\partial_\alpha z^\perp
    -t\frac{2A}{1-A^2}
    c\tau^\perp\cdot\partial_\alpha(c\tau^\perp)^\perp.
\end{equation}
From \eqref{eq:mixing-boundaries} we have
\[
    z_a^{(1)}
    =z^{(1)}+\frac{a}{1-aA}c\tau^\perp.
\]
We use the operators associated with the initial curve,
\begin{equation} \label{Q-P}
    Q_0^*[f](\alpha)
    =
    \frac1{2\pi}\operatorname{pv}
    \int_{\mathbb R}
    \frac{f(\beta)}
    {z^\circ(\alpha)-z^\circ(\beta)}\,d\beta,
    \qquad
    P[f]=2Q_0[f]\cdot\partial_\alpha z^{\circ\perp}.
\end{equation}

\begin{lemma}
\label{lem:normal-velocity-traces}
The initial normal traces and their first temporal derivatives are
\begin{equation}\label{eq:normal-velocity-coefficients}
\begin{aligned}
    \bigl(\overline U_a\cdot\partial_\alpha z_a^\perp\bigr)^{(0)}
    &=
    -\frac1{1+aA}
    \left(
        \frac12(P+aI)[\sigma^{(0)}]
        +a\,\partial_\alpha z^\circ_1
    \right),
    \\[4pt]
    \bigl(\overline U_a\cdot\partial_\alpha z_a^\perp\bigr)^{(1)}
    &=
    -\frac1{1+aA}
    \left(
        \frac12(P+aI)[\sigma^{(1)}]
        +r_a
    \right),
\end{aligned}
\end{equation}
where
\begin{equation}\label{eq:normal-velocity-remainder}
    r_a
    =
    R_a\cdot\partial_\alpha z^{\circ\perp}
    +
    \bigl(Q_0[\sigma^{(0)}]+ai\bigr)
    \cdot\partial_\alpha z_a^{(1)\perp},
\end{equation}
and
\begin{equation}\label{eq:pressure-trace-remainder}
\begin{aligned}
    R_a^*={}&
    z_a^{(1)}Q_0^*\!\left[
        \partial_\alpha\!\left(
            \frac{\sigma^{(0)}}{\partial_\alpha z^\circ}
        \right)\right]
    -\sum_{b=\pm1}Q_0^*\!\left[
        \partial_\alpha\!\left(
            \frac{z_b^{(1)}\sigma_b^{(0)}}
            {\partial_\alpha z^\circ}
        \right)\right]
-\frac{i}{(1-A^2)\partial_\alpha z^\circ}
    \partial_\alpha\!\left(
        \frac{c\tau^\perp\sigma_{-a}^{(0)}}
        {\partial_\alpha z^\circ}
    \right).
\end{aligned}
\end{equation}
In particular, $r_a$ depends only on $z^\circ$, $c$, $z^{(1)}$,
and the initial densities, and is independent of $\sigma^{(1)}$
and $z^{(2)}$.
\end{lemma}

\begin{proof}
Introduce the corrected pressure traces
\[
    \widetilde Q_a^*
    =
    Q_a^*+\frac{ai\sigma_{-a}}{2\partial_\alpha z_a}.
\]
Then \eqref{eq:normal-velocity-trace} becomes
\begin{equation}\label{eq:corrected-normal-velocity}
    (1+aA)\overline U_a\cdot\partial_\alpha z_a^\perp
    =
    -\bigl(\widetilde Q_a+ai\bigr)
    \cdot\partial_\alpha z_a^\perp
    -\frac a2\sigma.
\end{equation}
We claim that
\begin{equation}\label{eq:corrected-pressure-coefficients}
    \widetilde Q_a^{(0)}=Q_0[\sigma^{(0)}],
    \qquad
    \widetilde Q_a^{(1)}=Q_0[\sigma^{(1)}]+R_a.
\end{equation}

Indeed, writing
\[
    Q_a^*=\sum_{b=\pm1}Q_{a,b}^*[\sigma_b],
    \qquad
    Q_{a,b}^*[f](t,\alpha)
    =
    \frac1{2\pi}\operatorname{pv}
    \int_{\mathbb R}
    \frac{f(t,\beta)}
    {z_a(t,\alpha)-z_b(t,\beta)}\,d\beta,
\]
the boundary limits are
\begin{equation}\label{eq:pressure-operator-initial-limit}
    Q_{a,b}^*[f]\longrightarrow
    Q_0^*[f^{(0)}]
    -\delta_{b,-a}
    \frac{ai f^{(0)}}{2\partial_\alpha z^\circ}.
\end{equation}
The correction defining $\widetilde Q_a$ cancels the local term,
proving the first identity in
\eqref{eq:corrected-pressure-coefficients}.

For the second identity, differentiation and integration by parts yield
\begin{equation}\label{eq:pressure-operator-time-derivative}
\begin{aligned}
    \partial_t\bigl(Q_{a,b}^*[f]\bigr)
    =Q_{a,b}^*[\partial_tf]
    +\partial_tz_a\,Q_{a,b}^*\!\left[
        \partial_\alpha\!\left(
            \frac f{\partial_\alpha z_b}
        \right)\right]
    -Q_{a,b}^*\!\left[
        \partial_\alpha\!\left(
            \frac{\partial_tz_b\,f}{\partial_\alpha z_b}
        \right)\right].
\end{aligned}
\end{equation}
Applying \eqref{eq:pressure-operator-initial-limit}
and differentiating the correction gives
\[
\begin{aligned}
    \bigl(\widetilde Q_a^*\bigr)^{(1)}
    ={}&Q_0^*[\sigma^{(1)}]
    +z_a^{(1)}Q_0^*\!\left[
        \partial_\alpha\!\left(
            \frac{\sigma^{(0)}}{\partial_\alpha z^\circ}
        \right)\right]
    -\sum_{b=\pm1}Q_0^*\!\left[
        \partial_\alpha\!\left(
            \frac{z_b^{(1)}\sigma_b^{(0)}}
            {\partial_\alpha z^\circ}
        \right)\right]
    \\
    &-\frac{ai}{2\partial_\alpha z^\circ}
    \partial_\alpha\!\left(
        \frac{(z_a^{(1)}-z_{-a}^{(1)})\sigma_{-a}^{(0)}}
        {\partial_\alpha z^\circ}
    \right).
\end{aligned}
\]
Using
\[
    z_a^{(1)}-z_{-a}^{(1)}
    =\frac{2a}{1-A^2}c\tau^\perp
\]
proves the second identity in
\eqref{eq:corrected-pressure-coefficients}.

Finally, evaluating \eqref{eq:corrected-normal-velocity}
and its temporal derivative at $t=0$, and using
\eqref{eq:corrected-pressure-coefficients} together with the
definition of $P$, gives
\eqref{eq:normal-velocity-coefficients}--%
\eqref{eq:normal-velocity-remainder}.
\end{proof}

The unknown coefficient $\sigma^{(1)}$ appears only through $P+aI$.
The remainder $r_a$ is expressed entirely through $Q_0$,
multiplication, and differentiation of quantities determined
at the preceding order.

\subsection{First order dynamics}

We choose the initial Saffman-Taylor density so that the two normal velocity traces
in Lemma~\ref{lem:normal-velocity-traces} have the same limit,
independently of $a$. We then prescribe the initial velocity of the
central curve to agree with this common normal trace.

\begin{lemma}
\label{lem:first-order-compatibility}
Let $\sigma^{(0)}$ solve
\begin{equation}\label{eq:first-order-density}
    (I-AP)\sigma^{(0)}
    =-2\,\partial_\alpha z^\circ_1,
\end{equation}
and set
\begin{equation}\label{eq:first-order-curve}
    z^{(1)}=-Q_0[\sigma^{(0)}].
\end{equation}
For any prescribed opening speed $c$, the initial boundary values satisfy
\begin{equation}\label{eq:first-order-boundary-compatibility}
    \partial_\alpha G_a^{(0)}
    =c|\partial_\alpha z^\circ|+\frac14\sigma^{(0)},
    \qquad
    \sigma_a^{(0)}=\frac{1-aA}{2}\sigma^{(0)}
    -aAc|\partial_\alpha z^\circ|.
\end{equation}
In particular, $\partial_\alpha G_+^{(0)}=\partial_\alpha G_-^{(0)}$.
\end{lemma}

\begin{proof}
Equation \eqref{eq:sigma} at $t=0$ gives
\[
    \sigma^{(0)}
    =-2(Az^{(1)}+i)\cdot\partial_\alpha z^{\circ\perp},
\]
which is the initial total density associated with $z^{(1)}$.
Using \eqref{eq:first-order-curve}, we obtain
\eqref{eq:first-order-density}.
Substituting \eqref{eq:first-order-density} into
Lemma~\ref{lem:normal-velocity-traces} gives
\[
    \bigl(\overline U_a\cdot\partial_\alpha z_a^\perp\bigr)^{(0)}
    =-\frac12P[\sigma^{(0)}]
    =z^{(1)}\cdot\partial_\alpha z^{\circ\perp}.
\]
Evaluating \eqref{eq:potential-boundary-condition} at $t=0$,
we therefore obtain
\[
    \partial_\alpha G_a^{(0)}
    =c|\partial_\alpha z^\circ|+\frac14\sigma^{(0)},
\]
independently of $a$. The density prescription gives
\[
    \sigma_a^{(0)}
    =\frac{1-aA}{2}\sigma^{(0)}
    -aAc|\partial_\alpha z^\circ|,
\]
proving the remaining identity.
\end{proof}

The solvability of \eqref{eq:first-order-density}, with
$\sigma^{(0)}+2\in H^{k-1}$, is established in
Section~\ref{section6}.

\subsection{Second order dynamics}

The interpolation defining $G$ gives
\[
    \partial_\lambda G(t,\alpha,\lambda)
    =
    \frac12\int_0^\alpha
    \bigl(\partial_\alpha G_+-\partial_\alpha G_-\bigr)
    (t,\beta)\,d\beta.
\]
The equality of the two initial boundary derivatives in
Lemma~\ref{lem:first-order-compatibility} cancels the zeroth-order
coefficient of the integrand. We now choose $z^{(2)}$ to cancel
its first-order coefficient as well.

\begin{lemma}
\label{lem:second-order-compatibility}
Let $\sigma^{(1)}$ solve
\begin{equation}\label{eq:second-order-density}
    (I-AP)\sigma^{(1)}
    =
    \frac{2(1-A^2)}{2+A^2}
    \bigl(r_\sigma+A r_G\bigr).
\end{equation}
Prescribe $z^{(2)}$ as a curve normal to $z^\circ$ such that
\begin{equation}\label{eq:second-order-curve}
    z^{(2)}\cdot\partial_\alpha z^{\circ\perp}
    =
    -\frac12P[\sigma^{(1)}]
    +\frac{3A r_\sigma-2(1-A^2)r_G}{2(2+A^2)},
\end{equation}
where the remainders are
\begin{equation}\label{eq:density-remainder}
    r_\sigma
    :=
    -2(Az^{(1)}+i)\cdot\partial_\alpha z^{(1)\perp}
    -\frac{2A}{1-A^2}
    c\tau^\perp\cdot\partial_\alpha(c\tau^\perp)^\perp
\end{equation}
and
\begin{equation}\label{eq:potential-remainder}
    r_G
    :=
    \sum_{a=\pm1}
    \left[
        \left((1-aA)z_a^{(1)}-\frac a2i\right)
        \cdot\partial_\alpha z_a^{(1)\perp}
        +\frac{2-aA}{2(1+aA)}r_a
    \right],
\end{equation}
with $r_a$ given by Lemma~\ref{lem:normal-velocity-traces}.
Both remainders depend only on quantities determined at first order.

Then, it holds that
\begin{equation}\label{eq:second-order-boundary-compatibility}
    \partial_\alpha G_+^{(1)}
    =\partial_\alpha G_-^{(1)},
\end{equation}
and
\begin{equation}\label{eq:boundary-density-derivatives}
\begin{aligned}
    \sigma_a^{(1)}
    ={}&-(1-aA)A z^{(2)}\cdot\partial_\alpha z^{\circ\perp}\\
    &-\bigl((1-aA)(Az^{(1)}+i)+aAc\tau^\perp\bigr)
    \cdot\partial_\alpha z_a^{(1)\perp}.
\end{aligned}
\end{equation}
\end{lemma}
Note that the subscripts in $r_\sigma$ and $r_G$ indicate the equations from which the
remainders arise: $r_\sigma$ collects the terms in the first temporal derivative of the
total density \eqref{eq:sigma} that are already known at first order (see \eqref{eq:total-density-derivative} below), while $r_G$ collects the corresponding terms in $\partial_t(\partial_\alpha G_+-\partial_\alpha G_-)(0)$. Neither contains the unknown
second-order coefficients $\sigma^{(1)}$ and $z^{(2)}$.

\begin{proof}
Differentiating the total density \eqref{eq:sigma} at $t=0$
gives the compatibility relation
\begin{equation}\label{eq:total-density-derivative}
    \sigma^{(1)}
    =-2A z^{(2)}\cdot\partial_\alpha z^{\circ\perp}
    +r_\sigma.
\end{equation}
Next, differentiating \eqref{eq:potential-boundary-condition},
summing over $a=\pm1$, and applying
Lemma~\ref{lem:normal-velocity-traces} yields
\[
\begin{aligned}
    \sum_{a=\pm1}a\partial_\alpha G_a^{(1)}
    =2z^{(2)}\cdot\partial_\alpha z^{\circ\perp}
    +\frac{2+A^2}{2(1-A^2)}P[\sigma^{(1)}]
    -\frac{3A}{2(1-A^2)}\sigma^{(1)}
    +r_G.
\end{aligned}
\]
Using \eqref{eq:total-density-derivative}, this becomes
\[
\begin{aligned}
    \sum_{a=\pm1}a\partial_\alpha G_a^{(1)}
    =
    \frac{2+A^2}{1-A^2}
    \left(
        z^{(2)}\cdot\partial_\alpha z^{\circ\perp}
        +\frac12P[\sigma^{(1)}]
    \right)
    +r_G-\frac{3A}{2(1-A^2)}r_\sigma.
\end{aligned}
\]
The right-hand side vanishes under \eqref{eq:second-order-curve}.
Substituting \eqref{eq:second-order-curve} into
\eqref{eq:total-density-derivative} gives precisely
\eqref{eq:second-order-density}, so both compatibility conditions hold.
The formula for $\sigma_a^{(1)}$ follows by differentiating
\eqref{eq:boundary-densities} at $t=0$.
\end{proof}

The first- and second-order equations involve the same operator $I-AP$.
Their solvability, the regularity of the resulting coefficients,
and the temporal remainder estimates are established in
Section~\ref{section6}.

The regularity estimates follow the chain
\[
    R_a\in H^{k-2}
    \quad\Longrightarrow\quad
    r_a,\ r_\sigma,\ r_G\in H^{k-2}
    \quad\Longrightarrow\quad
    \sigma^{(1)},\ z^{(2)}\in H^{k-2}.
\]
Indeed, once the first-order quantities have been controlled,
the bounds for $R_a$ follow from the boundedness of $Q_0$
and Sobolev product estimates. The same product estimates
control the scalar remainders, while the invertibility of $I-AP$
and the prescription \eqref{eq:second-order-curve}, which fixes
$z^{(2)}$ as a normal vector field, yield the regularity of
$\sigma^{(1)}$ and $z^{(2)}$, respectively.

\section{Inverting the Saffman–Taylor equation}\label{section6}
The equations in Section~\ref{section5} determine the temporal
coefficients by inverting twice with the same operator $I-AP$.
In this section we carry out these inversions and establish the
regularity of the resulting densities and velocity coefficients
for the prescribed opening speed $c$ satisfying \eqref{eq:c-assumptions}.
The analytic proof of invertibility is deferred to
Section~\ref{section8}.

As in Theorem~\ref{thm:main}, let $z^\circ$ be a chord--arc curve with
$z^\circ-(\alpha,0)\in H^k(\mathbb R;\mathbb R^2)$, $k\ge5$,
and fix $A\in(-1,1)$.
We retain the assumptions on $c$ from Section~\ref{section5}
and the operators $Q_0$ and $P$ introduced in \eqref{Q-P}.
We keep these regularity assumptions throughout
Sections~\ref{section6}--\ref{section8}.
The operator result needed here is the following.

\begin{proposition}
\label{p:inverse}
For $|A|<1$ and every $0\le j\le k-1$, the operators $Q_0$ and $P$
are bounded on $H^j$, and $I-AP$ is invertible on
$H^j(\mathbb R;\mathbb R)$, with bounded inverse.
Moreover,
\[
    Q_0[1],\ P[1]\in H^{k-1}.
\]
\end{proposition}
The proof is given in Section~\ref{section8.3}. The last assertion is needed because $\sigma^{(0)}$ does not decay at infinity, it allows
us to work with $\sigma^{(0)}+2\in H^{k-1}$, see the proof of Proposition~\ref{p:first}.
\begin{lemma}
\label{l:products}
Let $m\ge1$ be an integer and
$f-f_\infty\in H^m(\mathbb R)$, $f_\infty\in\mathbb C$.
Multiplication by $h\in H^j$ is bounded on $H^j$, $0\le j\le m$, with
\[
    \|fh\|_{H^j}
    \le C\bigl(|f_\infty|+\|f-f_\infty\|_{H^m}\bigr)
    \|h\|_{H^j}.
\]
Moreover, if $g-g_\infty\in H^m$, $g_\infty\in\mathbb C$, then
$fg-f_\infty g_\infty\in H^m$.

For a smooth map $F$ defined near the closure of the range of $f$,
including $f_\infty$, one has
$F(f)-F(f_\infty)\in H^m$.
In particular, if $f-1\in H^m$ and $\inf|f|>0$,
\[
    f^{-1}-1,\quad |f|^{-1}-1,\quad |f|^{-2}-1,\quad
    f/|f|-1\in H^m.
\]
These operations preserve $C^r$ time dependence.
Their bounds are uniform on bounded families, with uniform bounds
on the derivatives of $F$ near their ranges and a common positive
lower bound for denominators.
\end{lemma}

\begin{proof}
The embedding $H^1(\mathbb R)\hookrightarrow L^\infty$ and Leibniz's
rule prove the product estimate at integer orders: derivatives of $f$
of order at most $m-1$ are bounded, while the possible term
$(\partial_\alpha^m f)h$ is controlled by
$\|\partial_\alpha^m f\|_2\|h\|_\infty$.
Interpolation gives intermediate orders.
Expanding
\[
    fg-f_\infty g_\infty
    =f(g-g_\infty)+g_\infty(f-f_\infty)
\]
gives the second assertion.

For compositions, smoothness controls the zeroth-order difference
by $|f-f_\infty|$; the chain rule expresses each derivative as products
of bounded coefficients and derivatives of $f$.
Put one derivative factor in $L^2$ and the others in $L^\infty$,
treating the top-order term as $F'(f)\partial_\alpha^m f$.
The reciprocal assertions follow by choosing the indicated smooth
maps away from zero.
Applying the same product and chain rules to time derivatives
proves the last assertion.
\end{proof}

\subsection{First-order coefficients}

We first solve the initial Saffman--Taylor equation and check the
regularity of the choices made in
Lemma~\ref{lem:first-order-compatibility}.

\begin{proposition}
\label{p:first}
Under the standing assumptions,
equation \eqref{eq:first-order-density} in
Lemma~\ref{lem:first-order-compatibility} has a unique solution
with $\sigma^{(0)}+2\in H^{k-1}$.
The coefficient $z^{(1)}$ defined by \eqref{eq:first-order-curve}
belongs to $H^{k-1}$.
Moreover, the first-order coefficients satisfy
\[
    z_a^{(1)}-\frac{aic_\infty}{1-aA}\in H^{k-1},
    \qquad
    \sigma_a^{(0)}+1-aA(1-c_\infty)\in H^{k-1}.
\]
\end{proposition}

\begin{proof}
Rewrite \eqref{eq:first-order-density} as
\[
    (I-AP)(\sigma^{(0)}+2)
    =-2(\partial_\alpha z_1^\circ-1)-2AP[1].
\]
Since $z^\circ-(\alpha,0)\in H^k$, the right-hand side belongs to $H^{k-1}$ by Proposition~\ref{p:inverse}, and the same proposition gives
$\sigma^{(0)}+2\in H^{k-1}$.
Likewise,
\[
    z^{(1)}
    =-Q_0[\sigma^{(0)}+2]+2Q_0[1]\in H^{k-1}.
\]
By the assumptions on $c$ and Lemma~\ref{l:products},
\[
    \tau-1\in H^{k-1},
    \qquad
    c\tau-c_\infty\in H^{k-1}.
\]
The remaining assertions follow from
\eqref{eq:first-order-boundary-compatibility} and
$z_a^{(1)}=z^{(1)}+ac\tau^\perp/(1-aA)$.
\end{proof}

\subsection{Second-order coefficients}

Once the first-order coefficients are fixed, all the remainders
in the second-order equations are known.
We first check their regularity and then solve
\eqref{eq:second-order-density} with the same operator $I-AP$
as in Proposition~\ref{p:first}, now in $H^{k-2}$.

\begin{proposition}
\label{p:second}
Under the standing assumptions and those of
Proposition~\ref{p:first},
equation \eqref{eq:second-order-density} in
Lemma~\ref{lem:second-order-compatibility} has a unique solution
$\sigma^{(1)}\in H^{k-2}$.
The normal coefficient $z^{(2)}$ prescribed by
\eqref{eq:second-order-curve} and the boundary density derivatives
$\sigma_a^{(1)}$ in \eqref{eq:boundary-density-derivatives}
also belong to $H^{k-2}$.
\end{proposition}

\begin{proof}
By Proposition~\ref{p:first} and Lemma~\ref{l:products},
the quotients
\[
    \frac{\sigma^{(0)}}{\partial_\alpha z^\circ},
    \qquad
    \frac{z_b^{(1)}\sigma_b^{(0)}}{\partial_\alpha z^\circ},
    \qquad
    \frac{c\tau^\perp\sigma_{-a}^{(0)}}{\partial_\alpha z^\circ}
\]
differ from constants by $H^{k-1}$ functions.
Their derivatives belong to $H^{k-2}$.
The formula for $R_a$ in
Lemma~\ref{lem:normal-velocity-traces} involves only these
derivatives, $Q_0^*$, and multiplication by functions that are
constant plus $H^{k-1}$.
Proposition~\ref{p:inverse} and Lemma~\ref{l:products} therefore give
$R_a\in H^{k-2}$.
The definitions \eqref{eq:normal-velocity-remainder}, \eqref{eq:density-remainder} and \eqref{eq:potential-remainder} of $r_a$, $r_\sigma$ and $r_G$ involve only $R_a$, first spatial derivatives of the first-order coefficients, which belong to $H^{k-2}$, and
multiplication by functions that are constant plus $H^{k-1}$. Lemma~\ref{l:products} then gives $r_a,r_\sigma,r_G\in H^{k-2}$.

Proposition~\ref{p:inverse} therefore solves
\eqref{eq:second-order-density} in $H^{k-2}$.
The right-hand side of \eqref{eq:second-order-curve}
belongs to $H^{k-2}$.
Choosing
\[
    z^{(2)}
    =
    \frac{z^{(2)}\cdot(\partial_\alpha z^\circ)^\perp}
    {|\partial_\alpha z^\circ|^2}
    (\partial_\alpha z^\circ)^\perp
\]
and applying Lemma~\ref{l:products} gives
$z^{(2)}\in H^{k-2}$.
The formula \eqref{eq:boundary-density-derivatives} for $\sigma_a^{(1)}$ in Lemma~\ref{lem:second-order-compatibility} consists of products
of these quantities and first spatial derivatives of the
first-order coefficients, so Lemma~\ref{l:products} gives
$\sigma_a^{(1)}\in H^{k-2}$.
\end{proof}

\section{Proof of Theorem~\ref{thm:main}}\label{section7}
In this section, we complete the proof of
Theorem~\ref{thm:main}. We first estimate the gradient of
the potential $g$ and establish the nondegeneracy of
$A\bar u+i$. These estimates allow us to choose the opening
rate so that the hull condition holds and apply the
$h$-principle. Finally, we extend the construction to
several layers to obtain larger opening rates.

Let the coefficients be those constructed in Section~\ref{section6},
under the hypotheses of Theorem~\ref{thm:main}, with $c$ still free
as in Section~\ref{section5}, and set
\begin{equation}\label{e:curves}
    z(t)=z^\circ+tz^{(1)}+\tfrac12t^2z^{(2)},
    \qquad
    z_a(t)=z(t)+\frac{a}{1-aA}tc\tau^\perp.
\end{equation}
The pressure densities \eqref{eq:boundary-densities} are
\[
    \sigma_a
    =-(1-aA)(A\partial_t z_a+i)\cdot
    (\partial_\alpha z_a)^\perp,
    \qquad
    \sigma=\sigma_++\sigma_-.
\]
We want to bound the derivatives
\begin{subequations}\label{e:Gderivs}
\begin{align}
    \partial_\alpha G(t,\alpha,\lambda)
    &=
    \sum_{a=\pm1}\frac{1+a\lambda}{2}
    \partial_\alpha G_a(t,\alpha),
    \label{e:Galpha}\\
    \partial_\lambda G(t,\alpha,\lambda)
    &=
    \frac12\int_0^\alpha
    (\partial_\alpha G_+-\partial_\alpha G_-)(t,\beta)
    \,d\beta,
    \label{e:Glambda}
\end{align}
\end{subequations}
establish a positive lower bound
for $|Au+i|$, and verify the subsolution properties.
All time derivatives at zero are right derivatives.
Throughout this section, $O(t^j)$ denotes a remainder bounded by
$Ct^j$ uniformly in the spatial and transverse variables.

\subsection{Estimate for
\texorpdfstring{$\partial_\alpha G$}{the tangential derivative of G}}
\label{section7.1}

\begin{proposition}
\label{p:tangent}
For sufficiently small positive times,
\begin{equation}\label{e:tangent}
    \partial_\alpha G
    =c|\partial_\alpha z^\circ|+\frac14\sigma^{(0)}+O(t).
\end{equation}
\end{proposition}

\begin{proof}
In view of \eqref{e:Galpha}, it is enough to prove that
$\|\partial_\alpha G_a(t)-\partial_\alpha G_a(0)\|_\infty\le Ct$
for $a=\pm1$.
Propositions~\ref{p:first} and~\ref{p:second}, together with \eqref{e:curves}, give
\[
    z_a(t)-\mathrm{id}-\frac{aic_\infty t}{1-aA}
    \in C^\infty([0,T];H^{k-2}).
\]
For small $T$, Lemma~\ref{lem:mixing-geometry} applies.
Since $z_a$ is a quadratic polynomial in $t$ by \eqref{e:curves}, the densities \eqref{eq:boundary-densities} are cubic polynomials in time, and
Lemma~\ref{l:products} gives
\begin{equation}\label{e:densityreg}
    \sigma_b+1-bA(1-c_\infty)
    \in C^\infty([0,T];H^{k-3}).
\end{equation}
Thus $k\ge5$ supplies the hypotheses of
Lemma~\ref{l:initialderivative}:
$z-\mathrm{id}\in C_t^2H^3$,
$c\tau^\perp-ic_\infty\in H^3$, and
$\sigma_b+1-bA(1-c_\infty)\in C_t^2H^2$,
with time-independent constant part $-1+bA(1-c_\infty)$.

Applying Lemma~\ref{l:initialderivative} with $f=\sigma_b$
shows that $Q_{a,b}^*[\sigma_b]$, extended to $t=0$ by its
right limit \eqref{eq:pressure-operator-initial-limit},
belongs to $C^1([0,T];L^\infty)$.
The pure-phase Darcy trace \eqref{eq:velocity}, together with the pressure traces \eqref{eq:pressure-traces}, gives
\[
    (1+aA)\overline U_a
    =
    -\sum_{b=\pm1}Q_{a,b}[\sigma_b]
    -\frac{ai\sigma_a}{2(\partial_\alpha z_a)^*}-ai.
\]
By Lemma~\ref{l:products}, the remaining factors in this formula
and in \eqref{eq:potential-boundary-condition} are $C^1$ in time
with values in $L^\infty$. For the factor $1/(\partial_\alpha z_a)^*$ we use the uniform lower bound $|\partial_\alpha z_a|\ge d_0$ from \eqref{eq:mixing-transversality}.
Consequently,
\[
    \partial_\alpha G_a
    \in C^1([0,T];L^\infty(\mathbb R)).
\]
Since Lemma~\ref{lem:first-order-compatibility} gives
\[
    \partial_\alpha G_a(0)
    =c|\partial_\alpha z^\circ|+\frac14\sigma^{(0)},
\]
integration in time proves the estimate.
\end{proof}

\subsection{Estimate for
\texorpdfstring{$\partial_\lambda G$}{the transverse derivative of G}}
\label{section7.2}

\begin{proposition}
\label{p:transverse}
For sufficiently small positive times,
\begin{equation}\label{e:transverse}
    \partial_\lambda G=O(t^2).
\end{equation}
\end{proposition}

\begin{proof}
In view of \eqref{e:Glambda}, it is enough to check that
\begin{equation}\label{e:Dprimitive}
    \sup_{\alpha\in\mathbb R}
    \left|\int_0^\alpha D(t,\beta)\,d\beta\right|
    \le Ct^2,
    \qquad
    D=\partial_\alpha G_+-\partial_\alpha G_-.
\end{equation}
The compatibility conditions
\eqref{eq:first-order-boundary-compatibility} and
\eqref{eq:second-order-boundary-compatibility} give
$D(0)=\partial_tD(0)=0$, with right limits justified by
\eqref{e:densityreg} and Lemma~\ref{l:initialderivative}.
We claim that
\begin{equation}\label{e:Dtprimitive}
    \sup_{\alpha\in\mathbb R}
    \left|\int_0^\alpha\partial_tD(t,\beta)\,d\beta\right|
    \le Ct.
\end{equation}
Since $D(0)=0$, this implies
\[
    \sup_{\alpha\in\mathbb R}
    \left|\int_0^\alpha D(t,\beta)\,d\beta\right|
    \le
    \int_0^t\sup_{\alpha\in\mathbb R}
    \left|\int_0^\alpha\partial_sD(s,\beta)\,d\beta\right|\,ds
    \le\frac C2t^2,
\]
which proves \eqref{e:Dprimitive} and hence
\eqref{e:transverse}.

It remains to establish \eqref{e:Dtprimitive}.
Let $T_a=\widetilde Q_a\cdot(\partial_\alpha z_a)^\perp$ be as in
\eqref{e:correctedtraces} with $f_b=\sigma_b$, so that
$\widetilde Q_a$ is the corrected pressure trace of
Lemma~\ref{lem:normal-velocity-traces}.
Substituting the corrected trace identity
\eqref{eq:corrected-normal-velocity} and the density prescription
\eqref{eq:sigma} into \eqref{eq:potential-boundary-condition},
and using $z_a=z+atc\tau^\perp/(1-aA)$, gives
\[
    D(t)
    =
    \frac{2+A^2}{1-A^2}\,
    \partial_tz\cdot(\partial_\alpha z)^\perp
    +\sum_{a=\pm1}\frac{2-aA}{2(1+aA)}T_a(t)+tR,
\]
where
\[
    R=
    \frac{
        \bigl((2+A^2)c\tau^\perp+(1-A^2)i\bigr)
        \cdot\partial_\alpha(c\tau^\perp)^\perp
    }{(1-A^2)^2}
\]
depends only on $A$ and $c\tau^\perp$ and is independent of time.
Differentiating and subtracting the value at zero eliminates $R$.
Since $\partial_tD(0)=0$, this yields
\begin{equation}\label{e:Didentity}
\begin{aligned}
    \partial_tD(t)
    ={}&
    \frac{2+A^2}{1-A^2}
    \Bigl(
        \partial_t\bigl(
            \partial_tz\cdot(\partial_\alpha z)^\perp
        \bigr)(t)
        -
        \partial_t\bigl(
            \partial_tz\cdot(\partial_\alpha z)^\perp
        \bigr)(0)
    \Bigr)
    \\
    &+\sum_{a=\pm1}\frac{2-aA}{2(1+aA)}
    \bigl(\partial_tT_a(t)-\partial_tT_a(0)\bigr).
\end{aligned}
\end{equation}

In the last term of the right-hand side, we apply
Lemma~\ref{l:temporalprimitive} with $f_b=\sigma_b$.
The requirements follow from \eqref{e:densityreg} and the curve
bounds verified in the proof of Proposition~\ref{p:tangent}.
Hence \eqref{e:temporalprimitive} bounds the spatial primitive
of each term of the sum in \eqref{e:Didentity} by $Ct$.

For the geometric term, the quadratic form of $z$ gives
\[
\begin{aligned}
    &\partial_t\bigl(
        \partial_tz\cdot(\partial_\alpha z)^\perp
    \bigr)(t)
    -
    \partial_t\bigl(
        \partial_tz\cdot(\partial_\alpha z)^\perp
    \bigr)(0)
    \\
    &\qquad
    =
    t\bigl(
        2z^{(2)}\cdot(\partial_\alpha z^{(1)})^\perp
        +z^{(1)}\cdot(\partial_\alpha z^{(2)})^\perp
    \bigr)
    +\tfrac32t^2z^{(2)}
    \cdot(\partial_\alpha z^{(2)})^\perp.
\end{aligned}
\]
All products belong to $L^1$ by Cauchy--Schwarz.
Their primitives are therefore $O(t)$ uniformly in $\alpha$,
proving \eqref{e:Dtprimitive}.
\end{proof}

\begin{proposition}
\label{p:gradient}
For sufficiently small positive times,
\begin{equation}\label{e:gradient}
    \nabla g\circ Z
    =
    \left(
        c+\frac{\sigma^{(0)}}{4|\partial_\alpha z^\circ|}
    \right)\tau+O(t).
\end{equation}
\end{proposition}

\begin{proof}
Formula \eqref{eq:gradient-G} reads
\[
    \nabla g\circ Z
    =
    \frac{1}{\partial_\alpha z_\lambda\cdot\tau}
    \left(
        \partial_\alpha G\,\tau
        +\frac{1-A^2}{ct}\partial_\lambda G\,
        (\partial_\alpha z_\lambda)^\perp
    \right).
\]
The denominator is $|\partial_\alpha z^\circ|+O(t)$,
uniformly bounded away from zero.
Propositions~\ref{p:tangent} and~\ref{p:transverse},
together with $\inf c>0$, give the stated expansion.
\end{proof}

\subsection{Nondegeneracy of \texorpdfstring{$Au+i$}{Au+i}}
\label{section7.3}

\begin{proposition}
\label{p:nondegenerate}
For sufficiently small positive times,
\begin{equation}\label{e:nondegenerate}
    (Au+i)(t,Z)\cdot\tau^\perp
    =
    -\frac{\sigma^{(0)}}{2|\partial_\alpha z^\circ|}+O(t).
\end{equation}
\end{proposition}

\begin{proof}
The initial densities in
Lemma~\ref{lem:first-order-compatibility} satisfy
\[
    \sigma_-^{(0)}-\sigma_+^{(0)}
    =A\sigma^{(0)}+2Ac|\partial_\alpha z^\circ|.
\]

Since $(\nabla\overline p)^*(t,Z)=\sum_{b=\pm1}Q^*_{\lambda,b}[\sigma_b]$ by \eqref{eq:pressure-gradient}, and $\operatorname{sgn}(\lambda-b)=-b$ for $-1<\lambda<1$, estimate \eqref{e:pressurecomparison} applied to $f=\sigma_b$ (admissible by \eqref{e:densityreg}), together with 
the identity above, yields

\begin{equation}\label{e:pressurelimit}
    \nabla p(t,Z)
    =
    Q_0[\sigma^{(0)}]
    +A\left(
        c+\frac{\sigma^{(0)}}{2|\partial_\alpha z^\circ|}
    \right)\tau^\perp+O(t).
\end{equation}
The mixing-region Darcy law \eqref{eq:velocity}, with $\overline\theta=A$, gives
\begin{equation}\label{e:projectedDarcy}
    \left(1+\frac{A^2}{2}\right)(Au+i)\cdot\tau^\perp
    =
    A^2\nabla g\cdot\tau
    +(i-A\nabla p)\cdot\tau^\perp.
\end{equation}
Substituting Proposition~\ref{p:gradient} and \eqref{e:pressurelimit}
into the right-hand side of \eqref{e:projectedDarcy}, the terms
containing $c$ cancel and we obtain
\[
    \left(1+\frac{A^2}{2}\right)(A\bar u+i)(t,Z)\cdot\tau^\perp
    =
    (i-AQ_0[\sigma^{(0)}])\cdot\tau^\perp
    -\frac{A^2\sigma^{(0)}}{4|\partial_\alpha z^\circ|}+O(t).
\]
Finally, by \eqref{eq:first-order-density} and the definition
\eqref{Q-P} of $P$,
\[
    (i-AQ_0[\sigma^{(0)}])\cdot\tau^\perp
    =
    -\frac{\sigma^{(0)}}{2|\partial_\alpha z^\circ|},
\]
and the claim follows after division by $1+A^2/2$.
\end{proof}

\subsection{Completion of the proof of Theorem \ref{thm:main}}\label{section7.4}

We now choose
\begin{equation}\label{eq:opening-speed}
    c=-\bar c\,\frac{\sigma^{(0)}}{|\partial_\alpha z^\circ|},
    \qquad 0<\bar c<\frac12.
\end{equation}
Uniform instability and Proposition~\ref{p:first} give
$\inf c>0$ and $c-2\bar c\in H^{k-1}$, so all preceding
estimates apply with $c_\infty=2\bar c$.

Write $(\bar\theta,\bar u,\bar m)$ for the fields defined by
\eqref{eq:phase}, \eqref{eq:velocity} and \eqref{eq:momentum},
with $g\circ Z=G$.
We verify that they form a subsolution in the sense of
Definition~\ref{definition_subsolution}.

By Section~\ref{section4}, the prescriptions
\eqref{eq:boundary-densities} and
\eqref{eq:potential-boundary-condition} ensure continuity
of the normal velocity and the phase-conservation jump conditions.
Together with the regional identities, these give the linear
\eqref{mIPM} system in distributions.
In the pure phases, $\bar\theta=\pm1$ and
$\bar m=\bar\theta\bar u$ by construction.
By Section~\ref{sec:hullcondition},
the strict hull condition in the mixing zone is equivalent to
\[
    2|\nabla g|<|A\bar u+i|.
\]
Propositions~\ref{p:gradient} and~\ref{p:nondegenerate}
imply, pointwise,
\[
    |A\bar u+i|-2|\nabla g|
    \ge
    -(1-|4\bar c-1|)
    \frac{\sigma^{(0)}}{2|\partial_\alpha z^\circ|}
    -Ct>0
\]
for sufficiently small $T>0$, since $0<\bar c<1/2$ and
$-\sigma^{(0)}/|\partial_\alpha z^\circ|$ has a positive
lower bound.

In $\Omega_{\mathrm{mix}}$, $\nabla\bar p$ is continuous,
since its complex conjugate in \eqref{eq:pressure-gradient}
is holomorphic off $z_\pm(t)$ and depends continuously on $t$.
Moreover, $\nabla g$ is continuous by
\eqref{eq:gradient-G}, \eqref{e:Gderivs} and
Lemma~\ref{lem:mixing-geometry}.
By \eqref{eq:velocity} and \eqref{eq:momentum},
so are $\bar u$ and $\bar m$. Therefore
\[
    (\bar\theta,\bar u,\bar m)
    \in C(\Omega_{\mathrm{mix}};\mathcal O).
\]

By \eqref{e:densityreg} and Lemma~\ref{l:products},
$\sigma_b/\partial_\alpha z_b$ is the constant
$-1+bA(1-c_\infty)$ plus a function bounded in $H^1$,
uniformly in $t$.
Lemma~\ref{l:offcurve}, applied to the curves $z_b(t)$
with its uniform constant
(Lemma~\ref{lem:mixing-geometry}), together with
\eqref{e:constant} for the constant part,
bounds the pressure gradient \eqref{eq:pressure-gradient} uniformly.
The Darcy law \eqref{eq:velocity},
Proposition~\ref{p:gradient} and \eqref{eq:momentum}
then bound $\bar u$ and $\bar m$.
Thus the triple is uniformly bounded.

As $t\to0$, the boundaries converge to $z^\circ$, so
$\bar\theta(t)\to\theta^\circ$ in $L^p_{\mathrm{loc}}$
for every finite $p\ge1$.
Moreover, every $x\notin z^\circ(\mathbb R)$ lies in
$\Omega_+(t)\cup\Omega_-(t)$ for $t$ small, and
\eqref{eq:pressure-gradient} converges at $x$ to the
Cauchy integral of $\sigma^{(0)}$ on $z^\circ$.
Hence $\bar u$ and $\bar m$ converge almost everywhere,
and the uniform bound and dominated convergence yield
weak-star continuity at $t=0$; the same argument applies
at positive times, off $z_\pm(t_0)$.
Thus the triple satisfies
Definition~\ref{definition_subsolution} and attains the
prescribed initial phase.

By Lemma~\ref{lem:mixing-geometry}, the sets $\Omega_\pm$
and $\Omega_{\mathrm{mix}}$ are pairwise disjoint and relatively
open, and their time slices cover $\mathbb R^2$ up to the
chord--arc curves $z_\pm(t)$, which have measure zero.
Theorem~\ref{thm:h-principle} now produces infinitely many
mixing solutions associated with this mixing zone and initial
phase, proving Theorem~\ref{thm:main}.
For the resulting solutions, $m=\theta u$, so \eqref{mIPMc}
states that $u+A\theta u+\theta i$ is curl-free in the sense
of distributions.
A pressure satisfying \eqref{IPMc} can therefore be recovered
distributionally; see also \cite[Appendix~A]{CFM22}
for pressure recovery in the mixing framework.

\subsection{Extension to several layers}\label{sec:multilayer}
Similarly to \cite[Section~6.2]{CFM22} and
\cite[Section~5]{ForsterSzekelyhidi18}, we explain how to extend
the construction to $0<\bar c<1$ by inserting finitely many
constant phases.
We keep the presentation brief to avoid unnecessary technical
detail, since our main goal is to construct mixing solutions
rather than optimize the opening speed.
The additional analytic details are collected in
Section~\ref{sec:layerestimates}.
For unequal viscosities we discretize the flux
$F_A(\theta)=(1-\theta^2)/(1-A\theta)$ of the macroscopic
profile in \cite[Section~2]{Mengual22}.

Fix $N\ge1$ with $\bar c<1-1/(2N)$, and put
\begin{equation}\label{e:layerlevels}
    \theta_j
    =
    \frac{A-1+j/N}{1+A(-1+j/N)},
    \qquad 0\le j\le2N.
\end{equation}
Define the ordered interfaces
\begin{equation}\label{e:layercurves}
    z_j=z+tc h_j\tau^\perp,
    \qquad
    h_j
    =
    -\frac{N}{2N-1}
    \frac{F_A(\theta_j)-F_A(\theta_{j-1})}
    {\theta_j-\theta_{j-1}},
    \quad 1\le j\le2N.
\end{equation}
Strict concavity of $F_A$ gives
$h_1<\cdots<h_{2N}$.
A direct calculation yields
\[
    h_1=-\frac1{1+A},
    \qquad
    h_{2N}=\frac1{1-A},
    \qquad
    \sum_{j=1}^{2N}(\theta_j-\theta_{j-1})h_j=0.
\]
Thus $z_1=z_-$ and $z_{2N}=z_+$ have precisely the opening
prescribed by the formula for $c$ in \eqref{eq:opening-speed},
now with $0<\bar c<1-1/(2N)$.
Set $\bar\theta=\theta_j$ in the strip between $z_j$ and
$z_{j+1}$, and retain the pure phases outside.
For $N=1$, the levels are $-1,A,1$.

\begin{figure}[htbp]
    \centering
    \includegraphics[width=0.84\textwidth]{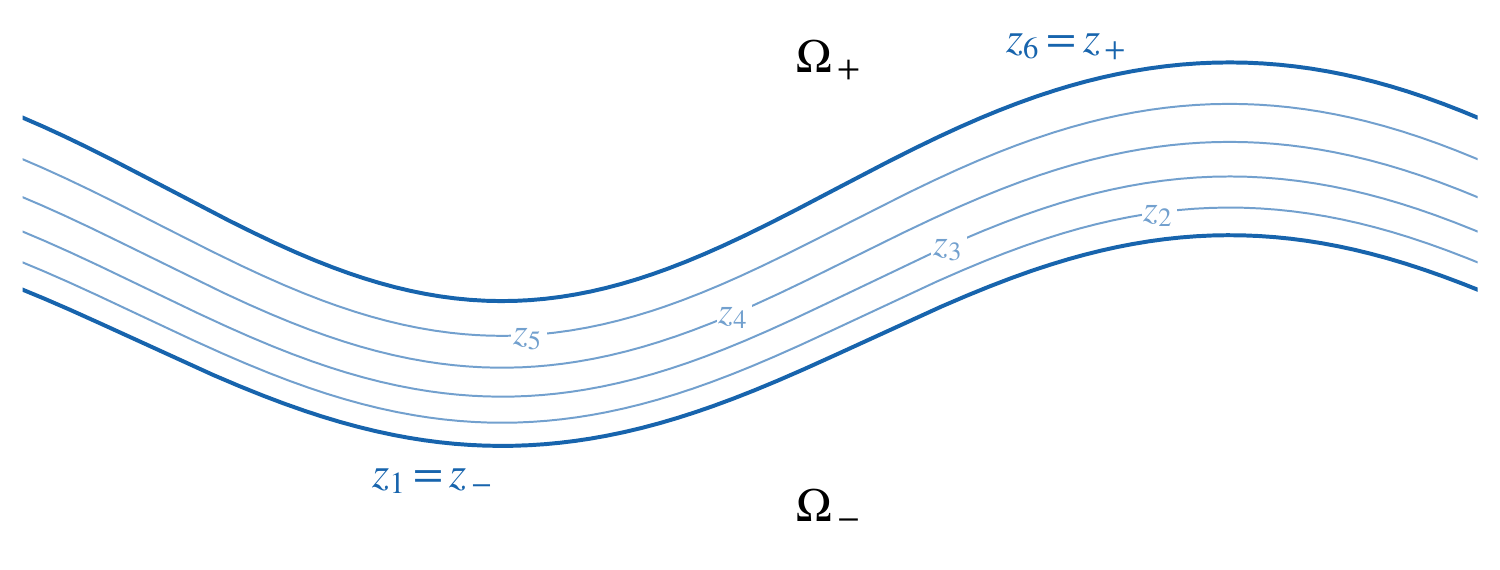}
    \caption{%
        Multilayer construction of the mixing zone
    for $N=3$ and $A=0.3$.
    The six interfaces $z_j=z+tc h_j\tau^\perp$
    delimit five intermediate strips, with
    $\bar\theta=\theta_j$ between $z_j$ and $z_{j+1}$,
    $j=1,\ldots,5$.
    }
    \label{fig:multilayer-mixing}
\end{figure}

The pressure is the sum of the same Cauchy integrals as in
Section~\ref{section4}, now with densities
\begin{equation}\label{e:layerdensities}
    \sigma_j
    =
    -(\theta_j-\theta_{j-1})(A\partial_tz_j+i)
    \cdot(\partial_\alpha z_j)^\perp.
\end{equation}
In each intermediate strip retain the momentum ansatz
\eqref{eq:momentum}, with a potential $g_j$,
\begin{equation}\label{e:layerfields}
    \bar m
    =
    \theta_j\bar u-F_A(\theta_j)
    \left(
        \nabla^\perp g_j+\frac12(A\bar u+i)
    \right).
\end{equation}
The velocity is determined by the same Darcy law
\eqref{eq:velocity}, with $\bar\theta=\theta_j$ and $g=g_j$.
In the pure phases, $\bar m=\bar\theta\bar u$.
Thus the potentials play exactly the same role as $g$
in the single-strip construction.

Lemma~\ref{l:layercompatibility}, proved in
Section~\ref{sec:layerestimates} using
Lemma~\ref{l:layerpressure}, constructs $g_j$ and a normal
coefficient $z^{(2)}\in H^{k-2}$, with the same
$z^{(1)}=-Q_0[\sigma^{(0)}]$, so that all transmission
conditions hold and
\begin{equation}\label{e:layerasymptotics}
\begin{aligned}
    (A\bar u+i)\cdot\tau^\perp
    &=
    -\frac{\sigma^{(0)}}{2|\partial_\alpha z^\circ|}
    +O(t),
    \\
    \bar m-\theta_j\bar u
    &=
    \frac{N\bar c\,\sigma^{(0)}}
    {(2N-1)|\partial_\alpha z^\circ|}
    F_A(\theta_j)\tau^\perp+O(t).
\end{aligned}
\end{equation}
All bounds are uniform within the strips for fixed $N$.

As before, the strict hull condition is
$2|\nabla g_j|<|A\bar u+i|$.
Writing $w=A\bar u+i$,
\eqref{e:layerfields}--\eqref{e:layerasymptotics} give
\begin{equation}\label{e:layerhull}
    |w|^2-4|\nabla g_j|^2
    =
    \frac{2N\bar c}{2N-1}
    \left(1-\frac{2N\bar c}{2N-1}\right)
    \frac{(\sigma^{(0)})^2}{|\partial_\alpha z^\circ|^2}
    +O(t)>0
\end{equation}
for small positive times.
Boundedness and the initial trace follow as above.
Apply Theorem~\ref{thm:h-principle} on the open union
of the strips.
The omitted internal interfaces have measure zero, and every
nonempty open subset of the full band contains an open subset
of a strip.
Consequently the resulting solutions mix throughout the full band.
For each $\bar c<1$, choosing
$N>1/(2(1-\bar c))$ proves the extension.

The quotient in \eqref{e:layercurves} is the
Rankine--Hugoniot prescription: for a scalar conservation law
$\partial_t\theta+\partial_y f(\theta)=0$, a discontinuity
joining two constant states moves at speed $[f]/[\theta]$.
Here, at each fixed $\alpha$, the physical speeds $c h_j$
correspond to the frozen scalar flux
$-Nc(\alpha)F_A/(2N-1)$.
As $N\to\infty$, the transverse step profiles converge to the
viscosity-dependent Otto profile with time scaled by $\bar c$.
Only this profile limit is used for interpretation;
the construction uses finitely many layers and does not require
a uniform existence time as $N\to\infty$.

\section{Estimates for the Birkhoff–Rott operators}\label{section8}
This section collects the operator estimates used in Sections~\ref{section6}
and~\ref{section7}. We first treat a fixed interface and prove
Proposition~\ref{p:inverse}, then study pressure traces on the moving
boundaries. The key additional estimate controls their spatial
primitives uniformly on $\mathbb R$. We conclude with the finite-layer
extension.

Throughout the fixed-interface estimates, $z^\circ$ is chord--arc,
$z^\circ-(\al,0)\in H^k(\mathbb R;\mathbb R^2)$, $k\ge5$.
Constants depend on this Sobolev norm and the chord--arc bound.
The operators $B$ and $P$ below depend on the fixed curve $z^\circ$, but we omit the subscript $0$ to keep the notation light. For $Q_0$ we retain the subscript to distinguish the initial operator the operators associated with the moving curve.

\subsection{Representation, boundary values and basic estimates}
\label{section8.1}

For real or complex densities, extend the Birkhoff--Rott integral
of Section~\ref{section3} off the interface by
\begin{equation}\label{e:Boff}
    B[f,z^\circ](x)
    =
    \frac1{2\pi i}\int_{\mathbb R}
    \frac{f(\beta)}{x-z^\circ(\beta)}\,d\beta,
    \qquad x\notin z^\circ(\mathbb R).
\end{equation}
It is holomorphic on each complementary domain.
On the interface, the same letter denotes the principal-value average
\begin{equation}\label{e:Bon}
    B[f,z^\circ](\alpha)
    =
    \frac1{2\pi i}\operatorname{pv}\int_{\mathbb R}
    \frac{f(\beta)}
    {z^\circ(\alpha)-z^\circ(\beta)}\,d\beta.
\end{equation}
We suppress the fixed curve and write $B[f]$.
The argument distinguishes the off-curve evaluation $x$ from the boundary parameter $\alpha$.

With the same kernel convention, the Cauchy integral is $C[f]=B[f\partial_\alpha z^\circ]$. Thus, most of the estimates below
also apply to $C$. Since we have not found references establishing these estimates in the precise form needed here, we include their proofs for the reader's convenience.

Throughout, principal values are understood symmetrically both at the singularity and at infinity. Densities with a constant part are interpreted using the same convention at infinity.


The operators \eqref{Q-P} are recovered as
\begin{equation}\label{e:operators}
    Q_0[f]^*=iB[f],
    \qquad
    P[f]
    =
    2Q_0[f]\cdot(\partial_\alpha z^\circ)^\perp
    =
    -2\operatorname{Re}(\partial_\alpha z^\circ B[f]).
\end{equation}
The identity for $P$ is for real $f$.
The Sokhotski--Plemelj formulas give
\begin{equation}\label{e:plemelj}
 B[f](x)\longrightarrow
 \begin{cases}
 B[f](\alpha)-\dfrac{f(\alpha)}{2\partial_\alpha z^\circ(\alpha)},
 &x\to z^\circ(\alpha),\quad x\in\Omega_{\mathrm{left}},\\[4pt]
 B[f](\alpha)+\dfrac{f(\alpha)}{2\partial_\alpha z^\circ(\alpha)},
 &x\to z^\circ(\alpha),\quad x\in\Omega_{\mathrm{right}}.
 \end{cases}
\end{equation}
For a smooth functon $f$, we can differentiate \eqref{e:Boff}. Integration by parts yields
\begin{equation}\label{e:complexderivative}
    \frac{d}{dx}B[f](x)
    =
    B\left[
        \partial_\alpha\left(
            \frac{f}{\partial_\alpha z^\circ}
        \right)
    \right](x).
\end{equation}
This formula extends to Sobolev functions.

\begin{lemma}
\label{l:offcurve}
The tangent density satisfies
\begin{equation}\label{e:constant}
    B[\partial_\alpha z^\circ](x)
    =
    \begin{cases}
        -\tfrac12,&x\in\Omega_{\mathrm{left}},\\
        \tfrac12,&x\in\Omega_{\mathrm{right}},
    \end{cases}
    \qquad
    B[\partial_\alpha z^\circ](\alpha)=0.
\end{equation}
For every $h\in H^1(\mathbb R)$,
\begin{equation}\label{e:offbound}
    \sup_{x\notin z^\circ(\mathbb R)}
    |B[h\partial_\alpha z^\circ](x)|
    \le C\|h\|_{H^1}.
\end{equation}
The principal-value integral on the curve satisfies the same estimate:
\[
    \|B[h\partial_\alpha z^\circ]\|_{L^\infty(\mathbb R)}
    \le C\|h\|_{H^1}.
\]

Moreover, if $z$ satisfies
$|z(\alpha)-z(\beta)|\ge c_0|\alpha-\beta|$ and
$\|z-(\alpha,0)-\ell\|_{H^3}\le M$
for a constant vector $\ell\in\mathbb R^2$, the same estimates hold with $C=C(c_0,M)$ for translated families.
\end{lemma}

\begin{proof}

Identity \eqref{e:constant} follows immediately from the residue theorem and the Sokhotski--Plemelj formula \eqref{e:plemelj}, since $B[\partial_\alpha z^\circ]$ is the Cauchy integral of the constant density $1$ along $z^\circ(\mathbb R)$.

For \eqref{e:offbound}, set
$d=\operatorname{dist}(x,z^\circ(\mathbb R))$ and choose
$\alpha_x\in\mathbb R$ such that $|x-z^\circ(\alpha_x)|=d$.
Minimality gives $|x-z^\circ(\beta)|\ge d$, while the chord--arc condition and the triangle inequality give the inequality $c_0|\beta-\alpha_x|\le |x-z^\circ(\beta)|+d\le2|x-z^\circ(\beta)|$. Thus
\begin{equation}\label{e:lowerkernel}
    |x-z^\circ(\beta)|
    \ge C^{-1}(|\beta-\alpha_x|+d).
\end{equation}
\emph{Case $d\ge1$.} By \eqref{e:lowerkernel} and Cauchy--Schwarz,
\[
    |B[h\partial_\alpha z^\circ](x)|
    \le C\|h\|_2\Bigl(\int_{\mathbb R}\frac{dr}{(|r|+d)^2}\Bigr)^{1/2}
    = C\|h\|_2\Bigl(\frac2d\Bigr)^{1/2}
    \leq C\|h\|_2.
\]
\emph{Case $d<1$.} We split at $|\beta-\alpha_x|=1$ and subtract $h(\alpha_x)$ in the near integral to get
\[
    2\pi i\,B[h\partial_\alpha z^\circ](x)= I_1+h(\alpha_x)\,I_2+I_3,
\]
where
\[
    I_1=\int_{|\beta-\alpha_x|<1}\!\!\!\!\!\!\!\!\!\!\!\!\!\frac{(h(\beta)-h(\alpha_x))\,\partial_\alpha z^\circ(\beta)}{x-z^\circ(\beta)}\,d\beta, \,
    I_2=\int_{|\beta-\alpha_x|<1}\frac{\partial_\alpha z^\circ(\beta)}{x-z^\circ(\beta)}\,d\beta, \,
    I_3=\int_{|\beta-\alpha_x|\ge1}\frac{h(\beta)\,\partial_\alpha z^\circ(\beta)}{x-z^\circ(\beta)}\,d\beta.
\]
In $I_1$, the fundamental theorem of calculus and the Cauchy--Schwarz inequality give
\[
    |h(\beta)-h(\alpha_x)|
    \le\|\partial_\alpha h\|_2|\beta-\alpha_x|^{1/2}.
\]
The above bound together with \eqref{e:lowerkernel} yield
\[
    |I_1|\le C\|\partial_\alpha h\|_2
    \int_0^1\frac{r^{1/2}}{r+d}\,dr
    \le C\|\partial_\alpha h\|_2\int_0^1 r^{-1/2}\,dr
    \le C\|\partial_\alpha h\|_2.
\]
By \eqref{e:lowerkernel} and the Cauchy--Schwarz inequality,
\[
    |I_3|\le C\|h\|_2\Bigl(\int_{|r|\ge1}\frac{dr}{r^2}\Bigr)^{1/2}\le C\|h\|_2.
\]
For $I_2$, we write the integral as the full integral in \eqref{e:constant} minus its far part as follows
$$
I_2 = 2\pi i\,B[\partial_\alpha z^\circ](x)
 -\operatorname{pv}\int_{|\beta-\alpha_x|\ge1}\frac{\partial_\alpha z^\circ(\beta)}{x-z^\circ(\beta)}\,d\beta.
$$
By \eqref{e:constant} the first term has modulus $\pi$. For $|\beta-\alpha_x|\ge1$, we subtract the flat kernel, whose symmetric integral vanishes and we rewrite the numerator to get
\[
    \frac{\partial_\alpha z^\circ(\beta)}{x-z^\circ(\beta)}-\frac1{\alpha_x-\beta}
    =\frac{\partial_\alpha z^\circ(\beta)-1}{x-z^\circ(\beta)}
    +\frac{(z^\circ(\beta)-\beta)-(x-\alpha_x)}{(x-z^\circ(\beta))(\alpha_x-\beta)}.
\]
Here $|z^\circ(\beta)-\beta|\le\|z^\circ-(\alpha,0)\|_\infty$ and
$|x-\alpha_x|\le|x-z^\circ(\alpha_x)|+|z^\circ(\alpha_x)-\alpha_x|\le d+\|z^\circ-(\alpha,0)\|_\infty$.
Hence, by \eqref{e:lowerkernel},
\[
    \left|
        \frac{\partial_\alpha z^\circ(\beta)}{x-z^\circ(\beta)}
        -\frac1{\alpha_x-\beta}
    \right|
    \le C\left(
        \frac{|\partial_\alpha z^\circ(\beta)-1|}{|\beta-\alpha_x|}
        +\frac{1+\|z^\circ-(\alpha,0)\|_\infty}{|\beta-\alpha_x|^2}
    \right).
\]
By Cauchy--Schwarz, the first term is integrable over $\{|\beta-\alpha_x|\ge1\}$ with integral at most $C\|\partial_\alpha z^\circ-1\|_2$, and the second is integrable directly. Thus $|I_2|\le C$.

Combining the bounds for $I_1$, $I_2$ and $I_3$ with $\|h\|_\infty\le C\|h\|_{H^1}$
proves \eqref{e:offbound}.

For the principal-value estimate, note that $h\in H^1(\mathbb R)\subset C^{1/2}(\mathbb R)$, so
\eqref{e:plemelj} applies to $f=h\partial_\alpha z^\circ$. Letting $\Omega_{\mathrm{left}} \ni x\to z^\circ(\alpha)$, we obtain
\[
    B[h\partial_\alpha z^\circ](\alpha)
    =\lim_{\substack{x\to z^\circ(\alpha)\\ x\in\Omega_{\mathrm{left}}}}B[h\partial_\alpha z^\circ](x)+\frac{h(\alpha)}{2}.
\]
Hence, by \eqref{e:offbound},
\[
    \|B[h\partial_\alpha z^\circ]\|_{L^\infty(\mathbb R)}
    \le \sup_{x\notin z^\circ(\mathbb R)}|B[h\partial_\alpha z^\circ](x)|+\tfrac12\|h\|_\infty
    \le C\|h\|_{H^1}.
\]

For the last assertion, set $\tilde z=z-\ell$. Then $\partial_\alpha\tilde z=\partial_\alpha z$,
$\tilde z$ satisfies the same chord--arc bound with constant $c_0$, and
$\|\tilde z-(\alpha,0)\|_{H^3}\le M$. Moreover, $x\notin z(\mathbb R)$ if and only if
$x-\ell\notin\tilde z(\mathbb R)$, and
$x-z(\beta)=(x-\ell)-\tilde z(\beta)$. Therefore
\[
    B[h\partial_\alpha z,z](x)=B[h\partial_\alpha\tilde z,\tilde z](x-\ell),
    \qquad
    B[h\partial_\alpha z,z](\alpha)=B[h\partial_\alpha\tilde z,\tilde z](\alpha),
\]
where the second identity uses $z(\alpha)-z(\beta)=\tilde z(\alpha)-\tilde z(\beta)$.
The proof above uses the curve only through $c_0$, $\|\partial_\alpha\tilde z\|_\infty$,
$\|\partial_\alpha\tilde z-1\|_2$ and $\|\tilde z-(\alpha,0)\|_\infty$, all of which are bounded in
terms of $c_0$ and $M$. Applying it to $\tilde z$ gives both estimates for $z$ with $C=C(c_0,M)$,
independently of $\ell$.
\end{proof}

We next estimate spatial primitives of $P$.
Set
\begin{equation}\label{e:q}
    q(\alpha,\beta)
    =
    \frac{z^\circ(\alpha)-z^\circ(\beta)}{\alpha-\beta}
    =
    \int_0^1
    \partial_\alpha z^\circ
    \bigl(\beta+r(\alpha-\beta)\bigr)\,dr.
\end{equation}
It is symmetric, with
$c_0\le|q|\le\|\partial_\alpha z^\circ\|_\infty$.
Differentiating \eqref{e:q} in $\alpha$ yields
\[
    \frac{\partial_\alpha z^\circ(\alpha)}{z^\circ(\alpha)-z^\circ(\beta)}
    =\frac{\partial_\alpha q}{q}(\alpha,\beta)+\frac1{\alpha-\beta}.
\]
Hence, for real $f$, \eqref{e:operators} and \eqref{e:Bon} yield
\begin{equation}\label{e:Pkernel}
    P[f](\alpha)
    =
    -\frac1\pi\int_{\mathbb R}
    \operatorname{Im}\left(
        \frac{\partial_\alpha q}{q}
    \right)(\alpha,\beta)f(\beta)\,d\beta,
\end{equation}
since the flat kernel $1/(\alpha-\beta)$ is real and does not contribute to the imaginary part.

\begin{lemma}
\label{l:initialprimitive}
For every real $h\in L^2(\mathbb R)$,
\begin{equation}\label{e:initialprimitive}
    \sup_{\alpha\in\mathbb R}
    \left|\int_0^\alpha P[h](\beta)\,d\beta\right|
    \le C\|h\|_2.
\end{equation}
\end{lemma}

\begin{proof}
Let $q$ be as in \eqref{e:q}, extended continuously to the diagonal by $q(\alpha,\alpha)=\partial_\alpha z^\circ(\alpha)$. Then
\[
    |q-1|
    \le\frac{2\|z^\circ-(\alpha,0)\|_\infty}{|\alpha-\beta|},
    \qquad
    q-1
    =
    \int_0^1
    \partial_\alpha(z^\circ-(\alpha,0))\bigl(\beta+r(\alpha-\beta)\bigr)\,dr.
\]
We claim that $q(\alpha,\beta)\to1$ as $|\alpha|+|\beta|\to\infty$. 

Fix $R>0$.
If $|\alpha-\beta|\ge R$, the first estimate gives $|q-1|\le 2\|z^\circ-(\alpha,0)\|_\infty/R$.

If $|\alpha-\beta|\le R$, every point $s$ of the segment between $\beta$ and $\alpha$ satisfies
$|s|\ge\frac12(|\alpha|+|\beta|)-R$, so the second identity gives
\[
|q(\alpha,\beta)-1|\le\sup_{|s|\ge\frac12(|\alpha|+|\beta|)-R}|\partial_\alpha(z^\circ-(\alpha,0))(s)|,
\]
which tends to zero as $|\alpha|+|\beta|\to\infty$, because $\partial_\alpha(z^\circ-(\alpha,0))\in H^{k-1}$
vanishes at infinity. Since $R$ is arbitrary, this proves the claim.

Since $q$ is continuous and nonvanishing on $\mathbb R^2$,
it has a continuous logarithm.
Normalize it to agree with the branch near $1$ outside a sufficiently
large disk, and set $L=\operatorname{Im}(\log q)$.
This real-valued function is bounded and satisfies
\begin{equation}\label{e:phasebounds}
    |L(\alpha,\beta)|
    \le\frac C{|\alpha-\beta|}
    \quad\text{for }|\alpha-\beta|>1,
    \qquad
    \sup_\alpha\|L(\alpha,\cdot)\|_2\le C.
\end{equation}
By \eqref{e:Pkernel}, the kernel of $P$ is
$-\pi^{-1}\partial_\alpha L$.
Hence
\begin{equation}\label{e:phaseprimitive}
    \int_0^\alpha P[h](\beta)\,d\beta
    =
    -\frac1\pi\int_{\mathbb R}
    h(\gamma)
    \bigl(L(\alpha,\gamma)-L(0,\gamma)\bigr)\,d\gamma.
\end{equation}

For smooth compactly supported $h$, Fubini's theorem justifies \eqref{e:phaseprimitive}, and the Cauchy--Schwarz inequality together with \eqref{e:phasebounds} gives \eqref{e:initialprimitive}.
For general real $h\in L^2$, choose such $h_n$ with $h_n\to h$ in $L^2$. By the same bound, the
right-hand side of \eqref{e:phaseprimitive} for $h_n$ converges uniformly in $\alpha$ to that for $h$. Since $P$ is bounded on $L^2$ by Proposition~\ref{p:sobolev} and \eqref{e:operators}, the left-hand
side converges uniformly on bounded intervals. Hence \eqref{e:phaseprimitive} and
\eqref{e:initialprimitive} hold for all real $h\in L^2$.
\end{proof}

\subsection{Sobolev estimates}

Let
\[
 H[f](\alpha)=\frac1\pi\operatorname{pv}\int_{\mathbb R}
 \frac{f(\beta)}{\alpha-\beta}\,d\beta
\]
be the Hilbert transform, so that $B[g]=-\tfrac i2H[g]$ for the flat interface $z^\circ=(\alpha,0)$.
Subtracting the flat operator from $B[g\partial_\alpha z^\circ]$ leaves
\begin{equation}\label{e:R}
    R[g]=B[g\partial_\alpha z^\circ]+\tfrac i2H[g],
\end{equation}
whose kernel, thanks to the tangent factor, is no longer singular, and which gains one derivative.
The same smoothing holds for the pressure operator $P$ and its adjoint. This regularization is
classical, see \cite[Lemmas~3.1 and~5.2]{CCG11}. We give a short proof for the present
asymptotically flat setting, keeping track of the behavior at infinity.

\begin{proposition}
\label{p:sobolev}\label{l:pressureoperator}
The operator $B$ is bounded on $H^j$, $0\le j\le k-1$.
For every $f\in H^1$,
\begin{equation}\label{e:Bderivative}
 \partial_\alpha B[f]
 =\partial_\alpha z^\circ
 B\!\left[\partial_\alpha\!\left(\frac f{\partial_\alpha z^\circ}\right)\right]
.
\end{equation}
Moreover, the operator
$g\mapsto B[g\partial_\alpha z^\circ]+\tfrac i2H[g]$
and the real operators $P,P^\dagger$ map $H^j$ boundedly into
$H^{j+1}$ for $0\le j\le k-2$.
The operator $P$ is compact on real $L^2$, and its adjoint is
\begin{equation}\label{e:adjoint}
 P^\dagger[g]=2\operatorname{Re}B[g\partial_\alpha z^\circ].
\end{equation}
\end{proposition}

\begin{proof}
The kernel of $R[g]$ in \eqref{e:R} is 
\begin{equation}\label{e:qbeta}
 K(\alpha,\beta)=\frac1{2\pi i}\left(
 \frac{\partial_\beta z^\circ(\beta)}{z^\circ(\alpha)-z^\circ(\beta)}
 -\frac1{\alpha-\beta}\right)
 =-\frac1{2\pi i}\frac{\partial_\beta q}{q}.
\end{equation}
Both $K$ and $\partial_\alpha K$ belong to $L^2(\mathbb R^2)$.
Indeed, the segment formula \eqref{e:q} gives
\[
 \partial_\alpha q=\int_0^1 r\,\partial_\alpha^2z^\circ(\beta+r(\alpha-\beta))\,dr,\qquad
 \partial_\beta q=\int_0^1(1-r)\,\partial_\alpha^2z^\circ(\beta+r(\alpha-\beta))\,dr,
\]
\[
 \partial_{\alpha\beta}q=\int_0^1r(1-r)\,\partial_\alpha^3z^\circ(\beta+r(\alpha-\beta))\,dr.
\]
Moreover,
\[
 \partial_\alpha K=-\frac1{2\pi i}
 \left(\frac{\partial_{\alpha\beta}q}{q}
 -\frac{\partial_\alpha q\,\partial_\beta q}{q^2}\right).
\]
The first term requires the third derivative of $z^\circ$.
For the product term, bound one factor by
$\|\partial_\alpha^2z^\circ\|_\infty$ and the other by its segment
average. Since $q$ is bounded away from zero, for $|\alpha-\beta|<1$,
\[
 |K|+|\partial_\alpha K|
 \le C\int_0^1 (|\partial_\alpha^2z^\circ|+|\partial_\alpha^3z^\circ|)
 (\beta+r(\alpha-\beta))\,dr.
\]

By Cauchy--Schwarz in $r$ and the change of variables $t=\alpha-\beta$, $s=\beta+r(\alpha-\beta)$, which has unit Jacobian for each fixed $r$,
\[
 \iint_{|\alpha-\beta|<1}\bigl(|K|+|\partial_\alpha K|\bigr)^2\,d\alpha\,d\beta
 \le C\int_0^1\!\!\int_{|t|<1}\!\int_{\mathbb R}(|\partial_\alpha^2z^\circ|+|\partial_\alpha^3z^\circ|)^2(s)\,ds\,dt\,dr
 \le C\bigl(\|\partial_\alpha^2z^\circ\|_2^2+\|\partial_\alpha^3z^\circ\|_2^2\bigr).
\]
For $|\alpha-\beta|\ge1$, writing the numerators as in the proof of Lemma~\ref{l:offcurve} gives
\[
 |K|+|\partial_\alpha K|
 \le C\left(
 \frac{|\partial_\alpha z^\circ(\alpha)-1|+
       |\partial_\beta z^\circ(\beta)-1|}{|\alpha-\beta|}
 +\frac{|z^\circ(\alpha)-\alpha|+|z^\circ(\beta)-\beta|}
 {|\alpha-\beta|^2}\right).
\]
Since each numerator depends on one variable only, integrating first in the other one bounds the squared integral by $C(\|\partial_\alpha z^\circ-1\|_2^2+\|z^\circ-(\alpha,0)\|_2^2)$.
By Cauchy--Schwarz, $R$ and $\partial_\alpha R$ are bounded on $L^2$, that is, $R:L^2\to H^1$.
The decomposition
\[
 B[f]=R[f/\partial_\alpha z^\circ]-\tfrac i2H[f/\partial_\alpha z^\circ]
\]
gives $L^2$ boundedness of $B$.
Differentiating in $\alpha$ the trace from $\Omega_{\mathrm{left}}$ in \eqref{e:plemelj} and using
\eqref{e:complexderivative}, the terms $\partial_\alpha\bigl(f/(2\partial_\alpha z^\circ)\bigr)$ cancel
and we obtain \eqref{e:Bderivative}.
Induction and Lemma~\ref{l:products} then give boundedness on $H^j$, $j\le k-1$.

For the higher regularizing bounds we use
\[
 [H,\varphi]:H^j\longrightarrow H^{j+1},\qquad 0\le j\le k-2,
 \quad \varphi-\varphi_\infty\in H^{k-1}.
\]
Here $[H,\varphi]g=H[\varphi g]-\varphi H[g]$.
Its kernel is
\[
 L_\varphi(\alpha,\beta)
 =\frac{\varphi(\beta)-\varphi(\alpha)}{\pi(\alpha-\beta)}
 =-\frac1\pi\int_0^1\partial_\alpha\varphi
 (\beta+r(\alpha-\beta))\,dr.
\]
Arguing as for $K$, for $|\alpha-\beta|<1$, $L_\varphi$ and $\partial_\alpha L_\varphi$ are bounded
by segment averages of $|\partial_\alpha\varphi|$ and $|\partial_\alpha^2\varphi|$, while for
$|\alpha-\beta|\ge1$,
\[
\begin{aligned}
 |L_\varphi|&\le C\frac{|\varphi(\alpha)-\varphi_\infty|
 +|\varphi(\beta)-\varphi_\infty|}{|\alpha-\beta|},\\
 |\partial_\alpha L_\varphi|&\le C\frac{|\partial_\alpha\varphi(\alpha)|}
 {|\alpha-\beta|}
 +C\frac{|\varphi(\alpha)-\varphi_\infty|
 +|\varphi(\beta)-\varphi_\infty|}{|\alpha-\beta|^2}.
\end{aligned}
\]
These are square-integrable since $\varphi-\varphi_\infty$ and
$\partial_\alpha\varphi$ belong to $L^2$. This proves $[H,\varphi]:L^2\to H^1$.

Higher orders follow by induction from
\[
 \partial_\alpha[H,\varphi]g
 =[H,\varphi]\partial_\alpha g
 +H[\partial_\alpha \varphi\,g]-\partial_\alpha \varphi\,H[g]
\]
and Lemma~\ref{l:products}.

Differentiating \eqref{e:R}, using \eqref{e:Bderivative} with $f=g\partial_\alpha z^\circ$, which gives
$\partial_\alpha B[g\partial_\alpha z^\circ]=\partial_\alpha z^\circ B[\partial_\alpha g]$, writing
$B[\partial_\alpha g]=R[\partial_\alpha g/\partial_\alpha z^\circ]-\tfrac i2H[\partial_\alpha g/\partial_\alpha z^\circ]$,
and combining the Hilbert-transform terms into a commutator, we obtain
\begin{equation}\label{e:Rderivative}
 \partial_\alpha R[g]
 =\partial_\alpha z^\circ R[\partial_\alpha g/\partial_\alpha z^\circ]
 -\tfrac i2\partial_\alpha z^\circ
 [H,1/\partial_\alpha z^\circ]\partial_\alpha g.
\end{equation}
For $g\in H^j$, the input $(\partial_\alpha g)/\partial_\alpha z^\circ$
belongs to $H^{j-1}$. The induction hypothesis for $R$ and the
commutator bound with $\varphi=1/\partial_\alpha z^\circ$, $\varphi_\infty=1$, therefore put both terms on the right in $H^j$.
Starting with $R:L^2\to H^1$, this proves $R:H^j\to H^{j+1}$
for $j\le k-2$.

The real kernel of $P$ is given by \eqref{e:Pkernel}.
By symmetry of $q$ and the preceding bounds for $K$, this kernel
is square-integrable, so $P$ is Hilbert--Schmidt and hence compact.
Finally,
\[
 \partial_\alpha z^\circ B[f]
 =\partial_\alpha z^\circ R[f/\partial_\alpha z^\circ]
 -\tfrac i2\partial_\alpha z^\circ[H,1/\partial_\alpha z^\circ]f
 -\tfrac i2H[f].
\]
The last term has zero real part, proving the gain of one derivative
for $P$. Transposing its real kernel gives \eqref{e:adjoint}, and
$P^\dagger[g]=2\operatorname{Re}R[g]$ for real $g$, which proves
its regularizing bounds.
\end{proof}

\subsection{Saffman--Taylor inversion}\label{section8.3}

We prove Proposition~\ref{p:inverse} by the classical harmonic-energy
argument; cf.\ \cite[Section~4]{CCG11}.

\begin{proof}[Proof of Proposition~\ref{p:inverse}]
The bounds for $Q_0,Q_0^*$ and $P$ follow from
Proposition~\ref{p:sobolev} and \eqref{e:operators}.
By \eqref{e:constant},
\[
 B[1]=B[1-\partial_\alpha z^\circ]\in H^{k-1},
\]
which also gives $Q_0[1],P[1]\in H^{k-1}$.

Since $P$ is compact, $I-AP$ is Fredholm of index zero on real $L^2$. Hence, for $A\ne0$,
it is invertible if and only if $I-AP^\dagger$ is injective, that is, if $P^\dagger g=\mu g$ with
$\mu=1/A$ has no solution $g\ne0$. Note that $|\mu|>1$, since $|A|<1$.
Let $g$ be such a solution. Since $g=AP^\dagger g$ and $P^\dagger$ gains one derivative,
we get $g\in H^{k-1}$. Define
\[
 \Phi(x)=\operatorname{Re}B[g\partial_\alpha z^\circ](x),\qquad x\notin z^\circ(\mathbb R),
\]
which is harmonic in each domain and bounded by Lemma~\ref{l:offcurve}.
On the curve, \eqref{e:adjoint} gives
$\operatorname{Re}B[g\partial_\alpha z^\circ](\alpha)=\tfrac12P^\dagger[g](\alpha)=\tfrac\mu2g(\alpha)$.
Hence \eqref{e:plemelj} with $f=g\partial_\alpha z^\circ$ yields the boundary traces
\begin{equation}\label{e:eigentraces}
 \Phi(x)\longrightarrow
 \begin{cases}
 \dfrac{\mu-1}{2}g(\alpha),
 &x\to z^\circ(\alpha),\quad x\in\Omega_{\mathrm{left}},\\[4pt]
 \dfrac{\mu+1}{2}g(\alpha),
 &x\to z^\circ(\alpha),\quad x\in\Omega_{\mathrm{right}}.
 \end{cases}
\end{equation}
By \eqref{e:complexderivative} with $f=g\partial_\alpha z^\circ$, the complex derivative of
$B[g\partial_\alpha z^\circ]$ is $B[\partial_\alpha g]$. By \eqref{e:plemelj}, its left trace minus
its right trace is $-\partial_\alpha g/\partial_\alpha z^\circ$.
Thus the jump of $\nabla\Phi\cdot(\partial_\alpha z^\circ)^\perp$
is $-\operatorname{Re}(i\partial_\alpha g)=0$.
This common normal trace belongs to $L^2$ by
\eqref{e:plemelj} and Proposition~\ref{p:sobolev}.
Green's formula gives
\begin{equation}\label{e:energies}
\begin{aligned}
 \int_{\Omega_{\mathrm{left}}}|\nabla\Phi|^2
 &=-\frac{\mu-1}{2}\int_{\mathbb R}
 g\,\nabla\Phi\cdot(\partial_\alpha z^\circ)^\perp\,d\alpha,\\
 \int_{\Omega_{\mathrm{right}}}|\nabla\Phi|^2
 &=\frac{\mu+1}{2}\int_{\mathbb R}
 g\,\nabla\Phi\cdot(\partial_\alpha z^\circ)^\perp\,d\alpha.
\end{aligned}
\end{equation}

For completeness, on these unbounded domains one first tests with
$\chi_R^2\Phi$, where $\chi_R=1$ on $B_R$, vanishes outside $B_{2R}$,
and $|\nabla\chi_R|\le C/R$.
The boundary product is integrable, and boundedness of $\Phi$ gives
$\int|\nabla\chi_R|^2\Phi^2\le C\|\Phi\|_\infty^2R^{-2}|B_{2R}|\le C$.
Young's inequality first yields finite Dirichlet energy. The remaining
cutoff term then vanishes by Cauchy--Schwarz and the vanishing energy
in $B_{2R}\setminus B_R$. This justifies \eqref{e:energies}.

For $|\mu|>1$, the coefficients of the common boundary integral in \eqref{e:energies} have
opposite signs. Since both energies are nonnegative, the common integral vanishes, and so do
both energies. Thus $\Phi$ is constant in each domain, and \eqref{e:eigentraces} shows that $g$
is constant. Since $g\in L^2(\mathbb R)$, we get $g=0$.
Fredholm therefore gives a bounded inverse on $L^2$. The case $A=0$
is immediate. Finally, if $f\in H^j$ and $g=(I-AP)^{-1}f$, the identity
$g=APg+f$ and the regularization of $P$ give successively
\[
 \|g\|_{H^j}\le C\|g\|_{H^{j-1}}+\|f\|_{H^j},
 \qquad 1\le j\le k-1.
\]
This proves the higher Sobolev inverse bounds.
\end{proof}

\subsection{Estimates for the pressure operator}
\label{sec:movingpressure}

We now compare pressure traces as the boundaries approach each other.
The use of the argument principle and transverse integration is related
to \cite[Lemmas~4.1--4.3]{CFM22}.

Let
\[
    z_\lambda(t,\alpha)
    =
    z(t,\alpha)+\frac{\lambda+A}{1-A^2}
    tc(\alpha)\tau^\perp(\alpha),
    \qquad -1\le\lambda\le1,
\]
with $z(0)=z^\circ$, as in Section~\ref{section4}.
Assume the geometric bounds of Lemma~\ref{lem:mixing-geometry}
uniformly for $0<t\le T$, and
\[
    z-(\alpha,0)\in C^2([0,T];H^3),
    \qquad
    c\tau^\perp-ic_\infty\in H^3,
    \qquad c_\infty>0,
    \qquad \inf c>0.
\]
For $b=\pm1$, write
\begin{equation}\label{e:Qmoving}
    Q_{\lambda,b}^*[f](t,\alpha)
    =
    \frac1{2\pi}\operatorname{pv}\int_{\mathbb R}
    \frac{f(t,\beta)}
    {z_\lambda(t,\alpha)-z_b(t,\beta)}\,d\beta.
\end{equation}
This is precisely $iB[f,z_b(t)]$ evaluated at
$z_\lambda(t,\alpha)$, using the principal-value average
when $\lambda=b$.

\begin{lemma}
\label{l:uniformpressure}
Uniformly for $0<t\le T$ and $-1\le\lambda\le1$, if $h\in H^1$,
\begin{equation}\label{e:uniformpressure}
    \|Q_{\lambda,b}^*[h\partial_\alpha z_b](t)\|_\infty
    \le C\|h\|_{H^1}.
\end{equation}
If $f-f_\infty\in C^1([0,T];H^2)$
for a time-independent constant $f_\infty$, then
\begin{equation}\label{e:pressurecomparison}
    \sup_{\lambda\in[-1,1]}
    \left\|
        Q_{\lambda,b}^*[f](t)-Q_0^*[f(0)]
        +\frac{i\operatorname{sgn}(\lambda-b)f(0)}
        {2\partial_\alpha z^\circ}
    \right\|_\infty
    \le Ct,
\end{equation}
where $\operatorname{sgn}(0)=0$.
If instead $f-f_\infty\in C([0,T];H^1)$, the left-hand side of \eqref{e:pressurecomparison} still tends to zero as $t\to0$, but without rate.
\end{lemma}

\begin{proof}
By \eqref{e:Qmoving} and \eqref{e:Boff} (or \eqref{e:Bon}),
\[
Q^*_{\lambda,b}[h\partial_\alpha z_b](t,\alpha)=iB[h\partial_\alpha z_b,z_b(t)](z_\lambda(t,\alpha)),
\]
evaluated off the curve $z_b(t)$ if $\lambda\ne b$ and as a principal value if $\lambda=b$.
The curves $z_b(t)$ have uniform chord--arc bounds, and
\[
 z_b(t)-(\alpha,0)-\frac{bic_\infty t}{1-bA}
 =z(t)-(\alpha,0)+\frac{bt}{1-bA}\bigl(c\tau^\perp-ic_\infty\bigr)
\]
is bounded in $H^3$ uniformly in $t\in(0,T]$. Hence Lemma~\ref{l:offcurve} for translated families applies to the source curve $z_b(t)$ with a constant independent of $t$,
which proves \eqref{e:uniformpressure}.

 In order to take advantage of \eqref{e:uniformpressure}, for a general density $f$, write $f=h\partial_\alpha z_b$, that is,
$h=f/\partial_\alpha z_b$. If $f\in H^1$, Lemma~\ref{l:products} gives $h\in H^1$.
If $f-f_\infty\in H^1$,
\[
    h-f_\infty
    =\frac{f-f_\infty}{\partial_\alpha z_b}
    +f_\infty\,\frac{1-\partial_\alpha z_b}{\partial_\alpha z_b},
\]
then $h-f_\infty\in H^1$ by Lemma~\ref{l:products} for the first term in the right hand side and \eqref{e:constant} for the constant part.

Now, we prove \eqref{e:pressurecomparison} in three steps. First, we compare $Q^*_{\lambda,b}[f](t)$ with its boundary value on $z_b(t)$ from the side containing $z_\lambda(t)$. Then, the self-interaction $Q^*_{b,b}[f](t)$ with $Q^*_{b,b}[f](0)=Q_0^*[f(0)]$, and lastly jump coefficients at times $t$ and $0$.

The boundary limit on the source curve from the side containing $z_\lambda$ is
\[
    Q_{\lambda,b}^*[f]
    \longrightarrow
    Q_{b,b}^*[f]
    -\frac{i\operatorname{sgn}(\lambda-b)f}
    {2\partial_\alpha z_b}.
\]

Differentiation in $\lambda$ and integration
by parts give
\begin{equation}\label{e:transversedifferentiation}
    \partial_\lambda Q_{\lambda,b}^*[f]
    =
    \frac{tc\tau^\perp}{1-A^2}
    Q_{\lambda,b}^*\left[
        \partial_\alpha\left(
            \frac f{\partial_\alpha z_b}
        \right)
    \right].
\end{equation}
The density on the right is uniformly in $H^1$,
so \eqref{e:uniformpressure} and Lemma~\ref{l:products}
bound the derivative by $Ct$.
Integrating this bound in $\lambda$, starting from the boundary limit above, we obtain
\[
    \left\|
        Q_{\lambda,b}^*[f]
        -Q_{b,b}^*[f]
        +\frac{i\operatorname{sgn}(\lambda-b)f}{2\partial_\alpha z_b}
    \right\|_\infty
    \le Ct,
\]
uniformly in $\lambda\in[-1,1]$.

For the self-interaction we use the temporal differentiation identity \eqref{eq:pressure-operator-time-derivative} with $a=b$,
\begin{equation}\label{e:timedifferentiation}
\partial_tQ_{b,b}^*[f]
    =
    Q_{b,b}^*[\partial_tf]
    +\partial_tz_b\,Q_{b,b}^*\left[
        \partial_\alpha\left(
            \frac f{\partial_\alpha z_b}
        \right)
    \right]
    -Q_{b,b}^*\left[
        \partial_\alpha\left(
            \frac{\partial_tz_b\,f}{\partial_\alpha z_b}
        \right)
    \right].
\end{equation}
Its right-hand side is uniformly bounded in $L^\infty$ by \eqref{e:uniformpressure}. Integrating in time gives
\[
    \|Q_{b,b}^*[f](t)-Q_0^*[f(0)]\|_\infty\le Ct.
\]
Finally, for the jump coefficients we write
\[
    \frac{f(t)}{\partial_\alpha z_b(t)}-\frac{f(0)}{\partial_\alpha z^\circ}
    =\frac{f(t)-f(0)}{\partial_\alpha z_b(t)}
    +f(0)\,\frac{\partial_\alpha z^\circ-\partial_\alpha z_b(t)}
    {\partial_\alpha z_b(t)\,\partial_\alpha z^\circ}.
\]
Both terms are $O(t)$ in $L^\infty$ by the assumed regularity of $f$ and $z$
and the chord--arc condition. Combining the three estimates proves \eqref{e:pressurecomparison}.
The differentiation identities hold first for smooth data
and then by approximation with uniform bounds.

The rate in \eqref{e:pressurecomparison} comes from the differentiation
identities \eqref{e:transversedifferentiation} and \eqref{e:timedifferentiation},
which involve $\partial_\alpha(f/\partial_\alpha z_b)$ and $\partial_tf$ and hence
require one more derivative of $f$ in space and in time. For $H^1$ densities only
the qualitative statement survives, and we obtain it by density.

Let $f-f_\infty\in C([0,T];H^1)$ and $\varepsilon>0$. Choose $g$ with
$g-f_\infty\in H^2$ and $\|f(0)-g\|_{H^1}\le\varepsilon$, and regard $g$ as a
time-independent density. By \eqref{e:uniformpressure}, Lemma~\ref{l:products},
the corresponding bound for $Q_0^*$ at $t=0$, and the embedding
$H^1\subset L^\infty$, the expression in \eqref{e:pressurecomparison} changes by at most
$C\bigl(\|f(t)-f(0)\|_{H^1}+\|f(0)-g\|_{H^1}\bigr)$ when $f$ is replaced by $g$.
Applying \eqref{e:pressurecomparison} to $g$, its value is bounded by
\[
    C_g\,t+C\|f(t)-f(0)\|_{H^1}+C\varepsilon .
\]
Letting $t\to0$ and then $\varepsilon\to0$ proves the claim.
\end{proof}

\begin{lemma}
\label{l:initialderivative}
Under the assumptions of this subsection, let $a,b=\pm1$.
For a density $f$ with $f-f_\infty\in C([0,T];H^1)$ and $f_\infty$ a
time-independent constant, extend $Q_{a,b}^*[f]$ to $t=0$ by
\begin{equation}\label{e:extension}
    Q_{a,b}^*[f](0)
    =Q_0^*[f(0)]-\frac{i\operatorname{sgn}(a-b)\,f(0)}{2\partial_\alpha z^\circ},
\end{equation}
which is its limit as $t\to0^+$ by Lemma~\ref{l:uniformpressure}.
Moreover, if $f-f_\infty\in C^2([0,T];H^2)$, then
$Q_{a,b}^*[f]\in C^1([0,T];L^\infty(\mathbb R))$, and
\eqref{eq:pressure-operator-time-derivative} also holds at $t=0$,
with each term on the right-hand side extended as in \eqref{e:extension}.
\end{lemma}
\begin{proof}
The densities
\[
    \partial_tf,
    \qquad
    \partial_\alpha\left(\frac f{\partial_\alpha z_b}\right),
    \qquad
    \partial_\alpha\left(\frac{\partial_tz_b\,f}{\partial_\alpha z_b}\right)
\]
belong to $C^1([0,T];H^1)$ by Lemma~\ref{l:products}, with zero constant part.
The last assertion of Lemma~\ref{l:uniformpressure}
gives the limits as $t\to0^+$ of each term on the right-hand side of
\eqref{eq:pressure-operator-time-derivative},
and $\partial_tz_a(t)\to\partial_tz_a(0)$ in $L^\infty$.
Hence $\partial_tQ_{a,b}^*[f](t)$ converges
in $L^\infty$ to the
right-hand side of \eqref{eq:pressure-operator-time-derivative} at $t=0$.
Integrating \eqref{eq:pressure-operator-time-derivative} over $[0,t]$ identifies this limit as the right derivative of the extension at $t=0$.
\end{proof}

\begin{lemma}
\label{l:temporalprimitive}
Let $f_b$ be real densities with
$f_b-f_{b,\infty}\in C^2([0,T];H^2)$,
where $f_{b,\infty}$ is independent of time.
Define
\begin{equation}\label{e:correctedtraces}
    \widetilde Q_a^*
    =
    \sum_{b=\pm1}Q_{a,b}^*[f_b]
    +\frac{aif_{-a}}{2\partial_\alpha z_a},
    \qquad
    T_a=\widetilde Q_a\cdot(\partial_\alpha z_a)^\perp.
\end{equation}
Then, for $0<t\le T$,
\begin{equation}\label{e:temporalprimitive}
    \sup_{\alpha\in\mathbb R}
    \left|
        \int_0^\alpha
        \bigl(
            \partial_tT_a(t,\beta)-\partial_tT_a(0,\beta)
        \bigr)\,d\beta
    \right|
    \le Ct.
\end{equation}
The value at zero is the right derivative from
Lemma~\ref{l:initialderivative}.
The constant depends on the stated curve and density norms,
their constant parts and the uniform geometric bounds.
\end{lemma}

\begin{proof}
We write $T_a$ in terms of a moving version $P_t$ of the operator $P$
in \eqref{e:operators}, with $P_0=P$. 

Fix $a,b\in\{\pm1\}$. For a real density $h$, set
\begin{equation}\label{e:NE}
 P_t[h]=-2\operatorname{Im}(\partial_\alpha z_a\,Q_{a,b}^*[h]) -\delta_{b,-a}a h,
\end{equation}
where $\delta_{b,-a}$ is the Kronecker delta. When the pair matters we
write $P_{a,b,t}$.
At $t=0$, $Q_{a,b}^*[h]$ is extended as in \eqref{e:extension}.
If $b=-a$, the Plemelj jump in \eqref{e:extension} contributes $ah$ to the first term of $P_0[h]$, and the local correction cancels it. Hence $P_0=P$ by
\eqref{e:operators}.
Moreover, $Q\cdot(\partial_\alpha z_a)^\perp
=-\operatorname{Im}(\partial_\alpha z_a\,Q^*)$, so
\eqref{e:correctedtraces} gives
\begin{equation}\label{e:Tsplit}
 T_a(t)=\frac12\sum_{b=\pm1}P_{a,b,t}[f_b(t)].
\end{equation}
Therefore it suffices to prove \eqref{e:temporalprimitive} with $T_a$
replaced by $P_t[f_b(t)]$, for each fixed pair $(a,b)$. Since $P$ has a bounded primitive by Lemma~\ref{l:initialprimitive}, the key point is the identity \eqref{e:conservation} below. The quantity $\partial_t(P_t[f(t)])$ equals $P_t[\partial_tf(t)]$ plus the $\alpha$-derivative of a bounded function $C_t[f(t)]$, so after integration in $\alpha$ this last term only contributes boundary values.

Let $0<t\le T$. Differentiating \eqref{e:Qmoving} and integrating by parts gives
\begin{equation}\label{e:spatialdifferentiation}
 \partial_\alpha Q_{a,b}^*[f]
 =\partial_\alpha z_a\,
 Q_{a,b}^*\left[\partial_\alpha\left(\frac f{\partial_\alpha z_b}\right)\right],
\end{equation}
in the principal-value sense when $a=b$. Using the product rule, we get 
\[
 \partial_t\bigl(\partial_\alpha z_a\,Q_{a,b}^*[f]\bigr)
 = (\partial_t \partial_\alpha z_a) Q^*_{a,b}[f]+\partial_\alpha z_a \partial_t Q^*_{a,b}[f].
\]
Introducing \eqref{eq:pressure-operator-time-derivative}, using \eqref{e:spatialdifferentiation} in its last two terms, and grouping the terms $(\partial_t\partial_\alpha z_a)\,Q^*_{a,b}[f]$ and $\partial_tz_a\,\partial_\alpha Q^*_{a,b}[f]$, we get
\[
 \partial_t\bigl(\partial_\alpha z_a\,Q_{a,b}^*[f]\bigr)
 =\partial_\alpha z_a\,Q_{a,b}^*[\partial_tf]
 +\partial_\alpha\bigl(\partial_tz_a\,Q_{a,b}^*[f]
 -Q_{a,b}^*[(\partial_tz_b)f]\bigr).
\]
Taking $-2\operatorname{Im}$ and subtracting $\delta_{b,-a}a\,\partial_tf$
from both sides gives
\begin{equation}\label{e:conservation}
 \partial_t(P_t[f(t)])
 =P_t[\partial_tf(t)]+\partial_\alpha C_t[f(t)],
 \qquad 0<t\le T,
\end{equation}
where $C_t[f]=C_{a,b,t}[f]=-2\operatorname{Im}\bigl( \partial_tz_a\,Q_{a,b}^*[f]-Q_{a,b}^*[(\partial_tz_b)f]\bigr)$. These identities hold first for smooth data, and then by approximation
with the uniform bounds of Lemma~\ref{l:uniformpressure}.

Let $h\in H^1$ be real and independent of time. Then \eqref{e:conservation}
reduces to $\partial_sP_s[h]=\partial_\alpha C_s[h]$ for $0<s\le t$.
Integrating over $[0,\alpha]\times[\varepsilon,t]$ gives
\[
 \int_0^\alpha(P_t[h]-P_\varepsilon[h])(\beta)\,d\beta
 =\left[\int_\varepsilon^t C_s[h]\,ds\right]_0^\alpha.
\]
As $\varepsilon\to0$, $P_\varepsilon[h]\to P_0[h]=P[h]$ in $L^\infty$ by the last assertion of Lemma~\ref{l:uniformpressure}. Hence the
left-hand side converges for each fixed $\alpha$.
On the right, $\partial_tz_b$ is bounded in $W^{1,\infty}$, so \eqref{e:uniformpressure} and
Lemma~\ref{l:products} give $\|C_s[h]\|_\infty\le C\|h\|_{H^1}$,
uniformly in $s$. Therefore
\begin{equation}\label{e:Nprimitive}
 \int_0^\alpha(P_t[h]-P[h])(\beta)\,d\beta
 =\left[\int_0^t C_s[h]\,ds\right]_0^\alpha,
 \qquad
 \sup_{\alpha\in\mathbb R}
 \left|\int_0^\alpha(P_t[h]-P[h])\,d\beta\right|
 \le Ct\|h\|_{H^1}.
\end{equation}

Fix $b$. For $0<t\le T$, integrating \eqref{e:conservation} over
$[0,\alpha]$ gives
\begin{equation}\label{e:integratedconservation}
 \int_0^\alpha\partial_t(P_t[f_b(t)])\,d\beta
 =\int_0^\alpha P_t[\partial_tf_b(t)]\,d\beta
 +\bigl[C_t[f_b(t)]\bigr]_0^\alpha.
\end{equation}
We let $t\to0$ with $\alpha$ fixed. By Lemma~\ref{l:initialderivative},
the integrand on the left converges in $L^\infty$ to the right derivative
at $t=0$. Since $\partial_tf_b(t)-\partial_tf_{b}(0)\to0$ in $H^1$, we have
$P_t[\partial_tf_b(t)]\to P[\partial_tf_b(0)]$ in $L^\infty$ by
\eqref{e:uniformpressure} for the change of density and the last
assertion of Lemma~\ref{l:uniformpressure} for the change of curves. Finally,
$C_t[f_b(t)]\to C_0[f_b(0)]$ by \eqref{e:pressurecomparison}.
Hence \eqref{e:integratedconservation} also holds at $t=0$:
\[
\int_0^\alpha\left.\partial_s(P_s[f_b(s)])\right|_{s=0}\,d\beta
 =\int_0^\alpha P[\partial_tf_b(0)]\,d\beta
 +\bigl[C_0[f_b(0)]\bigr]_0^\alpha.
\]
Subtracting this from \eqref{e:integratedconservation}, and splitting
\[
 P_t[\partial_tf_b(t)]-P[\partial_tf_b(0)]
 =P[\partial_tf_b(t)-\partial_tf_b(0)]
 +(P_t-P)[\partial_tf_b(t)],
\]
we obtain
\begin{equation}\label{e:lastbound}
\begin{aligned}
 \int_0^\alpha\left(
 \partial_t(P_t[f_b(t)])-\left.\partial_s(P_s[f_b(s)])\right|_{s=0}\right)\,&d\beta
 =\int_0^\alpha P[\partial_tf_b(t)-\partial_tf_b(0)]\,d\beta\\
 &+\int_0^\alpha(P_t-P)[\partial_tf_b(t)]\,d\beta
 +\bigl[C_t[f_b(t)]-C_0[f_b(0)]\bigr]_0^\alpha.
\end{aligned}
\end{equation}
We bound the three terms on the right uniformly in $\alpha$. By Lemma~\ref{l:initialprimitive}, the first term is bounded by
$C\|\partial_tf_b(t)-\partial_tf_b(0)\|_2\le Ct$, since
$f_b-f_{b,\infty}\in C_t^2H^2$. 

By \eqref{e:Nprimitive} with $h=\partial_tf_b(t)$, regarded as a time-independent density, the second term is at most $Ct\|\partial_tf_b(t)\|_{H^1}\le Ct$.

The third term is bounded by
$2\|C_t[f_b(t)]-C_0[f_b(0)]\|_\infty$. Note that
\[
 (\partial_tz_b)f_b-\frac{bic_\infty f_{b,\infty}}{1-bA}\in C_t^1H^2,
\]
with a time-independent constant part. Applying \eqref{e:pressurecomparison} to $f_b$ and to $(\partial_tz_b)f_b$, we get
$\|Q_{a,b}^*[f_b](t)-Q_{a,b}^*[f_b](0)\|_\infty\le Ct$ and
$\|Q_{a,b}^*[(\partial_tz_b)f_b](t)-Q_{a,b}^*[(\partial_tz_b)f_b](0)\|_\infty\le Ct$.
These bounds, which also give $\|Q_{a,b}^*[f_b](t)\|_\infty\le C$, together with
   $\|\partial_tz_a(t)-\partial_tz_a(0)\|_\infty\le Ct$, give
   $\|C_t[f_b(t)]-C_0[f_b(0)]\|_\infty\le Ct$.

Hence the left-hand side of \eqref{e:lastbound} is bounded by $Ct$
uniformly in $\alpha$. Summing over $b$ and using \eqref{e:Tsplit} proves \eqref{e:temporalprimitive}.
\end{proof}

\subsection{Estimates for finitely many layers}
\label{sec:layerestimates}

We record only the additions needed in Section~\ref{sec:multilayer}.
All constants below may depend on the fixed number and spacing
of the layers.

\begin{lemma}
\label{l:layerpressure}
Let $h_1<\cdots<h_{2N}$ be fixed, and set
\[
    Z(t,\alpha,\lambda)
    =
    z(t,\alpha)+tc(\alpha)\lambda\tau^\perp(\alpha),
    \qquad
    z_j=Z(t,\cdot,h_j).
\]
Assume $z-(\alpha,0)\in C_t^2H^3$,
$c\tau^\perp-ic_\infty\in H^3$,
$\inf c>0$, and the geometric bounds of
Lemma~\ref{lem:mixing-geometry}.
For real densities $f_j-f_{j,\infty}\in C_t^2H^2$
with constant $f_{j,\infty}$,
let $Q_j^*=\sum_{\ell=1}^{2N}Q_{j,\ell}^*[f_\ell]$, where $Q_{j,\ell}^*$ is defined as in \eqref{e:Qmoving} with $z_\lambda,z_b$ replaced by $z_j,z_\ell$, and define
\begin{equation}\label{e:layercorrected}
    \widetilde Q_j
    =
    Q_j-\frac{i}{2(\partial_\alpha z_j)^*}
    \sum_{\ell=1}^{2N}\operatorname{sgn}(j-\ell)f_\ell,
    \qquad
    T_j=\widetilde Q_j\cdot(\partial_\alpha z_j)^\perp.
\end{equation}
The pressure comparisons and initial temporal derivative estimates
of Lemmas~\ref{l:uniformpressure}--\ref{l:initialderivative}
hold for each pair of curves, with jump sign
$\operatorname{sgn}(\lambda-h_\ell)$.
Moreover,
$\widetilde Q_j(0)=Q_0[\sum_\ell f_\ell(0)]$ and
\begin{equation}\label{e:layerprimitive}
    \sup_\alpha
    \left|
        \int_0^\alpha
        \bigl(
            \partial_tT_j(t,\beta)-\partial_tT_j(0,\beta)
        \bigr)\,d\beta
    \right|
    \le Ct.
\end{equation}
\end{lemma}

\begin{proof}
Apply the proofs of Lemmas~\ref{l:uniformpressure}--\ref{l:temporalprimitive}
to each pair $(j,\ell)$. Transverse differentiation has the factor
$tc\tau^\perp$, Plemelj has the sign $\operatorname{sgn}(j-\ell)$,
and the local correction in $P_t$ is $-\operatorname{sgn}(j-\ell)h$.
The identities \eqref{eq:pressure-operator-time-derivative}, \eqref{e:conservation}
and \eqref{e:Nprimitive} are unchanged. Since $\partial_t z_\ell=\partial_t z+h_\ell(c\tau^\perp-ic_\infty)+ih_\ell c_\infty$, we have $(\partial_t z_\ell)f_\ell-ih_\ell c_\infty f_{\ell,\infty}\in C^1_tH^2$, which replaces $(\partial_tz_b)f_b-\frac{bic_\infty f_{b,\infty}}{1-bA}\in C^1_tH^2$ in the proof of Lemma~\ref{l:temporalprimitive}; thus \eqref{e:pressurecomparison} applies to both integrals defining $C_t$. Summing the pairwise estimates proves the assertions.
\end{proof}

\begin{lemma}
\label{l:layercompatibility}
For the phases and speeds in
\eqref{e:layerlevels}--\eqref{e:layercurves},
with $c$ given by the formula in \eqref{eq:opening-speed}
and $0<\bar c<1-1/(2N)$,
and with $z=z^\circ+tz^{(1)}+\frac{t^2}{2}z^{(2)}$,
$z^{(1)}=-Q_0[\sigma^{(0)}]$,
there exist a normal $z^{(2)}\in H^{k-2}$ and potentials $g_j$
for which \eqref{e:layerdensities}--\eqref{e:layerfields}
satisfy the linear system in distributions and the expansions
\eqref{e:layerasymptotics}.
\end{lemma}

\begin{proof}
We first construct auxiliary potentials $\phi_j$ and the fields
\begin{equation}\label{e:layerauxiliary}
    \bar u
    =
    A\nabla^\perp\phi_j
    -\frac{\nabla p+\theta_j i}{1+A\theta_j},
    \qquad
    \bar m
    =
    \theta_j\bar u-(1+A\theta_j)\nabla^\perp\phi_j.
\end{equation}
Set $\phi_0=\phi_{2N}=0$ in the pure phases.
This parametrization makes the transmission conditions explicit;
we recover $g_j$ at the end.
Write $\delta_j=\theta_j-\theta_{j-1}$ and
\[
    \omega_j
    =
    \frac{\delta_j}
    {(1+A\theta_{j-1})(1+A\theta_j)},
    \qquad
    \sum_{j=1}^{2N}\omega_j=\frac{2}{1-A^2}.
\]
The required tangential jump of the potentials is
\begin{equation}\label{e:layerresidual}
\begin{aligned}
    \partial_\alpha
    (\phi_j|_{z_j}-\phi_{j-1}|_{z_j})
    &=R_j,
    \\
    R_j
    &=
    -\omega_j
    \Bigl(
        \Bigl(1+\frac A2(\theta_{j-1}+\theta_j)\Bigr)\partial_tz_j+Q_j
        +\frac12(\theta_{j-1}+\theta_j)\, i
    \Bigr)\cdot(\partial_\alpha z_j)^\perp.
\end{aligned}
\end{equation}
Indeed, set

\[
 Y_-=\frac{((\nabla p)_{j,-}+\theta_{j-1} i)
 \cdot(\partial_\alpha z_j)^\perp}{1+A\theta_{j-1}},
 \qquad
 Y_+=\frac{((\nabla p)_{j,+}+\theta_j i)
 \cdot(\partial_\alpha z_j)^\perp}{1+A\theta_j},
\]
where $(\nabla p)_{j,-}$ and $(\nabla p)_{j,+}$ are the traces on $z_j$
from the $\theta_{j-1}$ and $\theta_j$ sides respectively.
The pressure jump
$(\nabla p)_{j,+}-(\nabla p)_{j,-}
=\sigma_j(\partial_\alpha z_j)^\perp/|\partial_\alpha z_j|^2$
and the density prescription give
\[
 Y_+-Y_-=AR_j,\qquad
 (1+A\theta_j)R_j=-\delta_j
 \bigl(\partial_tz_j\cdot(\partial_\alpha z_j)^\perp+Y_-\bigr).
\]
Since the normal components of the fields are
$A\partial_\alpha(\phi_{j-1}|_{z_j})-Y_-$ and
$-\partial_\alpha(\phi_{j-1}|_{z_j})-\theta_{j-1}Y_-$
on the $\theta_{j-1}$ side, and the same expressions with
$\phi_j,\theta_j,Y_+$ on the $\theta_j$ side, their jumps are
$0$ and $\delta_j\partial_tz_j\cdot(\partial_\alpha z_j)^\perp$.
These give both transmission conditions, including when $A=0$;
the regional equations follow from harmonicity of $p$ and Darcy's law.

Set $R=\sum_jR_j$ and use \eqref{e:layercorrected}
with $f_j=\sigma_j$.
This gives
\[
    Q_j\cdot(\partial_\alpha z_j)^\perp
    =
    T_j+\frac12\sum_{\ell=1}^{2N}\operatorname{sgn}(j-\ell)\sigma_\ell.
\]
Substitution of the densities and finite summation give
\begin{equation}\label{e:layerRidentity}
    R
    =
    -\frac{2}{1-A^2}
    \partial_tz\cdot(\partial_\alpha z)^\perp
    -\sum_j\omega_jT_j+tE,
\end{equation}
where $E$ is independent of time.
To see the cancellation, the coefficient of
$\partial_tz_j\cdot(\partial_\alpha z_j)^\perp$ in $-R$
is $\delta_j/(1-A^2)$, while the coefficient of
$i\cdot(\partial_\alpha z_j)^\perp$ is
\[
    \frac{\omega_j}{1-A^2}
    \bigl(
        \theta_{j-1}+\theta_j
        +A(1+\theta_{j-1}\theta_j)
    \bigr).
\]
These latter coefficients sum to zero.
Expanding $z_j=z+tc h_j\tau^\perp$
and using $\sum_j\delta_j h_j=0$ leaves only
time-independent multiples of $t$ besides the displayed terms.

At zero the total density is $\sigma^{(0)}$ and
\[
    T_j(0)
    =
    Q_0[\sigma^{(0)}]\cdot(\partial_\alpha z^\circ)^\perp.
\]
Since $\partial_tz(0)=z^{(1)}=-Q_0[\sigma^{(0)}]$,
\eqref{e:layerRidentity} and
$\sum_j\omega_j=2/(1-A^2)$ give $R(0)=0$.
Compute $r=\partial_tR(0)$ first with $z^{(2)}=0$.
The product estimates and the initial trace differentiation
formula give $r\in H^{k-2}$, as in
Proposition~\ref{p:second}.
Introducing the acceleration changes the derivative of the
total density by $-2A\,z^{(2)}\cdot(\partial_\alpha z^\circ)^\perp$
and that of each corrected trace $T_j$
by $-AP[z^{(2)}\cdot(\partial_\alpha z^\circ)^\perp]$.
Hence

\begin{equation}\label{e:layeracceleration}
    \partial_tR(0)
    =
    r-\frac{2}{1-A^2}(I-AP)\bigl[z^{(2)}\cdot(\partial_\alpha z^\circ)^\perp\bigr].
\end{equation}
Prescribe $z^{(2)}$ as a curve normal to $z^\circ$ such that
\[
    z^{(2)}\cdot(\partial_\alpha z^\circ)^\perp
    =
    \frac{1-A^2}{2}(I-AP)^{-1}r,
    \qquad
    z^{(2)}
    =
    \frac{z^{(2)}\cdot(\partial_\alpha z^\circ)^\perp}
    {|\partial_\alpha z^\circ|^2}
    (\partial_\alpha z^\circ)^\perp.
\]
Proposition~\ref{p:inverse} gives its regularity and
$\partial_tR(0)=0$.
Differentiating \eqref{e:layerRidentity} and subtracting
its initial value, Lemma~\ref{l:layerpressure} handles the
pressure terms and the $L^1$ calculation in
Proposition~\ref{p:transverse} handles the geometric term.
Integrating in time gives
\begin{equation}\label{e:layerRprimitive}
    \sup_\alpha
    \left|\int_0^\alpha R(t,\beta)\,d\beta\right|
    \le Ct^2.
\end{equation}
Here the hypotheses of Lemma~\ref{l:layerpressure}
follow from
\[
    \sigma_j+\delta_j(1+Ac_\infty h_j)
    \in C_t^\infty H^{k-3};
\]
the constant parts are independent of time.

For
\[
    S_j(t,\alpha)
    =
    \int_0^\alpha
    \sum_{\ell\le j}R_\ell(t,\beta)\,d\beta,
\]
define
\begin{equation}\label{e:layerpotentials}
    \phi_j(t,Z(t,\alpha,\lambda))
    =
    S_j(t,\alpha)
    -\frac{\lambda-h_1}{h_{2N}-h_1}
    S_{2N}(t,\alpha),
    \qquad h_j<\lambda<h_{j+1}.
\end{equation}
These potentials satisfy \eqref{e:layerresidual},
including both exterior boundaries.
The only transverse derivative is a constant multiple of
$S_{2N}=O(t^2)$ by \eqref{e:layerRprimitive}.
Since $F_A(\theta_0)=F_A(-1)=0$, the speed definition
\eqref{e:layercurves} telescopes to
\[
    \sum_{\ell\le j}\delta_\ell h_\ell
    =
    -\frac{N}{2N-1}F_A(\theta_j).
\]
Together with \eqref{eq:opening-speed} and the initial
pressure traces, this gives
\[
    \sum_{\ell\le j}R_\ell(0)
    =
    -\frac{
        N\bar c\,\sigma^{(0)}F_A(\theta_j)
    }{(2N-1)(1+A\theta_j)},
\]
and hence
\[
    \nabla^\perp\phi_j
    =
    -\frac{
        N\bar c\,\sigma^{(0)}F_A(\theta_j)
    }{
        (2N-1)|\partial_\alpha z^\circ|(1+A\theta_j)
    }\tau^\perp+O(t).
\]
In the $j$th strip the same Plemelj calculation yields
\[
    \nabla p
    =
    Q_0[\sigma^{(0)}]
    +\frac{\sigma^{(0)}}{2|\partial_\alpha z^\circ|}
    \left(
        \theta_j-\frac{2AN\bar c}{2N-1}F_A(\theta_j)
    \right)\tau^\perp+O(t).
\]
Substituting these expansions into \eqref{e:layerauxiliary},
and using
\[
    (i-AQ_0[\sigma^{(0)}])\cdot\tau^\perp
    =
    -\frac{\sigma^{(0)}}{2|\partial_\alpha z^\circ|},
\]
proves \eqref{e:layerasymptotics} for these fields.

Finally, on each simply connected intermediate strip the field
\[
 \frac{1+A\theta_j}{F_A(\theta_j)}\nabla^\perp\phi_j
 -\frac12(A\bar u+i)
\]
is divergence-free, since $\nabla\cdot\bar u=0$.
It is therefore $\nabla^\perp g_j$ for a scalar potential:
\begin{equation}\label{e:layergconversion}
 \nabla^\perp g_j=
 \frac{1+A\theta_j}{F_A(\theta_j)}\nabla^\perp\phi_j
 -\frac12(A\bar u+i).
\end{equation}
This recovers \eqref{e:layerfields} with the same fields and transmission
conditions.
Indeed, equation \eqref{e:layergconversion} gives
$F_A(\theta_j)\bigl(\nabla^\perp g_j+\frac12(A\bar u+i)\bigr)
=(1+A\theta_j)\nabla^\perp\phi_j$, so the momentum in
\eqref{e:layerfields} is the one in \eqref{e:layerauxiliary};
substituting it into \eqref{eq:velocity} with $\bar\theta=\theta_j$
and $C(\theta_j)=\theta_j-\frac A2F_A(\theta_j)$ gives
the velocity in \eqref{e:layerauxiliary}.
A normalization at a fixed point in strip coordinates
makes $g_j$ continuous in time.
\end{proof}

\appendix
\section{Relaxation of IPM} \label{AppendixA}
In this appendix, we prove Theorems~\ref{thm:relaxation}
and~\ref{thm:bounded-relaxation}. The identities involving the
$\Lambda$-lamination hulls were established in \cite{Mengual22}.
We identify the corresponding $\Lambda$-convex hulls and weak
relaxations by means of a family of $\Lambda$-affine functions
suggested by ChatGPT 6 and the div--curl lemma.
The converse approximation by exact solutions follows from
the $h$-principle in \cite[Theorem~2.1]{Mengual22}.

For $z=(\theta,u,m)$, define the quadratic function
\[
Q(z):=u\cdot(u+Am+\theta e_2).
\]
This function is $\Lambda$-affine. Indeed, if
$\bar z=(\bar\theta,\bar u,\bar m)\in\Lambda$, then
$\bar u$ is orthogonal to the corresponding nonzero
spatial frequency $\xi$, whereas
$\bar u+A\bar m+\bar\theta e_2$ is parallel to $\xi$.
Hence $Q(\bar z)=0$, and the quadratic nature of $Q$
implies that $s\mapsto Q(z+s\bar z)$ is affine
for every $z$.

We also define
\[
g(z):=(1-A\theta)|m-\theta u|^2
+(1-\theta^2)(m-\theta u)\cdot(Au+e_2).
\]
This function characterizes $\overline{\mathcal O}$ through
\begin{equation}\label{eq:g<=0}
z\in\overline{\mathcal O}
\quad\Longleftrightarrow\quad
|\theta|\le1 \quad\text{and}\quad g(z)\le0,
\end{equation}
see \cite[Proposition~2.3]{Szekelyhidi12} for $A=0$ and
\cite[Lemmas~4.1--4.2]{Mengual22} for $|A|<1$.

Following the considerations in \cite[Remark~4.5]{Mengual22},
the second author conjectured that $g(z)/(1-\theta^2)$ is a
$\Lambda$-convex function characterizing the relaxed constraint
for $|\theta|<1$. Motivated by recent advances in AI mathematical
capabilities, we posed this question to ChatGPT~6 out of
curiosity. It suggested expressing this function as the
supremum of a family of $\Lambda$-affine functions $H_w$, an approach we develop and verify below.
For each $w=(w_+,w_-)\in\mathbb R^2\times\mathbb R^2$,
define
\[
H_w(z):=
\sum_{a=\pm1}(1+aA)
\left(u+am-\frac{1+a\theta}{2}w_a\right)\cdot w_a
+m_2-Q(z).
\]
Since all remaining terms are affine, each $H_w$ is
also $\Lambda$-affine. We will show that these functions
are nonpositive on $K$ and characterize
$\overline{\mathcal O}$ through the inequalities
$H_w\le0$.

\begin{lemma}\label{lem:H-on-K}
For $\theta=\pm1$,
\begin{equation}\label{eq:lem:H-on-K}
H_w(\theta,u,\theta u)
=-(1+\theta A)|w_\theta-u|^2.
\end{equation}
Therefore, $H_w\le0$ on $K$.
\end{lemma}

\begin{proof}
Let $\theta=\pm1$ and $m=\theta u$. Then
\[
Q(\theta,u,\theta u)
=(1+\theta A)|u|^2+\theta u_2.
\]
Moreover,
\[
u+am-\frac{1+a\theta}{2}w_a
=(1+a\theta)\left(u-\frac12w_a\right).
\]
Thus, the summand in $H_w$ vanishes for $a=-\theta$,
while for $a=\theta$ it equals
$(1+\theta A)(2u-w_\theta)\cdot w_\theta$.
Since $m_2=\theta u_2$, we obtain
\eqref{eq:lem:H-on-K}.
The conclusion follows from $1+\theta A>0$.
\end{proof}

\begin{lemma}\label{lem:sup-H}
For $|\theta|<1$,
\[
\sup_{w\in\mathbb R^2\times\mathbb R^2}H_w(z)
=\frac{g(z)}{1-\theta^2}.
\]
\end{lemma}

\begin{proof}
Consider the auxiliary function
\[
p(t,x,y)=\left(x-\frac t2y\right)\cdot y
\]
defined for $t\in\mathbb R$ and
$x,y\in\mathbb R^2$.
For $t>0$, completing the square gives
\[
p(t,x,y)
=
\frac{|x|^2}{2t}
-\frac t2\left|y-\frac xt\right|^2.
\]
Hence, this inverted paraboloid is maximized for
$t>0$ at $y=x/t$, with
\[
\sup_{y\in\mathbb R^2}p(t,x,y)
=\frac{|x|^2}{2t}.
\]
For $t=0$, the supremum is zero if $x=0$ and infinite
otherwise; for $t<0$, it is infinite.

Applying this identity to each term in
\[
H_w(z)
=
\sum_{a=\pm1}(1+aA)
p(1+a\theta,u+am,w_a)
+m_2-Q(z),
\]
we deduce that
\begin{equation}\label{eq:supH:1}
\sup_{w\in\mathbb R^2\times\mathbb R^2}H_w(z)
=
\frac12\sum_{a=\pm1}
\frac{1+aA}{1+a\theta}|u+am|^2
+m_2-Q(z).
\end{equation}
On the one hand, by splitting
\[
m+au=(m-\theta u)+a(1+a\theta)u,
\]
and using the identities
\begin{align*}
\frac{1+A}{1+\theta}
+\frac{1-A}{1-\theta}
&=2\frac{1-\theta A}{1-\theta^2},
&
(1+A)(1+\theta)+(1-A)(1-\theta)
&=2(1+\theta A),
\end{align*}
we get
\[
\frac12\sum_{a=\pm1}
\frac{1+aA}{1+a\theta}|u+am|^2
=
\frac{1-\theta A}{1-\theta^2}|m-\theta u|^2
+2Au\cdot(m-\theta u)
+(1+\theta A)|u|^2.
\]
On the other hand, we decompose
\[
m_2-Q(z)
=
(m-\theta u)\cdot(e_2-Au)
-(1+\theta A)|u|^2.
\]
Combining these identities in \eqref{eq:supH:1},
we obtain
\[
\sup_{w\in\mathbb R^2\times\mathbb R^2}H_w(z)
=
\frac{1-\theta A}{1-\theta^2}|m-\theta u|^2
+(m-\theta u)\cdot(Au+e_2)
=
\frac{g(z)}{1-\theta^2}.
\]
\end{proof}

\begin{corollary}\label{cor:separators}
It holds that
\[
\overline{\mathcal O}
=
\{|\theta|\le1\}
\bigcap_{w\in\mathbb R^2\times\mathbb R^2}
\{H_w(z)\le0\}.
\]
\end{corollary}

\begin{proof}
For $|\theta|<1$, Lemma~\ref{lem:sup-H} and \eqref{eq:g<=0} gives
\[
H_w(z)\le0\ \text{for every }w
\quad\Longleftrightarrow\quad
g(z)\le0
\quad\Longleftrightarrow\quad
z\in\overline{\mathcal O}.
\]

For $\theta=\pm1$, the quadratic term in $w_{-\theta}$ vanishes,
so that
$$
H_w(z)
=(1-\theta A)(u-\theta m)\cdot w_{-\theta}
+(1+\theta A)(u+\theta m-w_\theta)\cdot w_\theta
+m_2-Q(z).
$$
Fixing $w_\theta=0$, we see that unless $u-\theta m=0$,
choosing $w_{-\theta}=t(u-\theta m)$ makes
$H_w(z)\to+\infty$ as $t\to+\infty$.
Thus, $H_w(z)\le0$ for every $w$ requires $m=\theta u$.
Conversely, for these states, Lemma~\ref{lem:H-on-K} gives
$H_w(z)\le0$.
Finally, when $\theta=\pm1$,
$g(z)=(1-\theta A)|m-\theta u|^2,$
so $g(z)\le0$ is equivalent to $m=\theta u$.
Hence these are precisely the states of $\overline{\mathcal O}$
with $|\theta|=1$, completing the proof.
\end{proof}

\subsection{Proof of Theorem~\ref{thm:relaxation}}\label{sec:thm:relaxation}

\subsubsection{The $\Lambda$-convex hull}

The identity $K^{\mathrm{lc},\Lambda}=\overline{\mathcal O}$
from \cite{Mengual22} gives
$\overline{\mathcal O}\subset K^\Lambda$.
Conversely, the affine functions $\theta-1$ and
$-\theta-1$, together with the $\Lambda$-affine
functions $H_w$, are nonpositive on $K$.
Therefore, every $z\in K^\Lambda$ satisfies
\[
|\theta|\le1,\qquad H_w(z)\le0
\quad\text{for every }w\in\mathbb R^2\times\mathbb R^2.
\]
Corollary~\ref{cor:separators} then yields
$z\in\overline{\mathcal O}$, proving the claim.

\subsubsection{The weak relaxation}

We first establish the weak stability of the relaxed constraint.

\begin{lemma}\label{lem:weak-stability}
Let $T>0$ and let $B\subset\mathbb R^2$ be a bounded open ball.
Suppose that $z_j\to z$ in $C([0,T];L^2_w(B;\mathbb R^5))$,
and that each $z_j$ solves the linear system \eqref{mIPM} and takes
values in $\overline{\mathcal O}$ almost everywhere in
$(0,T)\times B$. Then $z$ inherits the same properties.
\end{lemma}

\begin{proof}
Since the linear equations pass to the limit in distributions, it remains to check that $z$ takes
values in $\overline{\mathcal O}$ almost everywhere in
$(0,T)\times B$.
Since $\operatorname{div}_x u_j=0$ and
$\operatorname{curl}_x(u_j+Am_j+\theta_j e_2)=0$,
the div--curl lemma gives, for almost every $t$,
\[
Q(z_j(t,\cdot))\rightarrow Q(z(t,\cdot))
\quad\text{in }\mathcal D'(B).
\]
All remaining terms in $H_w$ are affine, so the inequalities
$H_w\le0$ pass to the limit, as does $|\theta|\le1$.
Corollary~\ref{cor:separators} completes the proof.
\end{proof}

Since $K\subset\overline{\mathcal O}$,
Lemma~\ref{lem:weak-stability} gives
$\mathcal R\subset\overline{\mathcal O}$.
Conversely, approximate $z\in\overline{\mathcal O}$ by constant
states $z_n\in\mathcal O$, which belong to a common bounded
perturbation set $\mathcal O_M$ by \cite[Lemma~4.6]{Mengual22},
with $M$ chosen as in \cite[Lemma~4.7]{Mengual22}.
The approximation property of the $h$-principle
\cite[Theorem~2.1]{Mengual22} then yields $K_M$-valued solutions
converging to $z$ in $C([0,T];L^2_w(B;\mathbb R^5))$.
Thus $z\in\mathcal R$, proving the reverse inclusion.

\subsection{Proof of Theorem~\ref{thm:bounded-relaxation}}
\label{sec:thm:bounded-relaxation}

As observed in
\cite{Szekelyhidi12,Mengual22}, the additional constraints \eqref{eq:OM-constraints} are
sublevel sets of a $\Lambda$-affine function in
\eqref{eq:OM-quadratic} and convex functions in
\eqref{eq:OM-minus}--\eqref{eq:OM-plus}. Their non-strict versions hold on $K_M$.
Together with Corollary~\ref{cor:separators} and the identity
$K_M^{\mathrm{lc},\Lambda}=\overline{\mathcal O_M}$
from \cite{Mengual22}, this gives
$K_M^\Lambda=\overline{\mathcal O_M}$ as in the unbounded case.

For the weak relaxation, Lemma~\ref{lem:weak-stability}
preserves the constraint $\overline{\mathcal O}$.
The non-strict version of \eqref{eq:OM-quadratic} passes
to the limit by div--curl, since its left-hand side is
$4Q(z)+4Au_2$; those of
\eqref{eq:OM-minus}--\eqref{eq:OM-plus} pass to the limit
by convexity. Hence $\mathcal R_M\subset\overline{\mathcal O_M}$.

The reverse inclusion follows from the same approximation
argument as above. Indeed, for $z_n\in\mathcal O_M$ converging
to $z\in\overline{\mathcal O_M}$, strictness of the constraints
allows us to choose $1<M_n<M$ with $z_n\in\mathcal O_{M_n}$,
avoiding the exceptional parameter in \cite[Lemma~4.7]{Mengual22}.
Applying the bounded $h$-principle gives approximating
$K_{M_n}$-valued solutions, which are uniformly bounded
and take values in $K_M$. Thus $z\in\mathcal R_M$.\\

\noindent\textbf{Acknowledgements.}
The authors thank \'Angel Castro, Daniel Faraco, Francisco Gancedo and Eduardo Garc\'ia-Ju\'arez for their valuable comments during the preparation of this work.

PLV acknowledges support from the Spanish Government through the FPI grant PRE2022-105212; from the AEI project PID2022-140494NA-I00 (Spain); and from IMAG, funded by MICINN through the María de Maeztu Excellence Grant CEX2020-001105-M/AEI/10.13039/501100011033.

FM acknowledges support from grant RYC2023-045748-I; from grants PID2024-158418NB-I00 and PID2024-158664NB-C2, funded by MCIN/AEI/10.13039/501100011033; and from the IMUS María de Maeztu Unit of Excellence grant CEX2024-001517-M (Apoyo a Unidades de Excelencia María de Maeztu), funded by MICIU/AEI/10.13039/501100011033.\\

\noindent\textbf{Declaration of AI assistance.}
This project was initiated, and the main mathematical arguments
were completed, without the use of AI tools. 
AI assistance was used specifically in the
relaxation of IPM result presented in Appendix~\ref{AppendixA}, which is independent
of the main theorem and is included for completeness. 
AI tools also assisted with the review,
presentation, and typesetting of the manuscript.
The authors take full responsibility for the mathematical content
and the final text.



\bibliographystyle{acm}
\bibliography{references}

\begin{flushleft}
	\quad\\
	Paula Luna-Velasco\\
	\textsc{Universidad de Granada
        \& Universidad de Sevilla
        \\
		18071 Granada, 41012 Sevilla, Spain}\\
	\emph{E-mail address: pluna@us.es} 
\end{flushleft}

\begin{flushleft}
	\quad\\
	Francisco Mengual\\
	\textsc{Universidad de Sevilla, IMUS\\
		41012 Sevilla, Spain}\\
	\emph{E-mail address:} fmengual@us.es
\end{flushleft}

\end{document}